\documentclass[11pt,final]{amsart}
\usepackage[utf8]{inputenc}
\usepackage[T1]{fontenc}
\usepackage[letterpaper,centering]{geometry}
\usepackage{lmodern}
\usepackage{eucal}
\usepackage{mathrsfs}
\usepackage{amsmath,amssymb,amsfonts}
\usepackage[unicode,pdfborder={0 0 0}]{hyperref}
\usepackage[ps,dvips,arrow,matrix,tips,line]{xy}
{\setbox0\hbox{$ $}}\fontdimen16\textfont2=\fontdimen17\textfont2
\entrymodifiers={+!!<0pt,\the\fontdimen22\textfont2>}
\SelectTips{cm}{11}

\usepackage{ifdraft}
\ifdraft{
\usepackage[notcite,notref]{showkeys}
\addtolength{\hoffset}{1.5cm}

}{}

\theoremstyle{plain}
\newtheorem{thm}{Theorem}[section]
\newtheorem{conj}[thm]{Conjecture}
\newtheorem*{hypHH}{Hypothesis~$(\mathrm{HH}_1)$}
\newtheorem{lem}[thm]{Lemma}
\newtheorem{cor}[thm]{Corollary}
\newtheorem{prop}[thm]{Proposition}
\theoremstyle{definition}
\newtheorem{cond}[thm]{Condition}
\newtheorem{defn}[thm]{Definition}
\newtheorem{rmk}[thm]{Remark}
\newtheorem{rmks}[thm]{Remarks}

\newtheorem{examples}[thm]{Examples}
\numberwithin{equation}{section}

\input cyracc.def
\DeclareFontFamily{U}{russian}{}
\DeclareFontShape{U}{russian}{m}{n}
        { <5><6> wncyr5
        <7><8><9> wncyr7
        <10><10.95><12><14.4><17.28><20.74><24.88> wncyr10 }{}
\DeclareSymbolFont{Russian}{U}{russian}{m}{n}
\DeclareSymbolFontAlphabet{\mathcyr}{Russian}
\makeatletter
\let\@math@cyr\mathcyr
\renewcommand{\mathcyr}[1]{\@math@cyr{\cyracc #1}}
\makeatother

\newcommand{\mylangle}{\xy(0,0),(1.17,1.485)**[|(1.025)]@{-},(0,0),(1.17,-1.485)**[|(1.025)]@{-}\endxy\mkern2mu}
\newcommand{\myrangle}{\mkern2mu\xy(0,0),(-1.17,1.485)**[|(1.025)]@{-},(0,0),(-1.17,-1.485)**[|(1.025)]@{-}\endxy}
\newcommand{\Frob}{\mathrm{Fr}}

\newcommand{\barsigma}{{\mkern2.5mu\overline{\mkern-2.5mu{}\sigma\mkern-1.7mu}\mkern1.7mu}}
\newcommand{\kbar}{{\mkern1mu\overline{\mkern-1mu{}k\mkern-1mu}\mkern1mu}}

\makeatletter
\def\myrightarrow{{\setbox\z@\hbox{$\rightarrow$}\dimen0\ht\z@\multiply\dimen0 6\divide\dimen0 10\ht\z@\dimen0\box\z@}}
\def\myrightarrowfill@{\arrowfill@\relbar\relbar\myrightarrow}
\newcommand{\myxrightarrow}[2][]{\ext@arrow 0359\myrightarrowfill@{#1}{#2}}
\makeatother
\newcommand{\mtilde}{{\mathchoice
    {\widetilde{m}}
    {\widetilde{m}}
    {\rlap{$\scriptscriptstyle{m}$}\vphantom{\raise0pt\hbox{$m$}}\smash{\lower2.5pt\hbox{$\scriptscriptstyle\widetilde{\phantom{\scriptscriptstyle{m}}}$}}}
    {\rlap{$\scriptscriptstyle{m}$}\vphantom{\raise.2pt\hbox{$m$}}\smash{\lower2.05pt\hbox{$\scriptscriptstyle\widetilde{\phantom{\scriptscriptstyle{m}}}$}}}}}
\newcommand{\ctilde}{{\widetilde{c}\mkern1.1mu}}

\newcommand{\Mtilde}{{\mathchoice
    {\rlap{$M$}\mkern1mu\smash[b]{\lower.5pt\hbox{$\widetilde{\phantom{M}}$}}\mkern-1mu}
    {\rlap{$M$}\mkern1mu\smash[b]{\lower.5pt\hbox{$\widetilde{\phantom{M}}$}}\mkern-1mu}
    {\rlap{$\scriptstyle{M}$}\mkern1mu\smash[b]{\lower.5pt\hbox{$\widetilde{\phantom{\scriptstyle{M}}}$}}\mkern-1mu}
    {\widetilde{M}}}}
\newcommand{\etabar}{{\bar\eta}}
\newcommand{\xibar}{{\bar\xi}}

\newcommand{\sC}{{\mathscr C}}
\newcommand{\sD}{{\mathscr D}}
\newcommand{\sE}{{\mathscr E}}
\newcommand{\sF}{{\mathscr F}}
\newcommand{\sO}{{\mathscr O}}
\newcommand{\sOint}{{\mathcal O}}
\newcommand{\sB}{{\mathscr B}}
\newcommand{\sL}{{\mathscr L}}
\newcommand{\sT}{{\mathscr T}}
\newcommand{\sU}{{\mathscr U}}
\newcommand{\sV}{{\mathscr V}}
\newcommand{\sW}{{\mathscr W}}
\newcommand{\sX}{{\mathscr X}}
\newcommand{\sY}{{\mathscr Y}}
\newcommand{\A}{{\mathbf A}}

\renewcommand{\C}{{\mathbf C}}

\newcommand{\cE}{(\mathrm{E})}
\newcommand{\cEu}{(\mathrm{E}_1)}
\newcommand{\F}{{\mathbf F}}
\newcommand{\Fp}{{\mathbf F}_{\mkern-2mup}}
\renewcommand{\Im}{{\mathrm{Im}}}

\newcommand{\effred}{{\mathrm{eff},\mathrm{red}}}
\newcommand{\eff}{{\mathrm{eff}}}
\newcommand{\Supp}{{\mathrm{Supp}}}
\newcommand{\Sym}{{\mathrm{Sym}}}

\newcommand{\N}{{\mathbf N}}
\renewcommand{\P}{{\mathbf P}}
\newcommand{\Q}{{\mathbf Q}}
\newcommand{\R}{{\mathbf R}}

\newcommand{\Z}{{\mathbf Z}}
\newcommand{\Zcyc}{{\mathrm Z}}

\newcommand{\CH}{\mathrm{CH}}
\newcommand{\Azero}{\mathrm{A}_0}

\renewcommand{\vert}{\mathrm{vert}}
\newcommand{\NS}{\mathrm{NS}}
\newcommand{\Gm}{\mathbf{G}_\mathrm{m}}
\newcommand{\Gal}{\mathrm{Gal}}
\newcommand{\Pic}{\mathrm{Pic}}
\newcommand{\Br}{\mathrm{Br}}
\newcommand{\Spec}{\mathrm{Spec}}
\newcommand{\Div}{\mathrm{Div}}
\renewcommand{\phi}{\varphi}
\renewcommand{\emptyset}{\varnothing}
\newcommand{\red}{{\mathrm{red}}}
\newcommand{\Ker}{{\mathrm{Ker}}}
\newcommand{\Coker}{{\mathrm{Coker}}}
\newcommand{\Hom}{{\mathrm{Hom}}}
\newcommand{\RHom}{R\mathscr{H}\mkern-4muom}

\newcommand{\xihat}{{\widehat{\xi}}}
\newcommand{\zhat}{{\widehat{z}}}

\newcommand{\Zhat}{{\widehat{\Z}}}
\newcommand{\chapeau}{{\rlap{\smash{\hbox{\lower4pt\hbox{\hskip1pt$\widehat{\phantom{u}}$}}}}}}
\newcommand{\Picplushat}{\Pic_+^{{\smash{\hbox{\lower4pt\hbox{\hskip0.4pt$\widehat{\phantom{u}}$}}}}}}
\newcommand{\PicplusAhat}{\Pic_{+,\A}^{{\smash{\hbox{\lower4pt\hbox{\hskip.4pt$\widehat{\phantom{u}}$}}}}}}
\newcommand{\Picplus}{\Pic_+}
\newcommand{\PicplusA}{\Pic_{+,\A}}
\newcommand{\Brplus}{\Br_+}
\newcommand{\CHzhat}{\CH_{0\phantom{,}}^\chapeau}
\newcommand{\CHzAhat}{\CH_{0,\A}^\chapeau}
\newcommand{\CHzA}{\CH_{0,\A}}
\newcommand{\Pichat}{\Pic^{{\smash{\hbox{\lower4pt\hbox{\hskip0.4pt$\widehat{\phantom{u}}$}}}}}}
\newcommand{\PicA}{\Pic_\A}
\newcommand{\PicAhat}{\Pichat_\A}
\renewcommand{\div}{{\mathrm{div}}}

\newcommand{\Cores}{\mathrm{Cores}}
\DeclareMathOperator{\inv}{inv}
\newcommand{\Fv}[1]{\F_{\mkern-2mu#1}}

\advance\textheight 1.25mm
\advance\topmargin -.5mm

\makeatletter\let\@wraptoccontribs\wraptoccontribs\makeatother

\date{August 29th, 2014; revised on April 2nd, 2015}
\title[On the fibration method]{On the fibration method for zero-cycles and\\rational points}

\author{Yonatan Harpaz}
\address{Institute for Mathematics, Astrophysics and Particle Physics, Faculty of Science,
Radboud University,
P.O. Box 9010, 6500~GL Nijmegen, The Netherlands}
\email{y.harpaz@math.ru.nl}

\author{Olivier Wittenberg}
\address{D\'epartement de math\'ematiques et applications, \'Ecole normale sup\'erieure, 45~rue d'Ulm, 75320 Paris Cedex 05, France}
\email{wittenberg@dma.ens.fr}

\begin{document}
\begin{abstract}
Conjectures on the existence of zero-cycles on arbitrary smooth projective
varieties over number fields were proposed by Colliot-Thélène, Sansuc,
Kato and
Saito in the 1980's.  We prove that these conjectures are compatible with
fibrations, for fibrations into rationally connected varieties over a
curve.  In particular, they hold for the total space of families of
homogeneous spaces of linear groups with connected geometric stabilisers.
We prove the analogous result for rational points, conditionally on a
conjecture on locally split values of polynomials which a recent work of
Matthiesen establishes in the case of linear polynomials over the rationals.
\end{abstract}

\maketitle

\vspace*{-1mm}
\section{Introduction}
\label{sec:intro}

Let~$X$ denote a smooth and proper algebraic variety over
a number field~$k$.

It is customary to embed the set~$X(k)$ of rational points of~$X$
diagonally into the space of adelic points~$X(\A_k)=\prod_{v \in \Omega}X(k_v)$.
In his foundational paper~\cite{maninicm}, Manin combined local class field
theory, global class field theory and Grothendieck's theory of Brauer
groups of schemes to define a natural closed subset~$X(\A_k)^{\Br(X)}$
of~$X(\A_k)$ which contains the image of~$X(k)$.  The absence of any
rational point on~$X$ is in many examples explained by the vacuity of
$X(\A_k)^{\Br(X)}$.  The following conjecture, first
formulated by Colliot-Th\'el\`ene and Sansuc in~1979 in the case of surfaces,
posits a partial converse.

\newcommand{\citeconjct}{\cite[p.~174]{ctbudapest}}
\begin{conj}[\citeconjct]
\label{conj:points}
For any smooth, proper, geometrically irreducible, rationally connected variety~$X$ over a number field~$k$,
the set~$X(k)$ is dense in $X(\A_k)^{\Br(X)}$.
\end{conj}

By ``rationally connected'', we mean that for any algebraically closed field~$K$ containing~$k$,
two general $K$\nobreakdash-points of~$X$ can be joined by a rational curve
(\emph{e.g.}, geometrically unirational varieties are rationally connected;
see~\cite[Chapter~IV]{kollarbook} for more on this notion).

In a parallel development, a conjecture concerning the image of the Chow group~$\CH_0(X)$ of zero-cycles up to rational equivalence
in the product of the groups~$\CH_0(X\otimes_k k_v)$
over the places~$v$ of~$k$ was put forward by Colliot-Th\'el\`ene, Sansuc, Kato and Saito (for rational surfaces in~\cite[\textsection4]{ctsansuc},
for arbitrary varieties in~\cite[\textsection7]{katosaitocontemp} and in~\cite[\textsection1]{ctbordeaux}).
Let us recall its statement.
Denote by $\Omega_f$ (resp.~$\Omega_\infty$) the set of finite (resp.~infinite) places of~$k$
and by~$\CHzA(X)$ the group
\begin{align}
\label{eq:defchza}
\CHzA(X)=\prod_{v \in \Omega_f} \CH_0(X\otimes_k k_v) \times \prod_{v \in \Omega_\infty} \frac{\CH_0(X \otimes_k k_v)}{N_{\kbar_v/k_v}(\CH_0(X\otimes_k\kbar_v))}\rlap{\text{.}}
\end{align}
By the reciprocity law of global class field theory,
the pairings
\begin{align*}
\rlap{ }\mylangle-,-\myrangle_v: \Br(X\otimes_k k_v) \times \CH_0(X \otimes_k k_v) \to \Br(k_v) \hookrightarrow \Q/\Z
\end{align*}
for $v \in \Omega$,
characterised by the property
that $\mylangle \alpha,P \myrangle_v$ is the local invariant of $\alpha(P) \in \Br(k_v(P))$ whenever~$P$ is a closed point of~$X \otimes_k k_v$,
fit together in a complex
\begin{align}
\label{eq:conjE}
\xymatrix@C=4.5em{
\widehat{{\mathrm{CH}}_0(X)} \ar[r] &
\widehat{{\mathrm{CH}}_{0,\mathbf{A}}(X)}
\ar[r]^(.5){\mbox{$\sum_{\scriptscriptstyle{}v \in \Omega}\scriptscriptstyle\mylangle-,-\myrangle_v$}} &
*!<-2em,0ex>\entrybox{\mathrm{Hom}(\mathrm{Br}(X),{\mathbf Q}/{\mathbf Z})\rlap{\text{,}}}
}
\end{align}
where we denote $\widehat{M}=\varprojlim_{n\geq 1} M/nM$ for any abelian group~$M$.
In the formulation given by
van~Hamel~\cite[Theorem~0.2]{vanhamel},
the above-mentioned conjecture is then the following.

\newcommand{\citeconjctsks}{\cite[\textsection7]{katosaitocontemp}, \cite[\textsection1]{ctbordeaux}; see also~\cite[\textsection1.1]{wittdmj}}
\begin{conj}[\citeconjctsks]
\label{conj:cycles}
For any smooth, proper, geometrically irreducible variety~$X$ over a number field~$k$,
the complex~\eqref{eq:conjE} is exact.
\end{conj}

Rationally connected varieties which satisfy Conjecture~\ref{conj:points} over every finite extension
of~$k$
are known to satisfy Conjecture~\ref{conj:cycles} (cf.~\cite{liangarithmetic}).

Two general methods were elaborated to attack Conjecture~\ref{conj:points} and Conjecture~\ref{conj:cycles}.
The
descent method (see~\cite{ctsandescent2}, \cite{skorobeyond}, \cite{skobook}; see~\cite{salbergerdescent}
for an adaptation to the context of zero-cycles) generalises the
classical descent theory of elliptic curves to arbitrary varieties.
The fibration method,
with which we are concerned in the present paper, consists, under various sets of hypotheses,
in exploiting the structure of a fibration $f:X\to Y$ to establish
the conjectures for~$X$ when they are known for~$Y$ and for the fibers of~$f$.

Our aim in this paper is twofold.  First, in \textsection\textsection\ref{sec:reminders}--\ref{sec:maintheorems}, we prove
the following general fibration theorem for the existence of zero-cycles, which is the
main result of the paper.

\begin{thm}
\label{th:introcycles}
Let~$X$ be a smooth, proper, geometrically irreducible variety over a number field~$k$ and let $f:X\to\P^1_k$ be a dominant morphism with rationally connected geometric generic fiber.
If the smooth fibers of~$f$ satisfy Conjecture~\ref{conj:cycles}, then so does~$X$.
\end{thm}

The precise results we prove on zero-cycles are stated in~\textsection\ref{sec:maintheorems}.  They generalise Theorem~\ref{th:introcycles} in the following respects:
\begin{enumerate}
\item the base of the fibration is allowed to be any variety birationally equivalent to the product of a projective space with a curve~$C$ which
satisfies Conjecture~\ref{conj:cycles};
\smallskip
\item
only the smooth fibers of~$f$ above the closed points of a Hilbert subset (for instance, of a dense open subset)
of the base
are assumed to satisfy Conjecture~\ref{conj:cycles};
\smallskip
\item instead of requiring that the geometric generic fiber be rationally connected, it is enough to assume that
it is connected, that its abelian \'etale fundamental group is trivial and
that its Chow group of zero-cycles of degree~$0$ is trivial over any algebraically closed field extension;
\smallskip
\item when the geometric generic fiber is rationally connected, the assumption that the smooth fibers (above the closed points of a Hilbert subset) satisfy Conjecture~\ref{conj:cycles} may be replaced with the assumption
that they satisfy Conjecture~\ref{conj:points}.
\end{enumerate}

We note that a smooth, proper, geometrically irreducible curve~$C$ satisfies
Conjecture~\ref{conj:cycles} as soon as the divisible subgroup of the
Tate--Shafarevich group of its Jacobian is trivial, by a theorem of
Saito~\cite{saito} (see~\cite[Remarques~1.1~(iv)]{wittdmj}).  At least when $C(k)\neq\emptyset$,
this condition is even equivalent to Conjecture~\ref{conj:cycles}
(see~\cite[Ch.~I, Theorem~6.26~(b)]{milneadt}).

To illustrate Theorem~\ref{th:introcycles}, let us consider the very specific example
of the affine $n$\nobreakdash-dimensional hypersurface~$V$
given by
\begin{align}
\label{eq:normhyp}
N_{K/k}(x_1\omega_1 + \dots + x_n\omega_n)=P(t)\rlap{\text{,}}
\end{align}
where $P(t) \in k[t]$ is a nonzero polynomial and where the left-hand side is the norm
form associated to a finite extension~$K/k$ and to a basis
$(\omega_1,\dots,\omega_n)$
of the $k$\nobreakdash-vector space~$K$.
The study of the arithmetic of a smooth and proper model~$X$ of~$V$ has a long history.
Conjecture~\ref{conj:cycles} was previously known for~$X$ only when~$K/k$ is cyclic (see~\cite{ctsksd98},
which builds on~\cite{salberger}),
when~$K/k$ has prime degree or degree~$pq$ for distinct primes~$p,q$ (see~\cite[Theorem~4.3]{weioneq} and \cite[Example~3.1, Remark~3.7]{lianglocalglobal}),
when~$K/k$ is quartic or is abelian with Galois group $\Z/n\Z \times \Z/n\Z$
under certain assumptions on~$P(t)$ (see~\cite{weioneq}, \cite{derenthalsmeetswei}, \cite[Corollary~2.3]{liangtowards}
and~\cite[Example~3.9]{lianglocalglobal}),
and, finally, for arbitrary~$K/k$, when $P(t)=ct^m(1-t)^n$ for some $c \in k^*$ and some integers~$m,n$ (see~\cite{swarbrickjones}, which builds
on~\cite{heathbrownskorobogatov} and~\cite{cthasko} and
applies descent theory and the Hardy--Littlewood circle method
to establish Conjecture~\ref{conj:points} for~$X$ over every finite
extension of~$k$, and see~\cite[Example~3.10]{lianglocalglobal}).
By contrast,
it follows uniformly from Theorem~\ref{th:introcycles} that~$X$ satisfies
Conjecture~\ref{conj:cycles} for any~$K/k$ and any~$P(t)$.
Indeed, on the one hand,
the generic fiber of the projection map~$V\to\A^1_k$ given by the~$t$ coordinate is a torsor under the norm torus associated to the finite extension~$K/k$,
and on the other hand, smooth compactifications of torsors under tori satisfy Conjecture~\ref{conj:points} (see~\cite[Theorem~6.3.1]{skobook} and use~(4) above;
alternatively, these varieties satisfy Conjecture~\ref{conj:cycles} by~\cite{liangarithmetic}).
One can even go further and replace the~``constant'' field extension~$K/k$ with an arbitrary finite extension~$K/k(t)$:
Theorem~\ref{th:introcycles} implies Conjecture~\ref{conj:cycles} for smooth and proper models
of the affine $n$\nobreakdash-dimensional hypersurface defined by
$N_{K/k(t)}(x_1\omega_1 + \dots + x_n\omega_n)=P(t)$,
if~$(\omega_1,\dots,\omega_n)$
now denotes a basis
of the $k(t)$\nobreakdash-vector space~$K$.

More generally, we obtain the following result as a corollary of Theorem~\ref{th:introcycles}.
This covers in particular all fibrations into toric varieties, into Ch\^atelet surfaces or into
Ch\^atelet $p$\nobreakdash-folds
in the sense of~\cite{varillyviraychatelet}.
See~\textsection\ref{subsec:statements} for a more detailed discussion.

\begin{thm}
Let~$X$ be a smooth, proper, irreducible variety over a number field~$k$.
Let~$Y$ be an irreducible variety over~$k$,
birationally equivalent to either~$\P^n_k$, or~$C$, or $\P^n_k\times C$, for some $n\geq 1$ and some
smooth, proper, geometrically irreducible curve~$C$ over~$k$
which satisfies Conjecture~\ref{conj:cycles}.
Let $f:X \to Y$ be a dominant morphism whose generic fiber is birationally
equivalent to a homogeneous space of a connected linear algebraic group, with
connected geometric stabilisers.
Then~$X$ satisfies Conjecture~\ref{conj:cycles}.
\end{thm}

We then proceed to rational points in~\textsection\ref{sec:rationalpoints}.  We propose
in~\textsection\ref{subsec:conj}
 a conjecture on locally split values of polynomials in one variable
(``locally split'' means that the values are required to be products of primes which split in a given
finite extension~$L/k$) and prove, in~\textsection\ref{subsec:consequencesrationalpoints}, that this conjecture implies
a general fibration theorem for the existence of rational points.

\begin{thm}
\label{th:intropoints}
Let~$X$ be a smooth, proper, geometrically irreducible variety over a number field~$k$ and let
 $f:X\to\P^1_k$ be a dominant morphism with rationally connected geometric generic fiber.
Suppose Conjecture~\ref{conj:mainstrong} holds.
If the smooth fibers of~$f$ above the rational points
of~$\P^1_k$ satisfy Conjecture~\ref{conj:points}, then so does~$X$.
\end{thm}

Theorem~\ref{th:intropoints} admits various generalisations similar to those of Theorem~\ref{th:introcycles} outlined in~(1)--(4) above.
See Corollary~\ref{cor:ratpointsfull} and Corollary~\ref{cor:ratpointsRC} for precise statements.

Various cases of Conjecture~\ref{conj:mainstrong} can be established by algebro-geometric methods, by sieve methods or by additive combinatorics.
We discuss these cases in~\textsection\ref{subsec:knowncases}
as well as the relation between Conjecture~\ref{conj:mainstrong} and Schinzel's hypothesis~$(\mathrm{H})$.
The most significant result towards Conjecture~\ref{conj:mainstrong} is Matthiesen's recent work~\cite{matthiesen},
which renders Theorem~\ref{th:intropoints}
unconditional when~$k=\Q$ and the singular fibers of~$f$ lie above rational points of~$\P^1_\Q$. More precisely,
we obtain the following theorem
in~\textsection\ref{subsec:someunconditional}.

\begin{thm}
\label{th:matthintro}
Let~$X$ be a smooth, proper, geometrically irreducible variety over~$\Q$.
Let $f:X \to \P^1_\Q$
be a dominant morphism
with rationally connected geometric generic fiber,
whose non-split fibers all lie over rational points of~$\smash[b]{\P^1_\Q}$.
If~$X_c(\Q)$ is dense in ${X_c(\A_\Q)^{\Br(X_c)}}$
for every rational point~$c$ of a Hilbert subset of~$\P^1_\Q$,
then~$X(\Q)$ is dense in~$\smash[t]{X(\A_\Q)^{\Br(X)}}$.
\end{thm}

This unconditional case of Theorem~\ref{th:intropoints} was previously known only under the quite strong assumption
that every singular fiber of~$f$ contains an irreducible component of multiplicity~$1$ split by an abelian extension of~$\Q$ and
that the smooth fibers of~$f$ above the rational points of~$\P^1_\Q$ satisfy weak approximation (see~\cite{hsw}),
or otherwise under the assumption that~$f$ has a unique non-split fiber (see~\cite{hararifleches}).
We note that the work~\cite{matthiesen} relies on the contents of a recent paper
of Browning and Matthiesen~\cite{browningmatthiesen} which applied
the descent method and additive combinatorics to prove Conjecture~\ref{conj:points} for a smooth and proper model of the hypersurface~\eqref{eq:normhyp}
when~$k=\Q$ and~$P(t)$ is a product of linear polynomials,
the number field~$K$ being arbitrary
(a case now covered by Theorem~\ref{th:matthintro}).

\bigskip
The oldest instance of the use of the structure of a fibration to establish a particular case of Conjecture~\ref{conj:points} or of Conjecture~\ref{conj:cycles}
can be traced back to Hasse's proof of the Hasse--Minkowski theorem
for quadratic forms in four variables with rational coefficients
starting from the case of quadratic forms in three variables with rational coefficients.
It~was based on
Dirichlet's theorem on primes in arithmetic progressions and on the global reciprocity law.
In~1979, Colliot-Th\'el\`ene and Sansuc~\cite{ctsansucschinzel} noticed that
a variant of Hasse's proof yields Conjecture~\ref{conj:points} for a large family of conic bundle surfaces over~$\P^1_\Q$
if one assumes
 Schinzel's hypothesis~$(\mathrm{H})$,
a conjectural generalisation of Dirichlet's theorem.
Delicate arguments relying on the same two ingredients as Hasse's proof later allowed Salberger~\cite{salberger}
to settle Conjecture~\ref{conj:cycles} unconditionally for those conic bundle surfaces over~$\P^1_k$ whose Brauer group is reduced to constant classes.
At the same time,
building on the ideas of~\cite{cssI},
an abstract formalism for the fibration method was established in~\cite{skofibration}
in the case of fibrations, over~$\A^n_k$,
whose codimension~$1$ fibers are split
and whose smooth fibers satisfy weak approximation.
Further work of Serre~\cite{serrecollege} and of Swinnerton-Dyer~\cite{sdpencils}
on conic and two-dimensional quadric bundles over~$\P^1_k$
led to the systematic study of fibrations, over the projective line and later over a curve of arbitrary genus or over a projective space,
into varieties which satisfy weak approximation (see~\cite{ctsd94}, \cite{ctsksd98}, \cite{ctreglees}, \cite{frossard}, \cite{vanhamel}, \cite{wittdmj}, \cite{liangastuce},
\cite{hsw}).

To deal with non-split fibers, all of these papers rely on the same reciprocity argument as Hasse's original proof, at the core of which lies the following well-known fact from class field theory:
if~$L/k$ is a finite \emph{abelian} extension of number fields, an element of~$k$ which is a local norm from~$L$ at all places of~$k$ except at most one must be a local norm at
the remaining place too (see~\cite[Ch.~VII, \textsection3, Cor.~1, p.~51]{artintate}).
The failure of this statement for non-abelian extensions~$L/k$ turned out to be a critical hindrance to the development of the fibration method.
On the one hand, it caused all of the papers mentioned in the previous paragraph to assume
that every singular fiber of~$f$ over~$\A^1_k$
contains an irreducible component of multiplicity~$1$ split by an abelian extension of the field over which the fiber is defined.
(We note, however, that a trick allowing to deal with a few very specific non-abelian extensions whose Galois closure contains a nontrivial cyclic subextension, \emph{e.g.}, cubic extensions, appeared in \cite[Theorem~3.5]{weioneq} and \cite[Theorem~4.6]{hsw}.)
On the other hand, it led to the restriction that the
smooth fibers should satisfy weak approximation.
Harari~\cite{harariduke} \cite{hararifleches}
 was able to relax this condition
only in the following two situations: fibrations over~$\P^1_k$ all of whose fibers over~$\A^1_k$ are
geometrically irreducible
and fibrations, over more general bases, which admit a rational section.
Further results in this direction were obtained in~\cite{smeets} and in~\cite{liangtowards}
but were always limited by the assumption that some field
extension built out of
the splitting fields of the irreducible components of the fibers
and of the residues, along these components,
 of a finite collection of classes of the Brauer group of the generic
fiber, should be abelian.

The proofs of Theorem~\ref{th:introcycles} and of Theorem~\ref{th:intropoints} given below do not rely
on Dirichlet's theorem on primes in arithmetic progressions or on any variant of it.
As a consequence, they bypass the reciprocity argument alluded to in the previous paragraph,
whose only purpose was to gain some control over the splitting behaviour of the
unspecified prime output by Dirichlet's theorem.
They bypass the ensuing abelianness assumptions as well.
We replace Dirichlet's theorem
by a simple consequence of strong approximation off a place
for affine space minus a codimension~$2$ closed subset (see~Lemma~\ref{lem:strongapprox}),
or, in the context of rational points, by a certain conjectural variant of this lemma (see Conjecture~\ref{conj:mainstrong}, Proposition~\ref{prop:geomcrit} and Corollary~\ref{cor:strongapproxconj}).  Given a finite Galois extension~$L/k$,
these two statements directly give rise to elements of~$k$ which, outside a finite set of places of~$k$, have positive valuation only at places splitting in~$L$.
To be applicable, however, they require that certain hypotheses be satisfied; the rest of the proofs of Theorem~\ref{th:introcycles} and of Theorem~\ref{th:intropoints}
consists in relating these hypotheses to the Brauer--Manin condition, with the help of
suitable versions of the Poitou--Tate duality theorem
and of the so-called ``formal lemma'', a theorem originally due to Harari~\cite{harariduke}.

\bigskip
The paper is organised as follows.  In~\textsection\ref{sec:reminders}, we
give simple definitions for the groups~$\Pic_+(C)$ and~$\Br_+(C)$
first introduced in~\cite[\textsection5]{wittdmj}
and recall their properties and a consequence of the arithmetic duality theorem that they
satisfy
(Theorem~\ref{th:arithduality}).
The group~$\Pic_+(C)$ lies halfway between the usual Picard group~$\Pic(C)$
and Rosenlicht's relative Picard group~$\Pic(C,M)$ associated to a curve~$C$ endowed
with
a finite closed subset~$M$.
It~plays a central role in the proof
of Conjecture~\ref{conj:cycles} for fibrations over a curve~$C$ of positive genus.
When~$C=\P^1_k$, one could equally well work with the algebraic tori which are implicit in the definition of~$\Pic_+(C)$
(see~\eqref{seq:picp} below),
but the formalism of~$\Pic_+(C)$ and~$\Br_+(C)$ turns out to be very convenient also in this case.
In~\textsection\ref{sec:formallemma}, we state and prove a version of the formal lemma for zero-cycles on a variety~$X$ fibered over a curve~$C$ of positive genus
in which the zero-cycles under consideration are required to lie over a fixed linear equivalence class of divisors on~$C$.
In~\textsection\ref{sec:specialisation}, we give a short proof of
a theorem of Harari on the specialisation maps for the Brauer groups
of the fibers of a morphism $f:X\to\P^1_k$ whose
geometric generic fiber satisfies
$H^1(X_{\bar\eta},\Q/\Z)=0$
and $H^2(X_\etabar,\sO_{X_\etabar})=0$
and we adapt it to fibrations over a curve of positive genus.
Building on the contents of \textsection\textsection\ref{sec:reminders}--\ref{sec:formallemma}, we then prove in~\textsection\ref{sec:mainexistenceresult} the core
existence theorem for effective zero-cycles on fibrations over curves from which Theorem~\ref{th:introcycles} and all its refinements will be deduced (Theorem~\ref{th:existenceresult}).
In~\textsection\ref{sec:hilbertsubsets}, we show that Theorem~\ref{th:existenceresult}
remains valid if the arithmetic assumption on the smooth fibers is replaced by the same
assumption over the closed points of a Hilbert subset of the base.
To prove Conjecture~\ref{conj:cycles}, one needs to deal not only with effective
cycles as in Theorem~\ref{th:existenceresult}
but more generally with elements of completed Chow groups; the reduction
steps required to bridge this gap are carried out in~\textsection\ref{sec:reductions}.
Our main results on zero-cycles,
including Theorem~\ref{th:introcycles}, are stated, established and compared with
the literature
in~\textsection\ref{sec:maintheorems}.
Finally, we devote~\textsection\ref{sec:rationalpoints} to rational points.
Conjecture~\ref{conj:mainstrong} is stated in~\textsection\ref{subsec:conj}.
Its relations
with additive combinatorics, with Schinzel's hypothesis, with strong approximation properties
and with sieve methods are all discussed in~\textsection\ref{subsec:knowncases}.
Various results on rational points are deduced in~\textsection\textsection\ref{subsec:consequencesrationalpoints}--\ref{subsec:someunconditional}.

\bigskip
\emph{Acknowledgements.}
Lilian Matthiesen's paper~\cite{matthiesen} plays a significant role
in~\textsection\ref{sec:rationalpoints}.  We are grateful to her for writing it up.
We also thank Tim Browning for pointing out the relevance of Irving's work~\cite{irving} to Conjecture~\ref{conj:mainstrong} and Jean-Louis Colliot-Th\'el\`ene, David Harari and the anonymous referees for their careful reading of the manuscript.
The results of~\cite{browningmatthiesen} provided the initial impetus for the present work.

\bigskip
\emph{Notation and conventions.}
All cohomology groups appearing in this paper are \'etale (or Galois) cohomology groups.
Following~\cite{skorodescent}, a variety~$X$ over a field~$k$ is said to be \emph{split} if it contains a geometrically integral open subset.
If~$X$ is a quasi-projective variety over~$k$, we denote by~$\Sym_{X/k}$
the disjoint union of the symmetric products~$\smash[t]{\Sym^d_{X/k}}$
for $d\geq 1$.
For any variety~$X$ over~$k$, we denote by~$\Zcyc_0(X)$ the group of zero-cycles on~$X$
and by~$\CH_0(X)$ its quotient by the subgroup of cycles rationally equivalent to~$0$.
When~$X$ is proper, we let~$\Azero(X)=\Ker\left(\deg:\CH_0(X)\to\Z\right)$.
We write $\Supp(z)$ for the support of $z\in \Zcyc_0(X)$
and denote by
\begin{align}
\rlap{ }\mylangle-,-\myrangle:\Br(X) \times \Zcyc_0(X) \to \Br(k)
\end{align}
the pairing
characterised by $\mylangle\beta,P\myrangle=\Cores_{k(P)/k}(\beta(P))$ if~$P$ is a closed point of~$X$
and $\beta \in \Br(X)$ (see~\cite[D\'efinition~7]{maninicm}).
If~$X$ is proper, this pairing factors through rational equivalence and induces a pairing
$\Br(X) \times \CH_0(X) \to \Br(k)$
which we still denote $\mylangle-,-\myrangle$
(see \emph{loc.\ cit.}, Proposition~8).

For any abelian group~$M$,
we set $\smash[t]{\widehat{M}}=\varprojlim_{n\geq 1} M/nM$.

When~$k$ is a number field, we let~$\Omega$ denote the set of its places
and~$\Omega_f$ (resp.~$\Omega_\infty$) the subset of finite (resp.~infinite) places.
For $v \in \Omega$, we denote by~$k_v$ the completion of~$k$ at~$v$ and by~$\sOint_v$ the ring of integers of~$k_v$.
For a finite subset $S \subset \Omega$, we denote by~$\sOint_S$ the ring of $S$\nobreakdash-integers of~$k$
(elements of~$k$ which are integral at the finite places not in~$S$).
If~$X$ is a variety over~$k$, we denote by $\Zcyc_{0,\A}(X)$ the subset
of $\prod_{v\in\Omega} \Zcyc_0(X \otimes_k k_v)$ consisting of those families $(z_v)_{v\in\Omega}$
such that if~$\sX$ is a model of~$X$ over~$\sOint_S$
for some finite subset $S \subset \Omega$
(\emph{i.e.}, a faithfully flat $\sOint_S$\nobreakdash-scheme of finite type with generic fiber~$X$),
the Zariski closure of~$\Supp(z_v)$ in $\sX\otimes_\sOint \sOint_v$ is a finite $\sOint_v$\nobreakdash-scheme
for all but finitely many places~$v$ of~$k$.  (This property does not depend on the choice of~$\sX$ and~$S$.)
Note that $\Zcyc_{0,\A}(X)=\prod_{v\in\Omega} \Zcyc_0(X \otimes_k k_v)$ when~$X$ is proper.
The sum of the local pairings $\mylangle-,-\myrangle:\Br(X \otimes_kk_v)\times \Zcyc_0(X\otimes_kk_v) \to \Br(k_v)$
followed by the invariant map $\inv_v:\Br(k_v)\hookrightarrow \Q/\Z$ of local class field theory gives rise to a well-defined pairing
\begin{align}
\label{eq:globalpairing}
\Br(X) \times \Zcyc_{0,\A}(X) \to \Q/\Z\rlap{\text{.}}
\end{align}
If~$X$ is smooth and proper, we let~$\CHzA(X)$ denote the group defined in~\eqref{eq:defchza}
and, for ease of notation, set
\begin{align*}
\Pichat(X)=\widehat{\Pic(X)},\mkern10mu
\CHzhat(X)=\widehat{\CH_0(X)},\mkern10mu
\CHzAhat(X)=\widehat{\CHzA(X)}\rlap{\text{.}}
\end{align*}
The pairing~\eqref{eq:globalpairing} factors through rational equivalence
and gives rise, as~$\Br(X)$ is a torsion group for smooth~$X$, to a pairing
\begin{align}
\label{eq:globalpairinghat}
\Br(X) \times \CHzAhat(X) \to \Q/\Z\rlap{\text{.}}
\end{align}
When~$X$ is a smooth curve, we write~$\PicA(X)$ and~$\PicAhat(X)$ for~$\CHzA(X)$ and~$\CHzAhat(X)$.

A \emph{reduced} divisor on a smooth curve is a divisor all of whose coefficients belong to~$\{0,1\mkern-1mu\}$.

If $f:X\to Y$ is a morphism of schemes and~$Y$ is integral, we write $X_\etabar = X \times_Y \etabar$
for the geometric
generic fiber of~$f$, where~$\etabar$ is a geometric point lying over the generic point of~$Y$.

If~$X$ is a variety over a field~$k$, endowed with a morphism $f:X\to C$ to a smooth proper curve over~$k$,
we denote by $\Zcyc_0^\effred(X)$ the set of effective zero-cycles~$z$ on~$X$
such that~$f_*z$ is a reduced divisor on~$C$.
If~$k$ is a number field, we let $\Zcyc_{0,\A}^\effred(X)=\Zcyc_{0,\A}(X) \cap \prod_{v\in\Omega} \Zcyc_0^\effred(X\otimes_kk_v)$.
If, in addition, a class $y \in \Pic(C)$ is fixed,
we denote by $\Zcyc_0^{\effred,y}(X\otimes_k k_v)$
the inverse image of~$y\otimes_kk_v$ by the push-forward map
$\Zcyc_0^{\effred}(X\otimes_kk_v)\to\Pic(C\otimes_kk_v)$
for $v\in\Omega$
and by~$\Zcyc_{0,\A}^{\effred,y}(X)$ the inverse image of~$y$
by the push-forward map $\Zcyc_{0,\A}^\effred(X) \to \Pic_\A(C)$.
In other words, the set
$\Zcyc_{0,\A}^{\effred,y}(X)$
consists of those collections of local effective zero-cycles~$(z_v)_{v\in\Omega}$
such that~$z_v$ is integral for all but finitely many~$v\in\Omega$, such that~$f_*z_v$ is reduced for all~$v\in\Omega$, such that~$f_*z_v$
is linearly equivalent to $y \otimes_k k_v$ for all~$v \in \Omega_f$
and such that~$f_*z_v$
is linearly equivalent to $y\otimes_k k_v$ up to a norm from $\kbar_v$
for all~$v\in\Omega_\infty$.

For brevity, we adopt the following terminology.

\begin{defn}
A smooth and proper variety~$X$ over a number field~$k$ \emph{satisfies~$\cE$} if
the complex~\eqref{eq:conjE} associated to~$X$ is exact;
it \emph{satisfies~$\cEu$} if the existence of a family $(z_v)_{v\in\Omega} \in \prod_{v\in\Omega} \Zcyc_0(X\otimes_k k_v)$ orthogonal to~$\Br(X)$ with respect
to~\eqref{eq:globalpairing} and such that $\deg(z_v)=1$ for all $v\in \Omega$
implies the existence of a zero-cycle of degree~$1$ on~$X$.
\end{defn}

Any variety which satisfies~$\cE$ also satisfies~$\cEu$ (see~\cite[Remarques~1.1~(iii)]{wittdmj}).

Following Lang \cite[Chapter~9, \textsection5]{langfunddioph},
we say that a subset~$H$ of an irreducible variety~$X$ is a \emph{Hilbert subset}
if there exist a dense open subset $X^0 \subseteq X$, an integer $n \geq 1$ and
irreducible finite \'etale $X^0$\nobreakdash-schemes~$W_1,\dots,W_n$
such that~$H$ is the set of points of~$X^0$
above which the fiber of~$W_i$ is irreducible for all $i\in\{1,\dots,n\}$.
We stress that~$H$ is not a subset of~$X(k)$; for example, the generic point of~$X$ belongs to every Hilbert subset.
Hilbert subsets in this
sense have also been
referred to as \emph{generalised Hilbert subsets} in the literature
(see~\cite{liangcourbe}).

Finally, for the sake of completeness, we include a proof of the following
lemma, which was observed independently by Cao and Xu~\cite[Proposition~3.6]{caoxu}
and by Wei~\cite[Lemma~1.1]{weitorus}
and which will be used in~\textsection\ref{subsec:applicationstrongapprox} below.
If~$v_0$ is a place of a number field~$k$,
we say that a variety~$X$ over~$k$
satisfies \emph{strong approximation off~$v_0$}
if either $X(k_{v_0})=\emptyset$ or the following condition is satisfied:
for any finite subset $S \subset \Omega$ containing $\Omega_\infty\cup\{v_0\}$
and any model~$\sX$ of~$X$ over~$\sOint_S$,
the diagonal map $\sX(\sOint_S)\to \prod_{v \in S \setminus\{v_0\}} X(k_v) \times \prod_{v\notin S} \sX(\sOint_v)$ has dense image.

\begin{lem}
\label{lem:strongapproxaffine}
Let $n \geq 1$ be an integer, let~$k$ be a number field, let~$v_0$ be a place of~$k$ and let $F \subseteq \A^n_k$ be a closed subset of codimension~$\geq 2$.  The variety $U=\A^n_k \setminus F$ satisfies strong approximation off~$v_0$.
\end{lem}

\begin{proof}[Proof, after \cite{weitorus}]
Let us fix a finite subset $S \subset \Omega$ containing $\Omega_\infty\cup\{v_0\}$,
a model~$\sU$ of~$U$ over~$\sOint_S$
and a family $(P_v)_{v\in \Omega} \in \prod_{v\in S}U(k_v) \times \prod_{v \notin S} \sU(\sOint_v)$.
Using weak approximation, we fix $Q \in U(k)$ arbitrarily close to~$P_v$ for~$v \in S$.
Let $S' \subset \Omega$ be a finite subset, containing~$S$,
such that $Q \in \sU(\sOint_{S'})$.  Let $Q' \in U(k)$ be arbitrarily close
to~$P_v$ for $v \in S'\setminus S$ and general enough, in the Zariski topology, that the line $L \subset \A^n_k$
passing through~$Q$ and~$Q'$ is contained in~$U$.
As~$L$ satisfies strong approximation off~$v_0$ and as the Zariski closure~$\sL$ of~$L$
in~$\sU$ possesses an $\sOint_{S'}$\nobreakdash-point (namely~$Q$) and an $\sOint_v$\nobreakdash-point
for each $v \in S' \setminus S$ (namely~$Q'$),
there exists $P \in \sL(\sOint_S)$ arbitrarily close to~$Q$, and hence to~$P_v$,
at the places of $S \setminus \{v_0\}$.
\end{proof}

\section{Reminders on the groups \texorpdfstring{$\Pic_+(C)$ and $\Br_+(C)$}{Pic₊(C) and Br₊(C)}}
\label{sec:reminders}

The groups we denote~$\Pic_+(C)$ and~$\Br_+(C)$ were defined in~\cite[\textsection5]{wittdmj}
as \'etale hypercohomology groups of certain explicit complexes.
We give down-to-earth definitions of these groups in~\textsection\ref{subsec:overarbitraryfield} and recall
in~\textsection\ref{subsec:overnumberfield} a corollary of an arithmetic duality theorem established
in \emph{loc.\ cit.}\ which we shall use in~\textsection\ref{sec:mainexistenceresult}
and in~\textsection\ref{sec:rationalpoints}
and for the proof of which the hypercohomological point of view seems unavoidable when~$C$ has positive genus.

\subsection{Over an arbitrary field}
\label{subsec:overarbitraryfield}

Let~$C$ be a smooth, proper, geometrically irreducible curve over a field~$k$ of characteristic~$0$.
Let $C^0 \subseteq C$ be a dense open subset.  Let $M=C \setminus C^0$.
For each $m \in M$, let~$L_m$ be a nonzero finite \'etale algebra over the residue field~$k(m)$ of~$m$.
(In all of the applications considered in this paper, the algebra~$L_m$ will be a finite field extension of~$k(m)$ whenever~$k$ is a number
field.)

\begin{defn}
\label{def:picpbrp}
To the data of~$C$, of~$M$ and of the finite $k(m)$\nobreakdash-algebras~$L_m$, we associate two groups, denoted $\Pic_+(C)$ and $\Br_+(C)$:
\begin{itemize}
\item the group $\Pic_+(C)$ is the quotient of $\Div(C^0)$ by the subgroup of principal divisors~$\div(f)$ such that for every $m \in M$,
the rational function~$f$ is invertible at~$m$ and its value~$f(m)$ belongs to the subgroup $N_{L_m/k(m)}(L_m^*) \subseteq k(m)^*$.
\smallskip
\item the group $\Br_+(C)$ is the subgroup of $\Br(C^0)$ consisting of those classes whose residue at~$m$ belongs to the kernel of the restriction map
\begin{align*}
H^1(k(m),\Q/\Z) \xrightarrow{\mkern10mu{}r_m\mkern10mu} H^1(L_m,\Q/\Z)
\end{align*}
for every $m\in M$.
\end{itemize}
\end{defn}

Thus~$\Pic_+(C)$ and~$\Br_+(C)$ depend on~$M$ and on the algebras~$L_m$,
though this dependence is not made explicit in the notation.
The group $\Pic_+(C)$ fits into a natural exact sequence
\begin{align}
\label{seq:picp}
\xymatrix{
k^* \ar[r] & \displaystyle\bigoplus_{m \in M} k(m)^*/N_{L_m/k(m)}(L_m^*) \ar[r]^(.68){\delta} & \Pic_+(C) \ar[r] & \Pic(C) \ar[r] & 0\rlap{\text{,}}
}
\end{align}
where~$\delta$ sends the class
of $(a_m)_{m \in M} \in \bigoplus_{m \in M} k(m)^*$
to the class in~$\Pic_+(C)$ of $\div(f)$ for any rational function~$f$ on~$C$, invertible in a neighbourhood of~$M$, such that $f(m)=a_m$ for every $m \in M$ (the existence of such~$f$ follows from the vanishing
of $H^1(U,\sO_U(-M))$ for an affine open neighbourhood~$U$ of~$M$).
Similarly, there is a natural exact sequence
\begin{align}
\label{seq:brp}
\xymatrix{
0 \ar[r] & \Br(C) \ar[r] & \Br_+(C) \ar[r] & \displaystyle\bigoplus_{m\in M}\Ker(r_m) \ar[r] & H^3(C,\Gm)
}
\end{align}
(see \cite[p.~2148]{wittdmj}).

\begin{rmks}
\label{rks:picbr}
(i) The group $\Pic_+(C)$ is a quotient of the relative Picard group $\Pic(C,M)$,
first considered by Rosenlicht.
Recall that $\Pic(C,M)$ may be defined either as the quotient of $\Div(C^0)$ by the subgroup of principal divisors~$\div(f)$ such that $f(m)=1$ for all $m \in M$,
or, in cohomological terms,
as $\Pic(C,M)=H^1(C,\Ker(\Gm\to i_*\Gm))$,
where $i:M \hookrightarrow C$ denotes the inclusion
of~$M$ in~$C$ (see \cite[\textsection2]{suslinvoevodsky}).

(ii) As $\Ker(r_m)$ is finite for each $m \in M$, the quotient $\Br_+(C)/\Br(C)$ is finite.
The finiteness of this quotient will play a crucial role in~\textsection\ref{sec:mainexistenceresult} and~\textsection\ref{sec:rationalpoints},
as the ``formal lemma'' (see~\textsection\ref{sec:formallemma}) can only be applied to a finite subgroup of the Brauer group.
This is the main reason why~$\Pic_+(C)$ is needed in this paper instead of the more classical $\Pic(C,M)$.  Using the latter would entail
dealing with the usually infinite quotient $\Br(C^0)/\Br(C)$.

(iii) The groups $\Pic_+(C)$ and $\Br_+(C)$ were defined in a slightly more general context in \cite[\textsection5]{wittdmj},
where
the input was,
for each~$m\in M$, a finite collection~$(K_{m,i})_{i\in I_m}$
of finite extensions of~$k(m)$ together
with a collection of integers $(e_{m,i})_{i\in I_m}$.
The more restrictive situation considered here, in which~$I_m$ has cardinality~$1$ and $e_{m,i}=1$ for all~$m$, suffices for our purposes.
\end{rmks}

It is also possible to give cohomological definitions for the groups~$\Pic_+(C)$ and~$\Br_+(C)$.
Namely,
let $j:C^0 \hookrightarrow C$
and $\rho:C \to \Spec(k)$ denote the canonical morphisms, and
 for $m \in M$, let $i_m:\Spec(L_m) \to C$ denote the map induced by the $k(m)$\nobreakdash-algebra~$L_m$ and by the inclusion of~$m$ in~$C$.
Endow the category of smooth $k$\nobreakdash-schemes with the \'etale topology.
Let~$\sE$ be the object of the bounded derived category of sheaves of abelian groups, on the corresponding site,
defined by
$\sE=\tau_{\leq 1}R\rho_*\big[j_*\Gm \to \bigoplus_{m \in M} i_{m*}\Z\big]$, where the brackets denote a two-term complex placed in degrees~$0$ and~$1$.

\begin{prop}
\label{prop:cohdef}
There are canonical isomorphisms
 $\Pic_+(C)=H^0(k, \RHom(\sE,\Gm))$
and $\Br_+(C)=H^2(k,\sE)$.
\end{prop}

\begin{proof}
The second isomorphism is obtained in \cite[Remarques~5.2~(i)]{wittdmj}.
For the first isomorphism, we note that
$\Pic(C,M)=H^0(k,\RHom(\tau_{\leq 1}R\rho_*(j_*\Gm),\Gm))$, as
 follows from
Deligne's universal coefficient formula~\cite[(5.4)]{wittdmj},
 from the cohomological definition
of~$\Pic(C,M)$
and
from the exact sequence
$0 \to \Gm \to j_*\Gm \to i_*\Z \to 0$, where~$i$ denotes the inclusion of~$M$ in~$C$.
The natural distinguished triangle
\begin{align}
\xymatrix{
\sE \ar[r] & \tau_{\leq 1}R\rho_*(j_*\Gm) \ar[r] & \displaystyle\bigoplus_{m\in M} (\rho \circ i_m)_*\Z \ar[r]^(.75){+1} &\text{ }
}
\end{align}
then presents~$H^0(k,\RHom(\sE,\Gm))$ as the desired quotient of $\Pic(C,M)$.
\end{proof}

The next proposition is an immediate consequence of Proposition~\ref{prop:cohdef}
(see~\cite[Theorem~18.6.4~(vii)]{kashiwaraschapira}).
The interested reader may also deduce it directly from
Definition~\ref{def:picpbrp} as an amusing exercise.

\begin{prop}
\label{prop:pairing}
The natural pairing $\mylangle-,-\myrangle:\Br(C^0) \times \Div(C^0) \to \Br(k)$
induces a pairing $\Br_+(C) \times \Pic_+(C) \to \Br(k)$.
\end{prop}

\subsection{Over a number field}
\label{subsec:overnumberfield}

Let us now assume that~$k$ is a number field.
We shall denote by $\Pic_+(C \otimes_k k_v)$, $\Br_+(C \otimes_k k_v)$ the
groups associated to the curve, the open subset, and the finite algebras
obtained by scalar extension from~$k$ to~$k_v$.  Let
\begin{align*}
\PicplusA(C) =
\prod_{v \in \Omega_f}\smash{\rlap{\raise 5pt\hbox{$\mkern-8mu{}'$}}} \Pic_+(C\otimes_kk_v)
\times \prod_{v \in \Omega_\infty}\Coker\mkern-1mu\left(N_{\kbar_v/k_v}\!:
\Picplus(C \otimes_k \kbar_v) \to \Picplus(C \otimes_k k_v)\right)
\end{align*}
where the symbol~$\prod'$ denotes the restricted product with respect
to the subgroups $\Im(\Pic_+(\sC \otimes \sOint_v) \to \Pic_+(C \otimes_k k_v))$, for a model~$\sC$ of~$C$
over a dense open subset of the spectrum of the ring of integers of~$k$
(see \cite[\textsection5.3]{wittdmj}).
Composing the pairing $\Pic_+(C \otimes_k k_v) \times \Br_+(C \otimes_k k_v) \to \Br(k_v)$
given by Proposition~\ref{prop:pairing} with the local invariant map $\inv_v:\Br(k_v) \hookrightarrow \Q/\Z$
of local class field theory, and then summing over all $v \in \Omega$, yields a well-defined pairing
\begin{align}
\label{eq:pairingbmplus}
\Brplus(C) \times \PicplusA(C) \to \Q/\Z\rlap{\text{,}}
\end{align}
since $\Br(\sOint_v)=0$ for $v\in\Omega_f$
(see \emph{loc.\ cit.}).  By the global reciprocity law, the image of $\Pic_+(C)$
in $\PicplusA(C)$ is contained in the right kernel of~\eqref{eq:pairingbmplus}.  Thus we obtain a complex of abelian groups
\begin{align}
\xymatrix{
\Pic_+(C) \ar[r] & \PicplusA(C) \ar[r] & \Hom(\Br_+(C),\Q/\Z)\rlap{\text{.}}
}
\end{align}
For topological reasons, this complex fails to be exact when~$C$ has positive genus.
(Indeed, assuming, for simplicity, that $C(k)\neq\emptyset$ and that $L_m=k(m)$ for all $m\in M$,
the group $\Pic_+(C)=\Pic(C)$ is finitely generated by the Mordell--Weil theorem,
so that its degree~$0$ subgroup
often fails to be closed
in the group of adelic points of the Jacobian of~$C$
and therefore also in $\PicplusA(C)=\PicA(C)$.)
The~exactness of an analogous complex where the groups $\Pic_+(C)$ and $\PicplusA(C)$ are replaced with suitable completions
is investigated in~\cite[Th\'eor\`eme~5.3]{wittdmj};
we shall only need the following corollary of this result.

\newcommand{\citecorcinqsept}{{\cite[Corollaire~5.7]{wittdmj}}}
\begin{thm}[\citecorcinqsept]
\label{th:arithduality}
Any element of $\PicplusA(C)$ which is orthogonal to~$\Brplus(C)$
with respect to~\eqref{eq:pairingbmplus}
and whose image in $\Pic_\A(C)$ comes from $\Pic(C)$
is itself the image of an element of $\Picplus(C)$.
\end{thm}

When $C=\P^1_k$, Theorem~\ref{th:arithduality} is a simple consequence of Poitou--Tate duality for the norm~$1$
tori associated to the \'etale algebras $L_m/k(m)$ (see~\cite[Ch.~I, Theorem~4.20]{milneadt} for the
relevant statement).

\section{A version of the formal lemma over a curve}
\label{sec:formallemma}

The name ``formal lemma'' refers to a theorem of Harari
according to which if~$X^0$ is a dense open subset of a smooth variety~$X$ over a number field~$k$
and $B \subset \Br(X^0)$ is a \emph{finite} subgroup, adelic points of~$X^0$ orthogonal
to~$B$ are dense in the set of adelic points of~$X$ orthogonal
to $B \cap \Br(X)$ (see \cite[Corollaire~2.6.1]{harariduke}, \cite[\textsection3.3]{ctsd94},
\cite[Proposition~1.1]{ctskodescent}, \cite[Th\'eor\`eme~1.4]{ctbudapest}).
This result plays a key role in all instances of the fibration method.
In this section, we establish a variant of this theorem for effective zero-cycles on the total space of a fibration over a curve, in which the zero-cycles
are constrained to lie over a fixed linear equivalence class of divisors on the curve.

\begin{prop}
\label{prop:formallemma}
Let~$C$ be a smooth, proper and geometrically irreducible curve
of genus~$g$
over a number field~$k$.
Let $y \in \Pic(C)$ be such that $\deg(y)>2g+1$.
Let~$X$ be a smooth and irreducible variety over~$k$.
Let $f:X \to C$ be a morphism with geometrically irreducible generic fiber.
Let $X^0 \subseteq X$ be a dense open subset
and $B \subset \Br(X^0)$ be a finite subgroup.
For any $(z_v)_{v \in \Omega} \in \Zcyc_{0,\A}^{\effred,y}(X^0)$
orthogonal to $B \cap \Br(X)$ with respect to~\eqref{eq:globalpairing}
and any finite subset $S \subset \Omega$,
there exists $(z'_v)_{v \in \Omega} \in \Zcyc_{0,\A}^{\effred,y}(X^0)$ orthogonal to~$B$
with respect to~\eqref{eq:globalpairing} such that
$z'_v=z_v$ for $v \in S$.
\end{prop}

\begin{proof}
We start with two lemmas.

\begin{lem}
\label{lem:lifty}
Let~$\sX^0$ be a model of~$X^0$ over $\sOint_S$
for some finite subset $S \subset \Omega$.
There exists a finite subset $S' \subset \Omega$ containing~$S \cup \Omega_\infty$ such that
for any $v \in \Omega \setminus S'$,
any finite closed subset $M' \subset C\otimes_kk_v$
and any $y_0 \in \Pic(C \otimes_k k_v)$
with $\deg(y_0)>2g$,
there exists an effective zero-cycle~$z_0$
on $X^0 \otimes_k k_v$ satisfying the following conditions:
\begin{itemize}
\item the divisor~$f_*z_0$ is reduced and supported outside~$M'$;
\item the class of~$f_*z_0$ in $\Pic(C \otimes_k k_v)$ is equal to~$y_0$;
\item the Zariski closure of $\Supp(z_0)$ in $\sX^0 \otimes_{\sOint_S} \sOint_v$ is finite
over~$\sOint_v$.
\end{itemize}
\end{lem}

\begin{proof}
This is a consequence of the moving lemma
\cite[Lemme~4.2]{wittdmj}, of the Lang--Weil--Nisnevich estimates (cf.~\cite{langweil}, \cite{nisnevic})
and of Hensel's lemma. See the proof of \cite[Lemme~4.3]{wittdmj}.
\end{proof}

\begin{lem}
\label{lem:evaluationbeta}
Let $\beta \in \Br(X^0)$.  If~$\beta$ does not belong
to the subgroup $\Br(X) \subseteq \Br(X^0)$,
the map $\Zcyc_0^{\effred,y}(X^0 \otimes_k k_v) \to \Br(k_v)$,
$z \mapsto \mylangle \beta,z\myrangle$ is nonzero
for infinitely many $v\in \Omega$.
\end{lem}

\begin{proof}
Choose a model~$\sX^0$ of~$X^0$ over~$\sOint_S$ for some finite subset $S \subset \Omega$.
After enlarging~$S$, we may assume that~$\beta$ comes from $\Br(\sX^0)$.
Let~$S'$ be given by Lemma~\ref{lem:lifty}.
According to \cite[Th\'eor\`eme~2.1.1]{harariduke}, there are infinitely many $v \in \Omega$
such that $\beta(a)\neq 0$ for some $a\in X^0(k_v)$.
Pick $v \notin S'$ and $a \in X^0(k_v)$
such that $\beta(a)\neq 0$
and let $y_0 \in \Pic(C \otimes_k k_v)$ denote the class $y-f(a)$.
By the definition of~$S'$,
there exists an effective zero-cycle~$z_0$ on $X^0 \otimes_k k_v$ satisfying
the conditions of Lemma~\ref{lem:lifty} with $M'=\{f(a)\}$.
Letting $z=z_0+a$,
we have $z \in \Zcyc_0^{\effred,y}(X^0\otimes_k k_v)$
and $\mylangle \beta, z \myrangle = \mylangle \beta, z_0 \myrangle + \beta(a)$.
As~$\beta$ comes from $\Br(\sX^0)$ and the closure of~$\Supp(z_0)$
in $\sX^0\otimes_{\sOint_S}\sOint_v$ is finite over~$\sOint_v$,
we have $\mylangle \beta,z_0\myrangle=0$, so $\mylangle \beta,z\myrangle\neq 0$.
\end{proof}

Proposition~\ref{prop:formallemma} follows from Lemma~\ref{lem:evaluationbeta}
by the formal argument used in the proof of~\cite[Corollaire~2.6.1]{harariduke}.
For the convenience of the reader, we reproduce it briefly.
Let $B^*=\Hom(B,\Q/\Z)$. For $v \in \Omega$, let $\phi_v:\Zcyc_0^{\effred,y}(X^0 \otimes_k k_v)\to B^*$ be
defined by $\phi_v(z)(\beta)=\inv_v\mylangle \beta, z\myrangle$.
Let $\Gamma \subseteq B^*$ be the subgroup generated by those elements
which belong to the image of~$\phi_v$ for infinitely many~$v\in\Omega$.
Recall that we are given a finite set~$S$ and a family $(z_v)_{v \in \Omega}$
in the statement of Proposition~\ref{prop:formallemma}.
After enlarging~$S$, we
may assume that~$\phi_v$ takes values in~$\Gamma$
for any $v \notin S$ and that $\phi_v(z_v)=0$ for $v \notin S$.
Consider the natural pairing $B \times B^* \to \Q/\Z$.
Any element of~$B$
which is orthogonal to~$\Gamma$ belongs to $B \cap \Br(X)$
according to Lemma~\ref{lem:evaluationbeta}, hence is orthogonal
to $w=\sum_{v\in S}\phi_v(z_v)$.
As~Pontrjagin duality is a perfect duality among finite abelian groups,
it follows that $w \in \Gamma$.
Therefore there exist a finite set $T \subset \Omega$ disjoint from~$S$ and
a family $(z'_v)_{v \in T} \in \prod_{v\in T}\Zcyc_0^{\effred,y}(X^0\otimes_kk_v)$ such that $w=-\sum_{v \in T}\phi_v(z'_v)$.
Letting $z'_v=z_v$ for $v \in \Omega\setminus T$
gives the desired family $(z'_v)_{v\in \Omega}$.
\end{proof}

\section{Specialisation of the Brauer group}
\label{sec:specialisation}

The following result is a theorem of Harari when $C=\P^1_k$ (see~\cite[Th\'eor\`eme~2.3.1]{hararifleches}, see also \cite[\textsection3]{harariduke}).
The arguments of \cite{hararifleches} still apply when~$C$ is an arbitrary curve, once one knows that $H^3(k(C),\Gm)=0$.  For the sake of completeness, we nevertheless include a shorter, self-contained proof.
In the next statement, we
denote by~$\eta$ the generic point of~$C$,
we identify~$h$ and~$\eta$
with $\Spec(k(h))$ and $\Spec(k(C))$
and we
denote by~$f_h:X_h \to \Spec(k(h))$ the fiber of~$f$ above~$h$.

\begin{prop}
\label{prop:specialisation}
Let~$C$ be a smooth irreducible curve over a number field~$k$.
Let~$X$ be an irreducible variety over~$k$
and $f:X \to C$ be a morphism whose
geometric generic fiber~$X_\etabar$ is smooth, proper and irreducible and satisfies
$H^1(X_{\bar\eta},\Q/\Z)=0$
and $H^2(X_\etabar,\sO_{X_\etabar})=0$.
Let $C^0 \subseteq C$ be a dense open subset, let $X^0=f^{-1}(C^0)$ and let
$B \subseteq \Br(X^0)$ be a subgroup.
If the natural map $B \to \Br(X_\eta)/f_\eta^*\mkern.5mu\Br(\eta)$ is surjective,
there exists
a Hilbert subset $H \subseteq C^0$ such that the natural map $B \to \Br(X_h)/f_h^*\mkern.5mu\Br(h)$ is surjective
for all $h \in H$.
\end{prop}

\begin{proof}
By shrinking~$C$, we may assume that $C^0=C$ and that~$f$ is smooth and proper.
Let~$V \subseteq H^2(X,\Q/\Z(1))$ be the inverse image of~$B$ by the second map of the exact sequence
\begin{align}
\label{eq:picbr}
\xymatrix{
0 \ar[r] & \Pic(X) \otimes_\Z\Q/\Z \ar[r] & H^2(X,\Q/\Z(1)) \ar[r] & \Br(X) \ar[r] & 0
}
\end{align}
(see~\cite[II, Th\'eor\`eme~3.1]{grbr}).
Let $\Gal(\etabar/\eta)$ and $\Gal(\bar h/h)$ denote the absolute Galois groups of~$k(C)$ and of~$k(h)$, respectively,
and consider the commutative diagram
\begin{align}
\begin{aligned}
\xymatrix@R=3.5ex@C=5em{
H^2(X_\eta,\Q/\Z(1))/f_\eta^*\mkern.5muH^2(\eta,\Q/\Z(1)) \ar[r]^(.55){\alpha_\eta} & H^2(X_{\bar\eta},\Q/\Z(1))^{\Gal(\etabar/\eta)} \\
V \ar[u]^{\beta_\eta} \ar[d]_(.42){\beta_h} \ar[r] & H^0(C,R^2f_*\Q/\Z(1)) \ar[u]^{\gamma_\eta} \ar[d]_(.42){\gamma_h} \\
H^2(X_h,\Q/\Z(1))/f_h^*\mkern.5muH^2(h,\Q/\Z(1)) \ar[r]^(.55){\alpha_h} & H^2(X_{\bar h},\Q/\Z(1))^{\Gal(\bar h/h)}
}
\end{aligned}
\end{align}
for $h \in C$,
where~$\gamma_\eta$ and~$\gamma_h$ are pull-back maps.
By comparing~\eqref{eq:picbr} with the analogous exact sequence for~$X_\eta$,
we note that the surjectivity
of the natural maps $B \to \Br(X_\eta)/f_\eta^*\mkern.5mu\Br(\eta)$
and $\Pic(X)\to\Pic(X_\eta)$ implies the surjectivity of~$\beta_\eta$.
We shall now deduce the existence of a Hilbert subset~$H \subseteq C$
such that~$\beta_h$ is surjective for all $h\in H$.

\begin{lem}
\label{lem:galfinitequotient}
The group $\Gal(\etabar/\eta)$ acts on $H^2(X_{\etabar},\Q/\Z(1))$ through a finite quotient.
\end{lem}

\begin{proof}
By~\cite[II, \textsection3.5]{grbr},
the group $H^2(X_\etabar,\Q/\Z(1))$ is an extension of $\Br(X_\etabar)$, which
is finite as
$H^2(X_\etabar,\sO_{X_\etabar})=0$ (see, \emph{e.g.}, \cite[Lemma~1.3]{ctskogoodreduction}),
by $\NS(X_\etabar)\otimes_\Z\Q/\Z$,
on which~$\Gal(\etabar/\eta)$ acts through a finite quotient
since~$\NS(X_\etabar)$ is finitely generated.
\end{proof}

\begin{lem}
\label{lem:existshilbert}
There exists a Hilbert subset $H \subseteq C$ such
that~$\gamma_h$ is an isomorphism for every $h \in H$.
In particular, the map~$\gamma_\eta$ is an isomorphism.
\end{lem}

\begin{proof}
As~$f$ is smooth and proper, the \'etale sheaf $R^2f_*\Q/\Z(1)$ is a direct limit of locally constant sheaves with finite stalks
(see \cite[Ch.~VI, Corollary~4.2]{milneet}).
Therefore
it gives rise, for any $h \in C$, to
an action of the profinite group $\pi_1(C,\bar h)$ on $H^2(X_{\bar h},\Q/\Z(1))$.
In terms of this action, the map~$\gamma_h$ can be naturally identified with the inclusion
$$H^2(X_{\bar h},\Q/\Z(1))^{\pi_1(C,\bar h)} \subseteq H^2(X_{\bar h},\Q/\Z(1))^{\Gal(\bar h/h)}$$
(see \cite[p.~156]{milneet}).
By Lemma~\ref{lem:galfinitequotient}, the group~$\pi_1(C,\bar\eta)$ acts on $H^2(X_{\bar\eta},\Q/\Z(1))$
through the automorphism group of some irreducible finite Galois cover $\pi:D\to C$.
The map~$\gamma_h$ is then an isomorphism for any $h \in C$
such that $\Gal(\bar h/h)$ surjects onto the quotient of~$\pi_1(C,\bar h)$
corresponding to~$\pi$, \emph{i.e.}, for any $h\in C$ such that $\pi^{-1}(h)$ is irreducible.
\end{proof}

As the groups $H^3(\eta,\Q/\Z(1))$,
$H^1(X_\etabar,\Q/\Z)$
and therefore also
$H^1(X_{\bar h},\Q/\Z)$ all vanish
(by \cite[Lemma~2.6]{ctpoonen} for the first group
and by the smooth base change theorem \cite[Ch.~VI, Corollary~4.2]{milneet} for the third group),
we deduce from the Hochschild--Serre spectral sequence
that~$\alpha_\eta$ is surjective and~$\alpha_h$ is injective
(these two maps are in fact isomorphisms).
If~$H$ is given by Lemma~\ref{lem:existshilbert},
we conclude that for all $h \in H$, the surjectivity of~$\beta_h$,
and hence of the natural map $B \to \Br(X_h)/f_h^*\mkern.5mu\Br(h)$,
follows from that of~$\beta_\eta$.
\end{proof}

\begin{rmk}
Proposition~\ref{prop:specialisation} and its proof go
through for fibrations over an arbitrary base
as long as the natural map
$H^3(\eta,\Q/\Z(1))\to H^3(X_\eta,\Q/\Z(1))$ is injective.
\end{rmk}

\section{Existence of zero-cycles}
\label{sec:mainexistenceresult}

The goal of~\textsection\ref{sec:mainexistenceresult} is to establish the following
existence result for zero-cycles on the total space of a fibration over a curve of arbitrary genus.

\begin{thm}
\label{th:existenceresult}
Let~$C$ be a smooth, projective and geometrically irreducible curve over
a number field~$k$.
Let~$X$ be a smooth, projective, irreducible variety over~$k$
and $f:X \to C$ be a morphism with geometrically irreducible generic fiber
and no multiple fiber (in the sense that the gcd of the multiplicities of the irreducible components of each fiber is~$1$).
Let~$C^0 \subseteq C$ be a dense open subset over which~$f$ is smooth.
Let~$X^0=f^{-1}(C^0)$.
Let~$B \subset \Br(X^0)$ be a finite subgroup.
Let~$y \in \Pic(C)$
and $(z_v)_{v \in \Omega} \in \Zcyc_{0,\A}^{\effred,y}(X^0)$.
Let $S \subset \Omega_f$ be a finite subset.
Let $\sU \subseteq \prod_{v \in S} \Sym_{X^0/k}(k_v)$
be a neighbourhood of $(z_v)_{v \in S}$.
Assume that
\begin{enumerate}
\item[(i)]
the family $(z_v)_{v \in\Omega}$ is
orthogonal to $(B+f_\eta^*\mkern.5mu\Br(\eta)) \cap \Br(X)$ with respect to~\eqref{eq:globalpairing};
\item[(ii)]
 $\deg(y)\geq\deg(M)+2g+2$, where $M=C \setminus C^0$ and~$g$ denotes the genus of~$C$;
\item[(iii)]
for each real place~$v$ of~$k$, there exists a smooth, proper, geometrically
irreducible curve $Z_v \subset X \otimes_k k_v$ which contains the support of~$z_v$
and dominates~$C$, such that $\deg(y) \geq [k_v(Z_v):k_v(C)]\deg(M)+2g_v+2$,
where~$g_v$ denotes the genus of~$Z_v$.
\end{enumerate}
Then there exist
a family $(z'_v)_{v \in \Omega} \in \Zcyc_{0,\A}(X^0)$
and an effective divisor~$c$ on~$C^0$ whose class
in~$\Pic(C)$ is~$y$,
such that
\begin{enumerate}
\item $f_*z'_v=c$ in $\Div(C^0 \otimes_k k_v)$ for each $v \in \Omega$;
\item the family $(z'_v)_{v \in \Omega}$ is orthogonal to~$B$ with respect to~\eqref{eq:globalpairing};
\item $z'_v$ is effective for each $v \in S$
and $(z'_v)_{v\in S}$ belongs to~$\sU$;
\item \label{item:B} for each real place~$v$ of~$k$
and each connected component~$\sB$ of $X^0(k_v)$,
if we write $z_v-z'_v=\sum_{x \in X^0 \otimes_k k_v} n_x\mkern1mux$,
then $\sum_{x \in \sB} n_x$ is even.
\end{enumerate}
Furthermore, if every fiber of~$f$ possesses an irreducible component of multiplicity~$1$,
one can ensure that~$z'_v$ is effective for all $v \in \Omega$.
\end{thm}

Let us emphasise that as per the conventions spelled out in~\textsection\ref{sec:intro}, the notation $\smash[b]{\Zcyc_{0,\A}^{\effred,y}(X^0)}$
in the above statement refers to the morphism $X^0 \to C$ induced by~$f$.
Thus, for $v\in \Omega_f$,
the linear equivalence class of~$f_*z_v$ is assumed to coincide with~$y$ in $\Pic(C \otimes_k k_v)$,
not solely in $\Pic(C^0 \otimes_k k_v)$ (and similarly, up to norms, for $v\in\Omega_\infty$).
Moreover,
let us stress that
in~\eqref{item:B},
we identify~$X^0(k_v)$ with a subset
of $X^0 \otimes_k k_v$. Thus, for $x \in X^0 \otimes_k k_v$,
one has~$x \in \sB$ if and only if $k_v(x)=k_v$
and the corresponding $k_v$\nobreakdash-point of~$X^0$ belongs to~$\sB$.

The proof of Theorem~\ref{th:existenceresult} occupies~\textsection\textsection\ref{subsec:applicationstrongapprox}--\ref{subsec:constructionzpv}.
We shall deal with curves~$C$ of arbitrary genus directly,
without reducing to the case of genus~$0$, using the Riemann--Roch theorem in a spirit
closer to~\cite{ctreglees} than to~\cite{wittdmj}.

\subsection{A consequence of strong approximation}
\label{subsec:applicationstrongapprox}

Lemma~\ref{lem:strongapprox} below, which we state with independent notation,
 will play a decisive role in the proof
of Theorem~\ref{th:existenceresult}.
It should be compared with Dirichlet's theorem on primes in arithmetic progressions
for general number fields.  According to the latter,
given
a finite set~$S$ of finite places of~$k$ and elements~$t_v \in k_v^*$ for~$v \in S$,
there exists a totally positive $t \in k^*$ arbitrarily close to~$t_v$ for~$v \in S$ such
that~$t$ is a unit outside~$S$ except at one place, at which it is a
uniformiser (see~\cite[Ch.~V, Theorem~6.2]{neukirchcft}).

\begin{lem}
\label{lem:strongapprox}
Let~$L/k$ be a finite extension of number fields.
Let~$S$ be a finite set of places of~$k$.
For $v \in S$, let $t_v \in k_v^*$.
Suppose~$t_v$ is a norm from $(L \otimes_k k_v)^*$ for every $v \in S$.
Then there
exists $t \in k^*$,
arbitrarily close to~$t_v$
for~$v \in S$,
such that for any finite place~$v \notin S$, either~$t$ is a unit at~$v$
or~$L$ possesses a place of degree~$1$ over~$v$.
In addition, if~$v_0$ is a place of~$k$, not in~$S$, over which~$L$ possesses a place of degree~$1$,
one can ensure that~$t$ is integral outside $S \cup \{v_0\}$.
\end{lem}

\begin{proof}
The complement of the quasi-trivial
torus $R_{L/k}(\mathbf{G}_{\mathrm{m},{L}})$
in the affine space $R_{L/k}(\A^1_{L})$
is a divisor with normal crossings (geometrically
isomorphic to the union of all coordinate hyperplanes).
We denote it by~$D$, we let~$F$
be its singular locus
and set $W = R_{L/k}(\A^1_L) \setminus F$.
Similarly, we let $\sT=R_{\sOint_L/\sOint_k}(\mathbf{G}_{\mathrm{m},{\sOint_L}})$,
 $\sD=R_{\sOint_L/\sOint_k}(\A^1_{\sOint_L})\setminus \sT$
and
 $\sW=R_{\sOint_L/\sOint_k}(\A^1_{\sOint_L}) \setminus \sF$,
where~$\sF$ denotes the singular locus of~$\sD$.

Letting $t_v=1$ for $v \notin S$, we may enlarge~$S$ and assume that~$L/k$
is unramified outside~$S$.
For each~$v \in S$, fix $x_v \in (L\otimes_kk_v)^*$
such that $t_v = N_{L\otimes_kk_v/k_v}(x_v)$.
For $v\notin S$, let $x_v=1$.
This defines an adelic point $(x_v)_{v \in \Omega} \in W(\A_k)$
which, with respect to~$\sW$, is integral outside~$S$.

Let~$v_0$ denote a place of~$k$, not in~$S$, over which~$L$ possesses a place of degree~$1$.
If~$v_0$ is not given, we fix it by Chebotarev's density theorem (see~\cite[p.~250, Lemma~5]{mumford} for an elementary proof of the existence of~$v_0$ splitting completely in~$L$).

Being the complement of a codimension~$2$ closed subset in an affine
space, the variety~$W$ satisfies strong approximation off any given place
(see Lemma~\ref{lem:strongapproxaffine}).
We can thus find a point $x \in W(k)$, arbitrarily close to~$x_v$ for $v \in S$,
which is integral, with respect to~$\sW$, outside $S\cup\{v_0\}$.
Let us consider~$x$ as an element of~$L^*$ and set $t=N_{L/k}(x)$.
Let~$v$ be a finite place of~$k$, not in~$S$, such that~$t$ has nonzero valuation at~$v$.
We need to check that~$L$ possesses a place of degree~$1$ over~$v$.  If $v=v_0$, there is nothing to prove.
Otherwise $x\in \sW(\sOint_v)$.
On the other hand, as $v(t)\neq 0$, we have $x \notin \sT(\sOint_v)$.
Thus the reduction of~$x$ mod~$v$ lies in
$(\sD \setminus \sF)(\F_v)$,
which, by Hensel's lemma, implies that~$(D\setminus F)(k_v)\neq\emptyset$.
This, in turn, implies that~$L$ embeds $k$\nobreakdash-linearly into~$k_v$,
as the structure morphism $D \setminus F \to \Spec(k)$ factors through $\Spec(L)$.
\end{proof}

\subsection{Reduction to adelic zero-cycles orthogonal to \texorpdfstring{$B$}{B}}
\label{subsec:reduction}

We resume the notation of Theorem~\ref{th:existenceresult}
and keep it until the end of~\textsection\ref{sec:mainexistenceresult}.

The fibers of~$f$ over~$C^0$ are geometrically connected
(see~\cite[Proposition~15.5.9~(ii)]{ega43}) and hence geometrically irreducible, as they are smooth.

For $m \in M$, let $(X_{m,i})_{i \in I_m}$ denote the family of irreducible
components of~$X_m$.
Let~$e_{m,i}$ denote the multiplicity of~$X_{m,i}$ in~$X_m$.
For each $m \in M$ and $i \in I_m$, choose a finite extension~$E_{m,i}/k(X_{m,i})$ such that
the residue of any element of~$B$ at the generic point of~$X_{m,i}$ belongs to the kernel of the
restriction map $H^1(k(X_{m,i}),\Q/\Z)\to H^1(E_{m,i},\Q/\Z)$.
Let~$L_{m,i}$ denote the algebraic closure of~$k(m)$ in~$E_{m,i}$.
Let~$L_m/k(m)$ be a finite Galois extension in which $L_{m,i}/k(m)$ embeds for all $i \in I_m$.

Let~$\Pic_+(C)$ and~$\Br_+(C)$ denote the groups associated in Definition~\ref{def:picpbrp} to the curve~$C$,
to the finite set $M \subset C$ and to the finite extensions~$L_m/k(m)$ for $m \in M$.

\begin{lem}
\label{lem:inclusionfiniteness}
We have the inclusion
$$(B+f_\eta^*\mkern.5mu\Br(\eta))\cap\Br(X) \subseteq B+f^*\Br_+(C)$$
of subgroups of $\Br(X^0)$.
Moreover,
the quotient $((B+f_\eta^*\mkern.5mu\Br(\eta))\cap \Br(X))/f^*\Br(C)$ is finite.
\end{lem}

\begin{proof}
The second assertion, which will be used in~\textsection\ref{sec:hilbertsubsets} and in~\textsection\ref{sec:rationalpoints}, follows from the first one and from Remark~\ref{rks:picbr}~(ii). To prove the first,
it suffices to check that if $\beta \in B$ and $\gamma\in\Br(\eta)$ satisfy $\beta+f_\eta^*\mkern.5mu\gamma \in \Br(X)$, then $\gamma \in \Br_+(C)$.
Fix $m \in C$. Let $\partial_\gamma \in H^1(k(m),\Q/\Z)$
denote the residue of~$\gamma$ at~$m$.
As the fibers of~$f$ over~$C^0$ are smooth and geometrically irreducible
and as~$f_\eta^*\mkern.5mu\gamma \in \Br(X^0)$, we have $\partial_\gamma=0$ if $m \in C^0$.
Suppose now $m \in M$.
For $i \in I_m$,
the residue of $\beta+f_\eta^*\mkern.5mu\gamma$
at the generic point of~$X_{m,i}$ is trivial and its image in $H^1(E_{m,i},\Q/\Z)$
coincides with the image of $e_{m,i}\partial_\gamma$
(see~\cite[Proposition~1.1.1]{ctsd94}).
Therefore $e_{m,i}\partial_\gamma$ vanishes in $H^1(L_{m,i},\Q/\Z)$, hence in $H^1(L_m,\Q/\Z)$.
As the gcd of $(e_{m,i})_{i\in I_m}$ is equal to~$1$,
it follows that~$\partial_\gamma$
vanishes in $H^1(L_m,\Q/\Z)$.
\end{proof}

\begin{lem}
\label{lem:appformallemma}
There exists a family $(z^1_v)_{v\in\Omega}\in \Zcyc_{0,\A}^{\effred,y}(X^0)$
orthogonal to $B+f^*\Br_+(C)$ with respect to~\eqref{eq:globalpairing}
such that $z^1_v=z_v$ for $v \in S\cup \Omega_\infty$.
\end{lem}

\begin{proof}
As the group $\Br_+(C)$ is torsion and the quotient $\Br_+(C)/\Br(C)$ is finite
(see Remark~\ref{rks:picbr}~(ii)),
there exists a finite subgroup $\Lambda \subset \Br_+(C)$ such that $\Br_+(C)=\Lambda+\Br(C)$.
The lemma then follows from
Proposition~\ref{prop:formallemma}
applied to the subgroup $B+f^*\Lambda \subset \Br(X^0)$,
in view of the fact that
any element of $\Zcyc_{0,\A}^{\effred,y}(X^0)$ is orthogonal to~$f^*\Br(C)$.
\end{proof}

It follows
from Lemmas~\ref{lem:inclusionfiniteness} and~\ref{lem:appformallemma}
that in order to establish Theorem~\ref{th:existenceresult},
we may assume that~$(z_v)_{v \in \Omega}$ is orthogonal to $B+f^*\Br_+(C) \subseteq \Br(X^0)$ with respect to~\eqref{eq:globalpairing}.
From now on we make this additional assumption.

\subsection{Construction of \texorpdfstring{$c_1$}{c₁} and choice of \texorpdfstring{$v_0$}{v₀}}
\label{subsec:constrc1v0}

We are now in a position to pinpoint the class, in $\Pic_+(C)$, of the desired effective divisor~$c$.
The construction of~$c$ itself will ultimately rely on an application of strong approximation off a place,
say~$v_0$,
on an affine space over~$k$.
The goal of~\textsection\ref{subsec:constrc1v0} is to define this place~$v_0$.

The condition that $(z_v)_{v \in \Omega}$ is orthogonal to~$f^*\Br_+(C)$ with respect to~\eqref{eq:globalpairing}
implies that the class of $(f_*z_v)_{v\in \Omega}$ in~$\PicplusA(C)$
is orthogonal to~$\Br_+(C)$ with respect to the pairing~\eqref{eq:pairingbmplus}.
By Theorem~\ref{th:arithduality}, we deduce
that there exists a divisor $c_1 \in \Div(C^0)$ whose class in~$\Pic(C)$ is~$y$
and whose class in~$\PicplusA(C)$ coincides with that of $(f_*z_v)_{v \in \Omega}$.

After enlarging the finite subset $S \subset \Omega_f$, we may assume that
$\mylangle \beta,z_v\myrangle=0$ for all $v \in \Omega_f\setminus S$ and all $\beta \in B$,
that the order of the finite group~$B$ is invertible in~$\sOint_S$,
that $\Pic(\sOint_S)=0$,
that the finite extensions $L_m/k$ for $m \in M$ are unramified outside~$S$,
that~$C$ and~$X$
extend to smooth and proper $\sOint_S$\nobreakdash-schemes~$\sC$ and~$\sX$ and that $f:X\to C$ extends to a flat morphism,
which we still denote $f:\sX \to \sC$.

For $m \in M$, we let~$\mtilde$ denote the Zariski closure of~$m$ in~$\sC$.
Let~$\Mtilde$ denote the Zariski closure of~$M$ in~$\sC$.
Let $\sC^0=\sC \setminus \Mtilde$ and $\sX^0=\sX \times_{\sC} \sC^0$.
Whenever~$D$ is a divisor on~$C$, we denote by~$\widetilde{D}$ the
horizontal divisor on~$\sC$ which extends~$D$,
\emph{i.e.},
the unique divisor on~$\sC$
whose support is flat over~$\sOint_S$ and
whose restriction to~$C$ is~$D$.
For $m \in M$ and $i \in I_m$, we write~$\sX_{\mtilde,i}$ for the Zariski closure of~$X_{m,i}$ in~$\sX$, endowed with the reduced
scheme structure, and we set $\sX_\mtilde=\sX \times_\sC \mtilde$.
For each $m \in M$ and $i \in I_m$, let $\sY_{\mtilde,i} \subseteq \sX_{\mtilde,i}$ be a dense open subset.
Let~$\sE_{\mtilde,i}$
denote the normalisation of~$\sY_{\mtilde,i}$ in the finite extension~$E_{m,i}/k(X_{m,i})$.
By shrinking~$\sY_{\mtilde,i}$, we may assume that the finite
morphism~$\sE_{\mtilde,i}\to\sY_{\mtilde,i}$
is \'etale and that the scheme $(\sX_{\mtilde})_\red$ is smooth over~$\mtilde$ at the points
of~$\sY_{\mtilde,i}$.

By enlarging~$S$ further, we may assume that
the equality of effective divisors on~$\sX$
\begin{align*}
\sX_\mtilde = \sum_{i \in I_m} e_{m,i}\sX_{\mtilde,i}
\end{align*}
holds for each $m \in M$,
that~$\sX^0$ is smooth over~$\sC^0$,
that $B \subset \Br(\sX^0)$,
that $\Supp(\widetilde{c_1})$ is disjoint from~$\Mtilde$
and that~$\Mtilde$ is \'etale over~$\sOint_S$.
In particular~$\Mtilde$ is a regular scheme, so that we may (and will freely)
identify,
for $m\in M$, the set of finite places of~$k(m)$ which do not lie above~$S$
with the set of closed points of~$\mtilde$.

After a final enlargement of~$S$,
we may assume, 
thanks to the Lang--Weil--Nisnevich bounds~\cite{langweil} \cite{nisnevic},
that the following statements hold:
\begin{itemize}
\item the fiber of~$f$ above any closed point of~$\sC^0$
contains a rational point;
\smallskip
\item for any $m \in M$, any $i \in I_m$ and any closed point $w \in \mtilde$
which, as a place of~$k(m)$, splits completely in~$L_{m,i}$,
the fiber of $\sE_{\mtilde,i} \to \mtilde$ above~$w$ contains a rational point.
\end{itemize}

By shrinking the subset $\sU \subseteq \prod_{v \in S} \Sym_{X^0/k}(k_v)$, we may assume that
$\mylangle \beta,z'_v\myrangle=\mylangle\beta,z_v\myrangle$ for any~$v \in S$, any $\beta \in B$ and any $(z'_v)_{v \in S} \in \sU$.

We now fix a place $v_0 \in \Omega_f \setminus S$ which splits completely
in~$L_m$ for all $m \in M$.

\subsection{Construction of \texorpdfstring{$c$}{c}}
\label{subsec:constructionc}

As the classes of~$c_1$ and of~$(f_*z_v)_{v\in\Omega}$ in $\PicplusA(C)$ are equal and the divisors~$f_*z_v$
for $v \in S$
are effective, we can find,
for each $v \in S$, a rational function $h_v \in H^0(C \otimes_k k_v, \sO_C(c_1))$
such that $\div(h_v)=f_*z_v-c_1$
and such that $h_v(m) \in N_{L_m \otimes_k k_v/k(m)\otimes_kk_v}((L_m\otimes_kk_v)^*)$
for all~$m\in M$.
At the infinite places~$v$ of~$k$, we fix $h_v \in H^0(C\otimes_k k_v,\sO_C(c_1))$ using the following lemma.

\begin{lem}
For each $v\in\Omega_\infty$,
there exist a rational function
 $h_v \in H^0(C \otimes_k k_v,\sO_C(c_1))$
and a zero-cycle
 $z^2_v \in \Zcyc_0^{\effred,y}(X^0 \otimes_kk_v)$
satisfying the following conditions:
\begin{enumerate}
\item[(a)] $\div(h_v)=f_*z^2_v - c_1$;
\item[(b)]
$h_v(m) \in N_{L_m \otimes_k k_v/k(m)\otimes_kk_v}((L_m\otimes_kk_v)^*)$
for all $m \in M$;
\item[(c)] for each connected component~$\sB$ of~$X^0(k_v)$, if we write
$z_v-z^2_v=\sum_{x \in X^0 \otimes_k k_v} n_x\mkern1mux$,
then $\sum_{x \in \sB} n_x$ is even.
\end{enumerate}
\end{lem}

\begin{proof}
The existence of~$h_v$ and~$z_v^2$ for complex~$v$ is clear as~$c_1$ is a very ample
divisor on~$C$ (recall that $\deg(y)\geq 2g+1$).
Let us fix a real place~$v$.
By the definition of~$c_1$, there exist $h_v^1 \in k_v(C)^*$
and $\xi_v \in \Div(C^0 \otimes_k \kbar_v)$
such that $\div(h_v^1)=f_*z_v-c_1+N_{\kbar_v/k_v}(\xi_v)$
and such that $h_v^1(m) \in (k(m)\otimes_k k_v)^*$ is a norm from~$(L_m \otimes_k k_v)^*$ for every $m\in M$.
Let $Z_v\subset X\otimes_kk_v$ denote the curve given by assumption~(iii) of Theorem~\ref{th:existenceresult}.
Let $Z_v^0=Z_v \times_C C^0$ and $M_v = Z_v \times_C M$.
By lifting~$\xi_v$ to $\Div(Z_v^0 \otimes_{k_v} \kbar_v)$,
we can find $z_v^1 \in \Div(Z_v^0)$ such that $z_v-z_v^1$ is a norm
from $\Div(Z_v^0 \otimes_{k_v} \kbar_v)$ and  such that $\div(h_v^1)=f_*z^1_v-c_1$.
According to assumption~(iii)
and to Serre duality, we have
$H^1(Z_v,\sO_{Z_v}(z_v^1-M_v))=0$, hence
there exists a rational function $\eta_v \in H^0(Z_v,\sO_{Z_v}(z_v^1))$
whose restriction to~$M_v$ is equal to~$1$.
We have $\div(\eta_v)=z_v^2-z_v^1$ for some effective $z_v^2\in \Div(Z_v^0)$.
By choosing~$\eta_v$ general enough in the sense of the Zariski topology,
we may assume that $f_*z_v^2$ is reduced (see the proof of \cite[Lemme~3.1]{ctreglees})
and hence that $z_v^2 \in \Zcyc_0^{\effred,y}(X^0\otimes_kk_v)$.
Let $h_v=\mkern2muN_{k_v(Z_v)/k_v(C)}(\eta_v)\mkern1muh_v^1$.
Conditions~(a) and~(b) are clearly satisfied (note that $h_v(m)=h_v^1(m)$ for $m\in M$).
Condition~(c) also holds because on the one hand $z_v-z_v^1$ is a norm from~$\kbar_v$
and on the other hand $z_v^1-z_v^2$ is the divisor, on~$Z_v$, of a rational function which takes positive
values at the $k_v$\nobreakdash-points of~$M_v$.
(The sign of a rational function on a smooth real algebraic curve changes exactly at the points
where the function has a pole or zero of odd order.)
\end{proof}

For each $m \in M$, applying Lemma~\ref{lem:strongapprox}
to the finite extension $L_m/k(m)$
and to the set of places of~$k(m)$ which lie above~$S$
now produces an element $t_m \in k(m)^*$
 arbitrarily close to
 $h_v(m) \in (k(m)\otimes_k k_v)^*$
 for $v \in S\cup\Omega_\infty$,
integral at the places of~$k(m)$ which do not lie above $S \cup \{v_0\}$,
such that~$t_m$ has nonzero valuation only at places of~$k(m)$
which lie above~$S$ and at places of~$k(m)$ which split completely in~$L_m$ (recall
that $L_m/k(m)$ is Galois).

Thanks to assumption~(ii) and to Serre duality,
if $\rho:\sC \to\Spec(\sOint_S)$ denotes the structure morphism,
we have $R^1 \rho_* \sO_{\sC}(\widetilde{c_1}-\Mtilde)=0$
(see \cite[Ch.~II, \textsection5, Corollary~2]{mumford}).
As~$\Spec(\sOint_S)$ is affine,
it follows that
$H^1(\sC,\sO_{\sC}(\widetilde{c_1}-\Mtilde))=0$
and hence that the natural sequence of $\sOint_S$\nobreakdash-modules
\begin{align}
\label{eq:apprr}
\xymatrix{
0 \ar[r]& H^0(\sC, \sO_\sC(\widetilde{c_1}-\Mtilde)) \ar[r] & H^0(\sC,\sO_\sC(\widetilde{c_1})) \ar[r]^(.53){r} & H^0\big(\Mtilde,\sO_\Mtilde\big) \ar[r] & 0
}
\end{align}
is exact.
These modules are free since they are finitely generated and torsion-free
over the principal ideal domain~$\sOint_S$ (recall
$\Pic(\sOint_S)=0$).
Choosing an $\sOint_S$\nobreakdash-linear section~$s$ of~$r$ and using strong approximation off~$v_0$
in the free $\sOint_S$\nobreakdash-module $H^0(\sC,\sO_\sC(\widetilde{c_1}-\Mtilde))$,
we can find
 $h_0 \in H^0(\sC \otimes_{\sOint_S} \sOint_{S\cup\{v_0\}},\sO_{\sC}(\widetilde{c_1}-\Mtilde))$
arbitrarily close to
$h_v-s(r(h_v))\in H^0(C \otimes_k k_v,\sO_C(c_1-M))$
for $v \in S \cup\Omega_\infty$.
Note that $\bigoplus_{m\in M} t_m \in H^0(\Mtilde,\sO_\Mtilde) \otimes_{\sOint_S} \sOint_{S\cup\{v_0\}}$
lies,
by construction,
 arbitrarily close to
$r(h_v) \in H^0(M \otimes_k k_v,\sO_M)$ for $v\in S \cup\Omega_\infty$.
Letting $h=h_0+s\mkern.5mu\big(\mkern-2.5mu\bigoplus_{m\in M}t_m\big)$,
we therefore obtain $h \in H^0(C,\sO_C(c_1))$
arbitrarily close to
$h_v\in H^0(C \otimes_k k_v,\sO_C(c_1))$
for $v \in S \cup\Omega_\infty$,
such that $h(m)=t_m$ for $m\in M$
and such that~$h$,
as a rational function on~$\sC$, has no pole outside~$\Supp(\widetilde{c_1}) \cup (\sC \otimes_{\sOint_S} \Fv{v_0})$.

Let $c \in \Div(C^0)$ be the effective divisor defined by $\div(h)=c-c_1$.
It remains to check the existence of a family $(z'_v)_{v \in \Omega} \in \Zcyc_{0,\A}(X^0)$ satisfying (1)--(4).

\subsection{Construction of \texorpdfstring{$(z'_v)_{v\in\Omega}$}{z'ᵥ for vϵΩ}}
\label{subsec:constructionzpv}

The next lemma takes care of~(3) and~(4).

\begin{lem}
There exists
$(z'_v)_{v \in S\cup\Omega_\infty} \in \prod_{v \in S\cup\Omega_\infty} \Sym_{X^0/k}(k_v)$
such that $(z'_v)_{v \in S} \in \sU$,
such that $f_*z'_v=c$ in $\Div(C^0\otimes_k k_v)$ for $v \in S \cup \Omega_\infty$
and such that~(4) holds.
\end{lem}

\begin{proof}
Note that for each $v \in S$ (resp.~for each real place~$v$ of~$k$),
the divisor~$c$ is arbitrarily close to~$f_*z_v$ (resp.~to~$f_*z_v^2$)
as a point of $\Sym_{C^0/k}(k_v)$.
As~$f_*z_v$ (resp.~$f_*z_v^2$) is a reduced divisor for any~$v \in S$ (resp.~for any real~$v$), the morphism $f_*:\Sym_{X^0/k} \to \Sym_{C^0/k}$
is smooth at~$z_v$ (resp.~at~$z_v^2$) for such~$v$.
Thus the existence of $(z'_v)_{v \in S \cup \Omega_\infty}$ satisfying the required conditions follows from the inverse function theorem.
\end{proof}

Let us turn to~(1) and~(2).
To construct~$z'_v$ for $v \in \Omega_f\setminus S$, we shall need to
consider the reduction of $f^{-1}(\Supp(c))\to\Supp(c)$ modulo~$v$.

For $v \in \Omega_f\setminus S$, the divisor $c\otimes_k k_v$ on~$C^0 \otimes_k k_v$
decomposes uniquely as
\begin{align*}
c \otimes_kk_v = c_{v,0}+\sum_{m\in M} \sum_{\smash[b]{\genfrac{}{}{0pt}{}{w \in \mtilde}{w|v}}} c_{w,m}
\end{align*}
where~$c_{v,0}$ and the~$c_{w,m}$ are effective divisors on~$C^0\otimes_kk_v$
satisfying the following properties: for any $m \in M$,
any $w \in \mtilde \otimes_{\sOint_S}\Fv{v}$
and any $y_v \in \Supp(c \otimes_kk_v)$,
the point~$y_v$ belongs to $\Supp(c_{w,m})$ if and only if its Zariski
closure
in $\sC\otimes_{\sOint_S}\sOint_v$ contains~$w$;
the closure of $\Supp(c_{v,0})$
in $\sC\otimes_{\sOint_S}\sOint_v$ is contained in $\sC^0 \otimes_{\sOint_S}\sOint_v$.
This is in fact a partition of the support of~$c \otimes_k k_v$.
The points of $\Supp(c_{v,0})$ are those above which the fiber of~$f$ has good reduction in~$\sX$,
while for each~$m$ and each~$w$,
the points of $\Supp(c_{w,m})$ are those above which the fiber of~$f$ has the same reduction modulo~$v$
as~$X_m$ modulo~$w$.

\begin{lem}
\label{lem:cmsplits}
Let $m \in M$ and $w \in \mtilde$ be a closed point.
If $w \in \Supp(\ctilde)$ then~$w$, as a place of~$k(m)$, splits completely in~$L_m$.
\end{lem}

\begin{proof}
If~$w$ divides~$v_0$, then~$w$ splits completely in~$L_m$ since~$v_0$ does.
Assume now that~$w$ does not divide~$v_0$.
Then the rational function~$h$ on~$\sC$ is regular
in a neighbourhood of~$w$, since $\mtilde \cap \Supp(\widetilde{c_1})=\emptyset$
and~$h$ has no pole outside~$\Supp(\widetilde{c_1}) \cup (\sC \otimes_{\sOint_S} \Fv{v_0})$.
As~$w \in \Supp(\ctilde)$, this regular function moreover vanishes at~$w$.
Thus, the function obtained by restricting~$h$
to~$\mtilde$ vanishes at~$w$.
In other words~$h(m)$ has positive valuation at the place~$w$.
It follows that~$w$ splits completely in~$L_m$,
since $h(m)=t_m$.
\end{proof}

According to Lemma~\ref{lem:cmsplits} and to the assumptions made at the end of~\textsection\ref{subsec:constrc1v0},
we can choose, for each $m \in M$, each $i \in I_m$ and each $w \in \mtilde \cap \Supp(\ctilde)$, a rational point of the fiber of~$\sE_{\mtilde,i}\to\mtilde$ above~$w$.
Let $\xi_{w,m,i}$ denote its image in~$\sY_{\mtilde,i}$.

As the fibers of $f:\sX \to \sC$ above the closed points of~$\sC^0$ are smooth and contain rational points,
we can lift such rational points by Hensel's lemma and find,
for each~$v \in \Omega_f \setminus S$,
an effective zero-cycle $z'_{v,0} \in \Zcyc_0(X^0\otimes_kk_v)$ such that $f_*z'_{v,0}=c_{v,0}$
in $\Div(C^0\otimes_k k_v)$.
By the next lemma, we can also find,
for each $v \in \Omega_f \setminus S$,
each $m \in M$, each $i\in I_m$ and each $w \in \mtilde \cap \Supp(\ctilde)$ which divides~$v$,
an effective zero-cycle $z'_{w,m,i} \in \Zcyc_0(X^0\otimes_kk_v)$ such that $f_*z'_{w,m,i}=e_{m,i}c_{w,m}$
and such that the Zariski closure of $\Supp(z'_{w,m,i})$ in $\sX \otimes_{\sOint_S}\sOint_v$
meets $\sX\otimes_{\sOint_S}\Fv{v}$ only at the point~$\xi_{w,m,i}$.

\begin{lem}
Let $v \in \Omega_f\setminus S$,
let $m \in M$, let $i \in I_m$, let $w \in \mtilde$ be a closed point which divides~$v$
and let~$\xi$ be a rational point of the fiber of $\sY_{\mtilde,i} \to \mtilde$
above~$w$.
Let $y_v \in C^0\otimes_kk_v$ be a closed point whose Zariski closure
in $\sC\otimes_{\sOint_S}\sOint_v$ contains~$w$.
There exists an effective zero-cycle $x_v \in \Zcyc_0(X^0\otimes_kk_v)$
such that $f_*x_v=e_{m,i}y_v$ in $\Div(C^0\otimes_kk_v)$
and such that the Zariski closure of $\Supp(x_v)$ in $\sX\otimes_{\sOint_S}\sOint_v$
meets $\sX\otimes_{\sOint_S}\Fv{v}$ only at the point~$\xi$.
\end{lem}

\begin{proof}
Let~$R$ denote the completion of~$\sO_{\sC\mkern-3mu,\mkern2muw}$.
Let $p \in \sOint_v$ be a uniformiser and $s \in R$ be a generator of the ideal of $\mtilde \times_{\sC} \Spec(R) \subset \Spec(R)$.
Let $U \subset \sX\times_\sC\Spec(R)$ be an affine open
neighbourhood of~$\xi$.  Letting $U=\Spec(A)$, we choose~$U$ small enough that
the ideal of~$A$ defined by the (reduced) closed subscheme $\sX_{\mtilde,i}\times_\sC\Spec(R) \subset U$ is principal.
Let $t \in A$ be a generator of this ideal.
As~$(\sX_{\mtilde})_\red$ is smooth over~$\mtilde$ at~$\xi$
and~$\mtilde$ is \'etale over~$\sOint_S$, the ring~$\sO_{U,\xi}/(p,t)$ is regular,
hence there exist $n \geq 0$ and $f_1,\dots,f_n \in \mathfrak{m}_{U,\xi}$
such that $(p,t,f_1,\dots,f_n)$ is a regular system of parameters
for~$\sO_{U,\xi}$
(see~\cite[Proposition~17.1.7]{ega41}).
The image of~$s$ in $\sO_{U,\xi}$ can be written
as $ut^{e_{m,i}}$ for some $u \in \sO_{U,\xi}^*$.
By shrinking~$U$, we may assume that $f_1,\dots,f_n\in A$ and $u \in A^*$.
On the other hand, as~$\mtilde$ is \'etale over~$\sOint_S$,
the sequence~$(p,s)$ is a regular system of parameters for~$R$;
hence
the fiber above the closed point
of~$\Spec(R)$ of the finite type $R$\nobreakdash-scheme $T=\Spec(A/(f_1,\dots,f_n))$
contains~$\xi$ as an isolated point, with multiplicity~$e_{m,i}$.
Thus we may assume, after further shrinking~$U$,
that~$T$ is a closed subscheme of $\sX\times_\sC\Spec(R)$, finite and
flat over~$R$
of degree~$e_{m,i}$
(see~\cite[Proposition~6.2.5]{ega2}, \cite[Proposition~15.1.21]{ega41}).
Pulling it back by the natural map~$\widetilde{y_v}\to \Spec(R)$,
where~$\widetilde{y_v}$ denotes the Zariski closure of~$y_v$ in $\sC \otimes_{\sOint_S}\sOint_v$,
yields a closed subscheme of~$\sX \times_{\sC} \widetilde{y_v}$,
finite and flat
of degree~$e_{m,i}$
over~$\widetilde{y_v}$,
whose special fiber
consists of the point~$\xi$ with multiplicity~$e_{m,i}$.
\end{proof}

For each $m \in M$, we choose integers $(f_{m,i})_{i\in I_m}$
such that $\sum_{i\in I_m}e_{m,i}f_{m,i}=1$
and we set~$z'_{w,m,i}=0$ for any $i \in I_m$ and any closed point $w \in \mtilde$
which does not belong to~$\Supp(\ctilde)$.
Finally, for each $v \in \Omega_f \setminus S$, we let
\begin{align*}
z'_v=z'_{v,0} + \sum_{m\in M} \sum_{i\in I_m}
\sum_{\genfrac{}{}{0pt}{}{w \in \mtilde}{w|v}} f_{m,i}z'_{w,m,i}\rlap{\text{.}}
\end{align*}
We have now defined a family $(z'_v)_{v \in \Omega} \in \prod_{v \in \Omega} \Zcyc_0(X^0\otimes_kk_v)$.
When~$X_m$ contains an irreducible component of multiplicity~$1$ for each~$m\in M$,
we can choose $f_{m,i} \in \{0,1\}$ for all~$m$ and all~$i$, so that~$z'_v$
is effective for all~$v$.
In any case, the family~$(z'_v)_{v\in\Omega}$ satisfies~(1); as a consequence, it belongs to $\Zcyc_{0,\A}(X^0)$.
Let us check that~(2) is also satisfied.

\begin{lem}
\label{lem:evalvanishes}
Let $v \in \Omega_f\setminus S$.
For any $\beta\in B$,
any $m\in M$, any $i\in I_m$ and any $w \in \mtilde$ which divides~$v$,
we have $\inv_v \mylangle \beta,z'_{v,0}\myrangle=0$
and $\inv_v \mylangle \beta,z'_{w,m,i}\myrangle=0$.
\end{lem}

\begin{proof}
As $B\subset\Br(\sX^0)$,
as
the Zariski closure of $\Supp(z'_{v,0})$ in $\sX\otimes_{\sOint_S}\sOint_v$ is contained in $\sX^0 \otimes_{\sOint_S}\sOint_v$
and as
the Brauer group of any finite $\sOint_v$\nobreakdash-algebra
vanishes (see~\cite[III, Proposition~1.5, Th\'eor\`eme~11.7~(2)]{grbr}),
we have $\mylangle \beta,z'_{v,0}\myrangle=0$.
Let us now fix~$m$, $i$ and~$w$ as in the statement.
If $w \notin \Supp(\ctilde)$, then $z'_{w,m,i}=0$. Otherwise, let~$R'$ denote the completion of $\sO_{\sX\mkern-3mu,\mkern2mu\xi_{w,m,i}}$
and set $V=\Spec(R')$
and $V^0=\sX^0\times_{\sX}V$.
The natural morphism $\Supp(z'_{w,m,i}) \to \sX^0$
factors through the projection $\sigma:V^0\to \sX^0$
since the Zariski closure of $\Supp(z'_{w,m,i})$ in $\sX\otimes_{\sOint_S}\sOint_v$
meets~$\sX\otimes_{\sOint_S}\Fv{v}$ only at~$\xi_{w,m,i}$.
To prove that $\mylangle \beta,z'_{w,m,i}\myrangle=0$ for $\beta\in B$, it therefore suffices
to check that~$B$ is contained in the kernel of $\sigma^*:\Br(\sX^0)\to\Br(V^0)$.
As the order~$n$ of~$B$ is invertible in~$\sOint_S$,
the class $\beta\in \Br(\sX^0)$ has a well-defined residue $\partial_\beta \in H^1(\sY_{\mtilde,i},\Z/n\Z)$, represented by a cyclic \'etale cover of~$\sY_{\mtilde,i}$
through which $\sE_{\mtilde,i}\to \sY_{\mtilde,i}$ factors.
Recall that the fiber of $\sE_{\mtilde,i}\to \sY_{\mtilde,i}$ above~$\xi_{w,m,i}$ possesses a rational point, by the definition of~$\xi_{w,m,i}$.
Hence $\partial_\beta(\xi_{w,m,i})=0$ in $H^1(\xi_{w,m,i},\Z/n\Z)$.
As $H^1(\sY_{\mtilde,i}\times_{\sX}V,\Z/n\Z)=H^1(\xi_{w,m,i},\Z/n\Z)$, it follows that
the residue of $\sigma^*\beta \in \Br(V^0)$ along $\sY_{\mtilde,i}\times_\sX V$ vanishes as well,
so that $\sigma^*\beta \in \Br(V) \subset \Br(V^0)$.
Finally, we have $\Br(V)=0$ (see~\emph{loc.\ cit.}), hence the lemma.
\end{proof}

Let $\beta\in B$.
We have
$\mylangle\beta,z'_v\myrangle=\mylangle\beta,z_v\myrangle$
for all $v \in S \cup \Omega_{\infty}$
in view of~(4) and of the fact that $(z'_v)_{v \in S} \in \sU$
(see~\textsection\ref{subsec:constrc1v0}).
On the other hand, we have $\mylangle \beta,z_v\myrangle=0$
and $\mylangle \beta,z'_v\myrangle=0$ for all $v \in
\Omega_f\setminus S$,
by the assumptions made in~\textsection\ref{subsec:constrc1v0} and
by Lemma~\ref{lem:evalvanishes} respectively.
As the family $(z_v)_{v \in\Omega}$ is orthogonal to~$B$ with respect to~\eqref{eq:globalpairing}
(see~\textsection\ref{subsec:reduction}),
we conclude that
$\sum_{v \in \Omega} \inv_v\mylangle\beta,z'_v\myrangle=0$, as desired.

\section{Hilbert subsets}
\label{sec:hilbertsubsets}

The following lemma was noted by Swinnerton-Dyer in the case where~$h$ is a rational point
(see~\cite[Proposition~6.1]{smeets}).

\begin{lem}
\label{lem:hilbertopen}
Let~$V$ be a normal quasi-projective variety over a number field~$k$.
Let~$H \subseteq V$ be a Hilbert subset.  Let $h \in H$ be a closed point.
For any finite subset $S \subset \Omega$,
there exist a finite subset~$T \subset \Omega$ disjoint from~$S$
and a neighbourhood $\sV \subset \prod_{v \in T} \Sym_{V/k}(k_v)$ of~$h$
such that any effective zero-cycle on~$V$
corresponding to a point of $\Sym_{V/k}(k) \cap \sV$
is in fact a closed point of~$V$ and belongs to~$H$.
\end{lem}

\begin{proof}
Let $V^0 \subseteq V$ and $W_i\to V^0$, for $i \in \{1,\dots,n\}$,
be the open subset and the irreducible \'etale covers defining the Hilbert subset~$H$.
Let $\pi:E\to U$ denote the universal family of reduced effective zero-cycles
of degree~$\deg(h)$ on~$V^0$
(thus~$U$ is an open subscheme of~$\Sym_{V^0/k}$ and~$E$ is a reduced closed subscheme
of~$U \times_k V^0$; the morphism~$\pi$ is finite \'etale of degree~$\deg(h)$).
Let $E_i = E\times_{V^0} W_i$
and let $\pi_i:E_i\to U$ denote the projection to~$E$ composed with~$\pi$.
By assumption, the fibers of~$\pi$, $\pi_1,\dots,\pi_n$ above the $k$\nobreakdash-point
of~$U$ corresponding to~$h$ are irreducible.
As~$U$ is normal and irreducible and the morphisms~$\pi$, $\pi_1,\dots,\pi_n$
are finite and \'etale, it follows that~$E$, $E_1,\dots,E_n$ are themselves
irreducible.
Applying \cite[Proposition~6.1]{smeets} (whose statement is identical
to Lemma~\ref{lem:hilbertopen} with the additional assumption that~$h$ is a rational point)
to the Hilbert subset of~$U$
defined by these $n+1$ irreducible \'etale covers of~$U$ now concludes the proof.
\end{proof}

Building on Lemma~\ref{lem:hilbertopen}, we now strengthen the conclusion
of Theorem~\ref{th:existenceresult}.

\begin{thm}
\label{th:existenceresulthilb}
Let us keep the notation and assumptions of Theorem~\ref{th:existenceresult}.
For any Hilbert subset $H \subseteq C$,
the conclusion of Theorem~\ref{th:existenceresult} still holds
if the effective divisor $c \in \Div(C^0)$ is required,
in addition, to be a closed point of~$C$ belonging to~$H$.
\end{thm}

\begin{proof}
As $\deg(y)\geq 2g$, the complete linear system~$|y|$ has no base point.
Composing the morphism to projective space defined by~$|y|$ with
projection from a general codimension~$2$ linear subspace
yields a finite map $\pi:C \to \P^1_k$
the linear equivalence class of whose fibers is~$y$.
According to Hilbert's irreducibility theorem applied to the irreducible covers of~$\P^1_k$ obtained by composing~$\pi$ with the covers of~$C$
which define~$H$,
we can find a rational point of~$\P^1_k$ whose inverse image by~$\pi$ is a closed point~$h$ of~$C^0$ belonging to~$H$.

By Lemma~\ref{lem:inclusionfiniteness}, there exists
a finite subgroup $B_0 \subset \Br(X)$ such that
\begin{align*}
(B+f_\eta^*\mkern.5mu\Br(\eta))\cap\Br(X) = B_0 + f^*\Br(C)\rlap{\text{.}}
\end{align*}
Let $S' \subset \Omega$ be a finite subset containing $S \cup \Omega_\infty$, large enough
that $X_h(k(h)_w)\neq\emptyset$ for any place~$w$ of~$k(h)$ which does not divide a place of~$S'$
(see~\cite{langweil}, \cite{nisnevic})
and large enough that $\mylangle \beta, u_v \myrangle=0$ for any $v \in \Omega\setminus S'$, any $\beta \in B_0$ and
any $u_v \in \Zcyc_0(X \otimes_k k_v)$.

Let $T \subset \Omega \setminus S'$ and $\sV \subset \prod_{v \in T} \Sym_{C/k}(k_v)$
be the finite subset and the neighbourhood of~$h$ given by Lemma~\ref{lem:hilbertopen}.
For each $v \in T$,
as $v \notin S'$,
 we can choose
an effective zero-cycle $z^1_v \in \Zcyc_0(X^0\otimes_kk_v)$ such that $f_*z^1_v=h$
in $\Div(C^0\otimes_k k_v)$.
For~$v \in \Omega\setminus T$, we let $z^1_v=z_v$.
As the class of~$h$ in~$\Pic(C)$ is~$y$, this defines a family~$(z^1_v)_{v \in \Omega} \in \Zcyc_{0,\A}^{\effred,y}(X^0)$.
This family is orthogonal to~$B_0$ with respect to the pairing~\eqref{eq:globalpairing} since
 $S' \cap T=\emptyset$
and~$(z_v)_{v \in \Omega}$ is orthogonal to~$B_0$.
As any element of $\Zcyc_{0,\A}^{\effred,y}(X^0)$ is orthogonal to $f^*\Br(C)$,
we deduce that $(z^1_v)_{v \in \Omega}$
is orthogonal to $(B+f_\eta^*\mkern.5mu\Br(\eta))\cap\Br(X)$.
We can therefore apply Theorem~\ref{th:existenceresult} to the family~$(z^1_v)_{v\in\Omega}$.
This yields $c \in \Div(C^0)$ and $(z'_v)_{v\in\Omega} \in \Zcyc_{0,\A}(X^0)$ satisfying~(1)--(4),
with~$c$ effective, and~$z'_v$ effective and arbitrarily close to~$z^1_v$
for $v \in T$.  By choosing~$z'_v$ close enough to~$z^1_v$ for $v \in T$,
we can ensure that $c\in\sV$, so that~$c$ must be a closed point belonging to~$H$.
\end{proof}

\section{From completed Chow groups to effective cycles}
\label{sec:reductions}

Theorem~\ref{th:existenceresult} takes, as input, a collection of local effective zero-cycles.
On the other hand, the statement of Conjecture~\ref{conj:cycles} involves elements
of~$\CHzAhat(X)$.
Theorem~\ref{th:hatsoff}
will allow us to bridge the gap between~$\CHzAhat(X)$ and~$\CHzA(X)$
and to reduce to effective~cycles.

\begin{thm}
\label{th:hatsoff}
Let~$C$ be a smooth, projective and geometrically irreducible curve over a number field~$k$.
Let~$X$ be a smooth, projective, irreducible variety over~$k$
and $f:X \to C$ be a morphism
whose geometric generic fiber~$X_\etabar$
is irreducible
and
satisfies $H^1(X_{\etabar},\Q/\Z)=0$ and $\Azero(X_\etabar \otimes K)=0$,
where~$K$ denotes an algebraic closure of the function field of~$X_\etabar$.
Fix a dense open subset $X^0 \subseteq X$,
a map $\Phi:\N^2\to \N$
and an element $\zhat_\A \in \CHzAhat(X)$
such that~$f_*\zhat_\A$ belongs to the image of
the natural map $\Pichat(C)\to\PicAhat(C)$.
There exist
classes $y\in \Pic(C)$, $\zhat\in\CHzhat(X)$, $z_\A^\eff=(z_v)_{v \in \Omega} \in \Zcyc_{0,\A}^{\effred,y}(X^0)$,
a finite subset $S \subset \Omega_f$
and a neighbourhood $\sU \subseteq \prod_{v \in S} \Sym_{X^0/k}(k_v)$
of $(z_v)_{v \in S}$
such that:
\begin{itemize}
\item $\zhat_\A=z_\A^\eff+\zhat$ in $\CHzAhat(X)$;
\smallskip
\item for all $v \in \Omega_\infty$,
the support of~$z_v$
is contained in a smooth, proper, geometrically irreducible curve $Z_v \subset X\otimes_kk_v$
which dominates~$C$ and
satisfies $\deg(y)\geq \Phi(g_v,d_v)$, where~$g_v$ denotes the genus of~$Z_v$ and $d_v=[k_v(Z_v):k_v(C)]$;
\smallskip
\item for any $z'_\A=(z'_v)_{v\in\Omega} \in \Zcyc_{0,\A}(X^0)$
such that the images of $f_*z'_\A$ and of~$y$ in~$\PicA(C)$ coincide
and such that properties~(3) and~(4) of Theorem~\ref{th:existenceresult} are satisfied, the equality
$z_\A^\eff=z'_\A$ holds in $\CHzA(X)$.
\end{itemize}
\end{thm}

We shall prove Theorem~\ref{th:hatsoff} in~\textsection\ref{subsec:proofhatsoff},
by combining
preliminary reduction steps for general fibrations established in~\textsection\ref{subsec:generalfibrations}
with the results of~\cite[\textsection2]{wittdmj}
on the kernel and cokernel of the natural map $f_*:\CH_0(X\otimes_kk_v)\to\Pic(C\otimes_kk_v)$
for fibrations whose generic fiber is subject to the hypotheses of Theorem~\ref{th:hatsoff}.
For ease of reference, we summarise these results in the next theorem.
They rely on a variant of the ``decomposition of the diagonal'' argument for families of varieties and on theorems
of Kato, Saito and Sato about Chow groups of zero-cycles over finite and local fields
(see \emph{loc.\ cit.}\ for more details).

\newcommand{\citesectiondwittdmj}{\cite[Th\'eor\`eme~2.1, Lemme~2.3, Lemme~2.4]{wittdmj}}
\begin{thm}[\citesectiondwittdmj]
\label{th:recollectionwittdmj}
Let $f:X \to Y$ be a proper dominant morphism between smooth, irreducible
varieties over a field~$k$.
\begin{enumerate}
\item[(i)] There exists an integer $n>0$ such that for any field extension~$k'/k$, the
cokernel of $f_*:\CH_0(X\otimes_k k') \to \CH_0(Y \otimes_kk')$ is killed by~$n$.
\smallskip
\item[(ii)] Assume that the geometric generic fiber~$X_\etabar$ of~$f$
is smooth and irreducible and
satisfies $\Azero(X_\etabar \otimes K)=0$,
where~$K$ denotes an algebraic closure of the function field of~$X_\etabar$.
Then there exists an integer $n>0$ such that for any field extension~$k'/k$,
the kernel of
$f_*:\CH_0(X\otimes_k k') \to \CH_0(Y \otimes_kk')$ is killed by~$n$.
\smallskip
\item[(iii)] Assume, in addition,
that
$H^1(X_{\etabar},\Q/\Z)=0$,
that~$Y$ is a proper curve and that~$k$ is a number field.
Then the map
$f_*:\CH_0(X\otimes_kk_v)\to \CH_0(Y\otimes_kk_v)$
is an isomorphism for all but finitely many places~$v$ of~$k$.
\end{enumerate}
\end{thm}

We shall also need a few results on abelian groups essentially contained in~\cite{wittdmj}.

\begin{lem}
\label{lem:abgroup}
Let $h:A\to B$ be a homomorphism of abelian groups whose cokernel has
finite exponent.  Letting $\widehat{h}:\widehat{A}\to\widehat{B}$ denote
the map induced by~$h$, we have:
\begin{enumerate}
\item[(i)] The natural map $\Coker(h) \to \Coker\mkern.5mu(\mkern.5mu\widehat{h})$ is injective.
\item[(ii)]  If the $N$\nobreakdash-torsion subgroup of~$B$ is finite for all $N>0$,
the natural map $$\widehat{\Ker(h)} \to \Ker\Big(\widehat{h}\Big)$$ is surjective
and the abelian group $\Coker\smash[t]{\left(\Ker(h) \to \Ker\Big(\widehat{h}\Big)\!\right)}$ is divisible.
\end{enumerate}
\end{lem}

\begin{proof}
The first assertion is \cite[Lemma~1.11]{wittdmj}.
The proof of \cite[Lemme~1.12]{wittdmj} is in fact a proof of the first half of~(ii).
The second half of~(ii) follows from the first half and from \cite[Lemme~1.10]{wittdmj}.
\end{proof}

\subsection{General reductions}
\label{subsec:generalfibrations}

For the whole of~\textsection\ref{subsec:generalfibrations},
we fix a smooth, projective, irreducible variety~$X$ over a number field~$k$,
a smooth, projective and geometrically irreducible curve~$C$ over~$k$ and
a morphism $f:X\to C$ with geometrically irreducible generic fiber.
We shall not make any other assumption on the generic fiber of~$f$
until~\textsection\ref{subsec:proofhatsoff}.
The goal of~\textsection\ref{subsec:generalfibrations} is to perform two general reduction steps towards
Theorem~\ref{th:hatsoff},
stated as
Proposition~\ref{prop:reductionnohat}
(from completed Chow groups to Chow groups)
and
Proposition~\ref{prop:reductioneffective}
(from cycles to effective
cycles).

\begin{prop}
\label{prop:reductionnohat}
Let $\zhat_\A \in \CHzAhat(X)$
be such that $f_*\zhat_\A \in \PicAhat(C)$ belongs to the image of
the natural map $\Pichat(C)\to\PicAhat(C)$.
For any $n>0$,
there exist
$z_\A \in \CHzA(X)$,
$\zhat \in \CHzhat(X)$,
$\xihat_\A\in\CHzAhat(X)$
such that
the equality
\begin{align}
\zhat_\A = z_\A + n(\zhat + \xihat_\A)
\end{align}
holds in $\CHzAhat(X)$ and such that
\begin{enumerate}
\item $f_*\xihat_\A=0$ in $\PicAhat(C)$;
\item $f_*z_\A \in \PicA(C)$ belongs to the image of the natural map $\Pic(C)\to\PicA(C)$.
\end{enumerate}
\end{prop}

\begin{proof}
Let $Z \subset X$ be a proper irreducible curve
dominating~$C$ and let $d=[k(Z):k(C)]$.
As a consequence of the projection formula $f_*(f^*y \cdot [Z])= y \cdot f_*[Z] = dy$
(see~\cite[Example~8.1.7]{fulton}),
the cokernels of the maps $\CHzhat(X) \to \Pichat(C)$
and $\CHzA(X) \to \PicA(C)$ induced by~$f$ are killed by~$d$.
On the other hand,
the natural map $\Pic(C)\to \Pichat(C)$ has a divisible cokernel (apply Lemma~\ref{lem:abgroup}~(ii)
with $A=\Pic(C)$ and $B=0$).
Thus, the cokernel of the map
\begin{align*}
\Pic(C) \times \CHzhat(X) \to \Pichat(C), \;\;(y,\zhat) \mapsto y+nf_*\zhat
\end{align*}
is at the same time divisible and killed by~$d$, hence it vanishes.
As~$f_*\zhat_\A$ belongs to the image
of~$\Pichat(C)$,
we conclude that there exist
$y_0 \in \Pic(C)$ and $\zhat \in \CHzhat(X)$
such that $f_*\zhat_\A=y_0+nf_*\zhat$ in $\PicAhat(C)$.

Let~$\phi$ denote the map $\CHzA(X) \to \PicA(C)$ induced by~$f$.
As $\Coker(\phi)$ is killed by~$d$,
the natural map $\Coker(\phi)\to\Coker(\widehat{\phi})$
is injective
(see Lemma~\ref{lem:abgroup}~(i)).
As the image of~$y_0$ in~$\Coker(\widehat{\phi})$ vanishes,
we deduce the existence
of $z_{0,\A} \in \CHzA(X)$ such that $f_*z_{0,\A}=y_0$ in $\PicA(C)$.
We then have
\begin{align}
\zhat_\A=z_{0,\A} + n\zhat + \xihat_{0,\A}
\end{align}
for some $\xihat_{0,\A} \in \Ker(\widehat{\phi})$.
As the groups $\Coker(N_{\kbar_v/k_v}:\Pic(C\otimes_k\kbar_v)\to\Pic(C\otimes_kk_v))$ for real places~$v$ of~$k$
and the $N$\nobreakdash-torsion subgroup of $\Pic(C\otimes_kk_v)$
for $v \in \Omega_f$ and~$N>0$
are all finite and the group
$\Coker(f_*:\CH_0(X \otimes_k k_v) \to \Pic(C \otimes_k k_v))$
is killed by~$d$ for all~$v \in \Omega$,
Lemma~\ref{lem:abgroup}~(ii) applied to the $v$\nobreakdash-adic component~$\phi_v$ of~$\phi$ for each~$v$
implies that
$\Coker(\Ker(\phi)\to\Ker(\widehat{\phi}))$ is divisible,
since $\widehat{\phi}=\prod_{v\in\Omega}\widehat{\phi_v}$.
Thus
there exist $\xihat_\A \in \Ker(\widehat{\phi})$ and $z_{1,\A} \in \Ker(\phi)$
such that
$\xihat_{0,\A}=n\xihat_\A+z_{1,\A}$.
Letting $z_\A=z_{0,\A}+z_{1,\A}$, we now have $\zhat_\A=z_\A+n(\zhat+\xihat_\A)$
and $f_*z_\A=y_0$.
\end{proof}

\begin{prop}
\label{prop:reductioneffective}
Let $z_\A \in \CHzA(X)$ be such that
$f_*z_\A \in \PicA(C)$ belongs to the image of the natural map $\Pic(C)\to\PicA(C)$.
For any dense open subset $X^0\subseteq X$, any finite subset $S \subset \Omega$,
any map~$\Phi:\N^2\to\N$ and any $n>0$,
there exist $y \in \Pic(C)$,
$z_\A^\eff = (z_v)_{v\in\Omega}\in \Zcyc_{0,\A}^{\effred,y}(X^0)$, $z\in \CH_0(X)$, $\xi_\A \in \CHzA(X)$ such that the equality
\begin{align}
\label{eq:zdecompbis}
z_\A = z_\A^\eff + z + \xi_\A
\end{align}
holds in $\CHzA(X)$ and such that
\begin{enumerate}
\item $f_*\xi_\A=0$ in $\PicA(C)$;
\item the $v$\nobreakdash-adic component of the family~$\xi_\A$ is zero for all $v\in S$;
\item for all $v \in S$,
the support of~$z_v$
is contained in a smooth, proper, geometrically irreducible curve $Z_v \subset X\otimes_kk_v$
which dominates~$C$ and
satisfies $\deg(y)\geq \Phi(g_v,d_v)$, where~$g_v$ denotes the genus of~$Z_v$ and $d_v=[k_v(Z_v):k_v(C)]$.
\end{enumerate}
\end{prop}

\begin{proof}
Let~$\sX^0$ be a model of~$X^0$ over~$\sOint_S$.
Let $S'\subset \Omega$ be the finite subset containing~$S\cup\Omega_\infty$ given by Lemma~\ref{lem:lifty}.
Let $P \in X^0$ be a closed point. For $v \in S'$, let $z^0_v \in \Zcyc_0(X\otimes_kk_v)$ be a representative
of the $v$\nobreakdash-adic component of~$z_\A$.
According to~\cite[Lemmes~3.1 et~3.2]{ctreglees}
(see also~\cite[Lemme~4.4]{wittdmj}),
there exist an integer~$n_0$ and,
for each $v \in S'$,
a smooth, proper and geometrically irreducible curve $Z_v \subset X\otimes_kk_v$
which dominates~$C$ and contains both~$P$ and the support of~$z^0_v$,
such that for any $n \geq n_0$ and any~$v \in S'$, there exists an effective divisor $z_v \in \Div(Z_v)$
satisfying the following condition:
\begin{enumerate}
\item[($\ast$)] $z_v$ is an effective divisor linearly equivalent to $z^0_v + nP$ on~$Z_v$,
its support is contained
in $X^0\otimes_kk_v$ and~$f_*z_v$ is a reduced divisor on~$C \otimes_kk_v$.
\end{enumerate}
We fix $y_0 \in \Pic(C)$ such that $f_*z_\A=y_0$
and choose an integer~$n \geq n_0$
such that
\begin{align}
\deg(y_0)+n\deg(P)>2g\text{,}\mkern10mu
\deg(y_0)+n\deg(P)\geq\Phi(g_v,d_v)
\end{align}
for all $v \in S'$,
where
$d_v=[k_v(Z_v):k_v(C)]$
and where~$g$ and~$g_v$ denote the genera of~$C$ and of~$Z_v$ respectively.
Let $y \in \Pic(C)$ be the class of~$y_0+nf_*P$.
For each $v \in S'$,
let us fix a divisor~$z_v$ on~$Z_v$ satisfying~($\ast$).
For each $v\in \Omega\setminus S'$, we fix an effective $z_v\in \Zcyc_0(X^0\otimes_kk_v)$
such that $f_*z_v$ is reduced, such that the class of~$f_*z_v$ in $\Pic(C\otimes_kk_v)$
is equal to the image of~$y$ and such that the Zariski closure of~$\Supp(z_v)$
in $\sX^0\otimes_{\sOint_S}\sOint_v$ is finite over~$\sOint_v$
(see the conclusion of Lemma~\ref{lem:lifty}).
The family $z_\A^\eff=(z_v)_{v \in \Omega} \in \prod_{v \in \Omega}\Zcyc_0(X^0\otimes_kk_v)$
then belongs to $\Zcyc_{0,\A}^{\effred,y}(X^0)$ and
the desired equality~\eqref{eq:zdecompbis} holds
if we let $z=-nP$
and $\xi_\A=(z^0_v-z_v+nP)_{v\in\Omega}\in\CHzA(X)$.
\end{proof}

\subsection{Proof of Theorem~\ref{th:hatsoff}}
\label{subsec:proofhatsoff}

Theorem~\ref{th:hatsoff} will be proved by combining Proposition~\ref{prop:reductionnohat},
Proposition~\ref{prop:reductioneffective} and Theorem~\ref{th:recollectionwittdmj}.
We henceforth assume, as in the statement of Theorem~\ref{th:hatsoff},
that $H^1(X_{\etabar},\Q/\Z)=0$ and $\Azero(X_\etabar \otimes K)=0$.
According to Theorem~\ref{th:recollectionwittdmj},
there exists a finite subset $S \subset \Omega_f$
such that the map
\begin{align}
\label{eq:fsloc}
f_*:\CH_0(X\otimes_kk_v)\to\Pic(C\otimes_kk_v)
\end{align}
is an isomorphism
for all $v \in \Omega \setminus (S\cup \Omega_\infty)$.

\begin{lem}
\label{lem:injectivityhatchzA}
There exists an integer $n>0$ such that the natural map
\begin{align}
\label{eq:mapchzArel}
\Ker\left(\CHzA(X)\xrightarrow{f_*}\PicA(C)\right) \to \CHzA(X)/n\CHzA(X)
\end{align}
is injective.
Moreover, the kernel of $f_*:\CHzAhat(X)\to\PicAhat(C)$
has finite exponent.
\end{lem}

\begin{proof}
For $v \in S$,
the kernel of~\eqref{eq:fsloc}
and the torsion
subgroup of $\Pic(C\otimes_kk_v)$ have finite exponent,
by Theorem~\ref{th:recollectionwittdmj}~(ii) and~\cite{mattuck} respectively.
Let~$N>0$ be an even integer which kills these two groups for all $v \in S$.
The map~\eqref{eq:mapchzArel}
is then injective for $n=N^2$.
The second assertion follows from the definition of~$S$
and from Lemma~\ref{lem:abgroup}~(ii)
applied to the map~\eqref{eq:fsloc} for~$v \in S$.
\end{proof}

Using Lemma~\ref{lem:injectivityhatchzA}, we choose an integer~$n>0$
such that~\eqref{eq:mapchzArel} is injective
and such that $\Ker(f_*:\CHzAhat(X)\to\PicAhat(C))$ is killed by~$n$.
Let us apply Proposition~\ref{prop:reductionnohat} to~$\zhat_\A$ and~$n$
and Proposition~\ref{prop:reductioneffective} to the resulting class~$z_\A$ and to $X^0$, $S\cup\Omega_\infty$, $\Phi$, $n$.
This yields~$y$, $z_\A^\eff=(z_v)_{v\in\Omega}$, $z$, $\zhat$, $\xi_\A$ and~$\smash[t]{\xihat_\A}$ satisfying the conclusions
of Propositions~\ref{prop:reductionnohat} and~\ref{prop:reductioneffective}.
By~properties~(1) and~(2) of Proposition~\ref{prop:reductioneffective}
and by the definition of~$S$, we have $\xi_\A=0$.
By property~(1) of Proposition~\ref{prop:reductionnohat}
and by the definition of~$n$,
we also have~$n\xihat_\A=0$.
We thus obtain the decomposition
$\zhat_\A=z_\A^\eff+(z+n\zhat)$
and it only remains to check that the classes of~$z_\A^\eff$ and~$z'_\A$
in~$\CHzA(X)$ are equal for any~$z'_\A$ as in the
statement of Theorem~\ref{th:hatsoff}, if the neighbourhood~$\sU$ is chosen small enough.
At the places $v\notin S\cup\Omega_\infty$, this holds by the definition of~$S$.
At the places $v \in S$,
this follows from \cite[Lemme~1.8]{wittdmj},
in view of the injectivity of~\eqref{eq:mapchzArel}.
Finally, for $v\in\Omega_\infty$, condition~(4) of Theorem~\ref{th:existenceresult}
implies that the class of $z_v-z_v'$ in $\CH_0(X \otimes_k k_v)$ is a norm from $\CH_0(X\otimes_k\kbar_v)$
(see~\cite[Proposition~3.2~(ii)]{ctischebeck} or~\cite[Lemme~1.8]{wittdmj}).

\section{Main theorems on zero-cycles}
\label{sec:maintheorems}

We now have at our disposal all of the required tools to prove that if~$X$
is a smooth, proper, irreducible variety over a number field~$k$
and if $f:X\to\P^1_k$ is a dominant morphism with rationally connected geometric
generic fiber whose smooth fibers satisfy~$\cE$,
then~$X$ satisfies~$\cE$ (Theorem~\ref{th:introcycles}).  As discussed in~\textsection\ref{sec:intro}, the results we obtain refine this statement in several ways:
$\P^1_k$ is replaced with a curve of arbitrary genus satisfying~$\cE$,
the hypothesis on the smooth fibers is restricted to the fibers over a Hilbert subset, etc.
Our main theorem on zero-cycles is Theorem~\ref{th:main} below.
We give its precise statement and compare its corollaries with the existing literature in~\textsection\ref{subsec:statements}
before proceeding to the proofs in~\textsection\textsection\ref{subsec:preliminaries}--\ref{subsec:proofcorollary}.

\subsection{Statements}
\label{subsec:statements}

To state Theorem~\ref{th:main}, we shall need to consider the following condition, which depends on the choice of a smooth and proper variety~$V$ over a number field~$k$.

\begin{cond}
\label{cond:intermediate}
The image of $\CHzhat(V)$ in $\CHzAhat(V)$ contains the image
of~$V(\A_k)^{\Br(V)}$ in the same group.
\end{cond}

Condition~\ref{cond:intermediate} clearly holds if~$V$ satisfies~$\cE$.
The following lemma records another interesting condition which implies it.

\begin{lem}
\label{lem:conditionintermediate}
Condition~\ref{cond:intermediate} holds
if~$V(k)$ is dense in $V(\A_k)^{\Br(V)}$ and~$V$ is rationally connected.
\end{lem}

\begin{proof}
For any $n\geq 1$ and any finite set~$S$ of places of~$k$,
the image
of the natural map
\begin{align*}
\CH_0(V)
\longrightarrow
\prod_{v\in S} \CH_0(V\otimes_kk_v)/n\CH_0(V\otimes_kk_v)
\end{align*}
contains
the image of $\smash[t]{V(\A_k)^{\Br(V)}}$
as~$V(k)$ is dense in $V(\A_k)^{\Br(V)}$
(see~\cite[Lemme~1.8]{wittdmj}).
On the other hand,
as~$V$ is rationally connected, the group
$\Azero(V\otimes_kk_v)$ vanishes for all but finitely many places~$v$
and has finite exponent for every~$v$ (see~\cite{kollarszabo} or
Theorem~\ref{th:recollectionwittdmj}).
The lemma follows from these two facts (see~\cite[Remarques~1.1~(ii)]{wittdmj}).
\end{proof}

We are now in a position to state the main result on zero-cycles obtained
in this paper.

\begin{thm}
\label{th:main}
Let~$C$ be a smooth, proper and geometrically irreducible curve over a number field~$k$.
Let~$X$ be a smooth, proper and irreducible variety over~$k$.
Let $f:X \to C$
be a morphism
whose geometric generic fiber~$X_\etabar$ is irreducible
and satisfies
$H^1(X_\etabar,\Q/\Z)=0$
and $\Azero(X_\etabar\otimes K)=0$,
where~$K$ denotes an algebraic closure
of the function field of~$X_\etabar$.

\begin{enumerate}
\item If~$C$ satisfies~$\cE$
and the smooth fibers of~$f$ above the closed points of a Hilbert subset of~$C$ satisfy~$\cE$, then~$X$ satisfies~$\cE$.
\smallskip
\item If the smooth fibers of~$f$ above the closed points of a Hilbert subset of~$C$ satisfy~$\cE$,
the complex
\begin{align*}
\xymatrix{
\CH_0(X/C) \ar[r] & \CHzA(X/C) \ar[r] & \Hom(\Br(X)/f^*\Br(C),\Q/\Z)
}
\end{align*}
is an exact sequence, where $\CH_0(X/C)$ and~$\CHzA(X/C)$ denote the kernels
of $f_*:\CH_0(X)\to\CH_0(C)$ and of $f_*:\CHzA(X)\to \CHzA(C)$ respectively.
\smallskip
\item If every fiber of~$f$ possesses an irreducible component of multiplicity~$1$,
assertions~(1) and~(2) above remain valid when the smooth fibers of~$f$ above the closed points of a Hilbert subset of~$C$ are
only assumed to satisfy Condition~\ref{cond:intermediate} instead of~$\cE$.
\smallskip
\item If there exists a family $(z_v)_{v \in \Omega} \in \Zcyc_{0,\A}(X)$
orthogonal to~$\Br(X)$ with respect to~\eqref{eq:globalpairing},
such that $\deg(z_v)=1$ for all $v \in \Omega$
and such that the class of $(f_*z_v)_{v \in \Omega}$
in~$\PicAhat(C)$ belongs to the image of~$\Pichat(C)$,
and if the smooth fibers of~$f$ above the closed points of a Hilbert subset of~$C$ satisfy~$\cEu$,
then~$X$ possesses a zero-cycle of degree~$1$.
\end{enumerate}
\end{thm}

We recall that~$C$ satisfies~$\cE$
if $C=\P^1_k$ or more generally
if the divisible subgroup of the Tate--Shafarevich group of the Jacobian of~$C$ is trivial,
by a theorem of Saito~\cite{saito} (see~\cite[Remarques~1.1~(iv)]{wittdmj}).
In this case,
the condition, in~(4),
that the class of $(f_*z_v)_{v\in\Omega}$ in~$\PicAhat(C)$ should belong to the image of~$\Pichat(C)$ automatically holds.

An easy induction argument will allow us to deduce a version of Theorem~\ref{th:main}
over some simple higher-dimensional bases.

\begin{cor}
\label{cor:zerocycleshigherdimbase}
Let~$X$ be a smooth, proper, irreducible variety over a number field~$k$.
Let~$Y$ be an irreducible variety over~$k$,
birationally equivalent to either~$\P^n_k$, or~$C$, or $\P^n_k\times C$, for some $n\geq 1$ and some
smooth, proper, geometrically irreducible curve~$C$ over~$k$
which satisfies~$\cE$.
Let $f:X \to Y$ be a morphism
whose geometric generic fiber~$X_\etabar$ is irreducible and satisfies
$H^1(X_\etabar,\Q/\Z)=0$
and $\Azero(X_\etabar\otimes K)=0$,
where~$K$ denotes an algebraic closure
of the function field of~$X_\etabar$.
\begin{enumerate}
\item If the smooth fibers of~$f$ above the closed points of a Hilbert subset of~$Y$ satisfy~$\cE$,
then~$X$ satisfies~$\cE$.
\smallskip
\item If the smooth fibers of~$f$ above the closed points of a Hilbert subset of~$Y$ satisfy Condition~\ref{cond:intermediate} and the fibers of~$f$ above the codimension~$1$
points of~$Y$
possess an irreducible component of multiplicity~$1$, then~$X$ satisfies~$\cE$.
\end{enumerate}
\end{cor}

By the generalised Bloch conjecture,
the vanishing of $\Azero(X_\etabar\otimes K)$ should be equivalent
to that of $H^i(X_\etabar,\sO_{X_\etabar})$ for all~$i>0$.
We refer the reader to~\cite[\textsection4]{voisinlectures}
for a discussion of this condition and for some nontrivial examples
of simply connected surfaces of general type which satisfy it
(Barlow surfaces).

When the geometric generic fiber of~$f$ is rationally connected,
both of the assumptions $H^1(X_\etabar,\Q/\Z)=0$ and $\Azero(X_\etabar\otimes K)=0$
hold
(see~\cite[Corollary~4.18~(b)]{debarrehigherdim} for the first).
In particular, Theorem~\ref{th:introcycles} follows from
Theorem~\ref{th:main}~(1).
Moreover, in this case, every fiber of~$f$ above
a codimension~$1$ point of~$Y$ contains an irreducible component of multiplicity~$1$
(see~\cite{ghs}).
Combining Corollary~\ref{cor:zerocycleshigherdimbase}
with Lemma~\ref{lem:conditionintermediate}
and with Borovoi's theorem on the rational points of homogeneous spaces of linear groups~\cite[Corollary~2.5]{borovoi},
we thus obtain the following corollary.

\begin{cor}
\label{cor:cycleshomspace}
Let~$X$ be a smooth, proper, irreducible variety over a number field~$k$.
Let~$Y$ be an irreducible variety over~$k$,
birationally equivalent to either~$\P^n_k$, or~$C$, or $\P^n_k\times C$, for some $n\geq 1$ and some
smooth, proper, geometrically irreducible curve~$C$ over~$k$
which satisfies~$\cE$.
Let $f:X \to Y$ be a dominant morphism whose generic fiber is birationally
equivalent to a homogeneous space of a connected linear algebraic group, with
connected geometric stabilisers.
Then~$X$ satisfies~$\cE$.
\end{cor}

Corollary~\ref{cor:cycleshomspace} applies to arbitrary families of toric varieties.
Even this particular case of the corollary suffices to cover all of the examples
of varieties satisfying~$\cE$ or~$\cEu$
dealt with in
\cite[\textsection3]{lianglocalglobal},
\cite{weioneq}, \cite{caoliang}, \cite{liangtowards}.
Corollary~\ref{cor:cycleshomspace}
also applies to fibrations into Ch\^atelet surfaces over~$Y$, or, more generally, into Ch\^atelet $p$\nobreakdash-folds
in the sense of~\cite{varillyviraychatelet},
since the total space of such a fibration is birationally equivalent to the total space of a fibration
into torsors under tori over~$Y \times_k \P^1_k$.
Such fibrations were considered in~\cite{pooneninsufficiency}.
This answers the question raised at the very end of \cite[Remarque~3.3]{ctsolides}.
(The case of fibrations into Ch\^atelet surfaces
all of whose fibers are geometrically integral was previously
treated in~\cite{liangcourbe}; see also \cite[Proposition~7.1]{liangastuce}.)
In the situation of a fibration into Ch\^atelet surfaces over a curve of positive genus,
Theorem~\ref{th:main}~(4) immediately implies the unconditional result~\cite[Th\'eor\`eme~3.1]{ctsolides}.
A similar application of Corollary~\ref{cor:cycleshomspace} is the validity of~$\cE$
for the total space of any fibration
into del Pezzo surfaces of degree~$6$ over~$Y$ (recall that degree~$6$ del Pezzo surfaces are toric~\cite{maninrational}).

By applying Corollary~\ref{cor:cycleshomspace} to the trivial
fibration $V \times \P^1_k\to\P^1_k$
(a trick first used in~\cite{liangarithmetic}), we recover Liang's theorem
(proved in~\emph{op.\ cit.})
that
varieties~$V$ birationally equivalent to a homogeneous space
of a connected linear algebraic group, with connected geometric stabilisers,
satisfy~$\cE$.

We note, finally, that
Theorem~\ref{th:main} and Corollary~\ref{cor:zerocycleshigherdimbase} also subsume
\cite[\textsection\textsection5--6]{ctsd94},
\cite[\textsection4]{ctsksd98},
\cite{ctreglees},
\cite{frossard},
\cite{vanhamel},
\cite{wittdmj}, 
\cite{liangcourbe},
\cite{liangprojectif},
\cite[\textsection5]{smeets},
\cite{liangastuce}.
In these papers, one can find versions of Corollary~\ref{cor:zerocycleshigherdimbase} in which every fiber of~$f$ is assumed
to contain an irreducible component of multiplicity~$1$ split by an abelian extension of the field over which the fiber is defined,
and in which the Brauer group of the generic fiber is assumed to be generated by classes which
are unramified along the singular fibers of~$f$.  Both of these assumptions turn out to be
quite restrictive for applications.

\subsection{Preliminaries}
\label{subsec:preliminaries}

Let us first recall some well-known implications between the geometric conditions appearing
in the statements of Theorem~\ref{th:main}, Theorem~\ref{th:existenceresulthilb},
Theorem~\ref{th:existenceresult} and of Proposition~\ref{prop:specialisation}.

\begin{lem}
\label{lem:implications}
The properties
\begin{itemize}
\item[(a)]
$X_\etabar$ is rationally connected;
\item[(b)]
$\Azero(X_\etabar\otimes K)=0$, where~$K$ denotes an algebraic closure
of the function field of~$X_\etabar$;
\item[(c)]
$H^i(X_\eta,\sO_{X_\eta})=0$ for all $i>0$;
\item[(d)]
$H^2(X_\eta,\sO_{X_\eta})=0$;
\item[(e)]
if $H^1(X_\etabar,\Q/\Z)=0$,
the quotient $\Br(X_\eta)/f_\eta^*\mkern.5mu\Br(\eta)$ is finite;
\item[(f)]
the gcd of the multiplicities of the irreducible components of each fiber of~$f$ is~$1$
\end{itemize}
satisfy the implications $(\mathrm{a})\Rightarrow(\mathrm{b})\Rightarrow(\mathrm{c})\Rightarrow(\mathrm{d})\Rightarrow(\mathrm{e})$ and $(\mathrm{c})\Rightarrow(\mathrm{f})$.
\end{lem}

\begin{proof}
The implication $(\mathrm{b})\Rightarrow(\mathrm{c})$ may be checked
by decomposing the diagonal in Hodge theory, see~\cite[p.~141]{chambertloir}
(see also Lemma~\ref{lem:decompositiondiagonal} below).
See \cite[Lemma~1.3~(i)]{ctskogoodreduction}
for the implication~$(\mathrm{d})\Rightarrow(\mathrm{e})$
and~\cite[Proposition~7.3~(iii)]{ctvoisin} or~\cite[Proposition~2.4]{elw}
for~$(\mathrm{c})\Rightarrow(\mathrm{f})$.
\end{proof}

The next proposition, which combines the contents of~\textsection\textsection2--7, is the core result
from which the four assertions of Theorem~\ref{th:main} will be deduced.

\begin{prop}
\label{prop:commonstatement}
Let us keep the notation and assumptions of Theorem~\ref{th:main}.
Fix a Hilbert subset $H \subseteq C$
and an element $\zhat_\A \in \CHzAhat(X)$
orthogonal to~$\Br(X)$
with respect to~\eqref{eq:globalpairinghat},
such that~$f_*\zhat_\A$ belongs to the image of
the natural map $\Pichat(C)\to\PicAhat(C)$.
There exist an element $\zhat\in\CHzhat(X)$, a closed point~$c$ of~$H$
and a family $z'_\A \in \Zcyc_{0,\A}(X_c)$
orthogonal to~$\Br(X_c)$ with respect to the pairing~\eqref{eq:globalpairing}
relative to~$X_c$, such that~$z'_\A$, when considered as an element of~$\Zcyc_{0,\A}(X)$,
satisfies $f_*z'_\A=c$ in $\Zcyc_{0,\A}(C)$
and $\zhat_\A=z'_\A+\zhat$ in $\CHzAhat(X)$.
Furthermore, if every fiber of~$f$ possesses an irreducible component
of multiplicity~$1$, one can ensure that~$z'_\A$ is a family of effective zero-cycles.
\end{prop}

\begin{proof}
By Chow's lemma and Hironaka's theorem, there exist an irreducible, smooth projective variety~$X'$ and a birational morphism $\pi:X'\to X$.
The induced morphism $\pi_c:X'_c\to X_c$
is a birational equivalence for all but finitely many $c\in C$.
As the groups~$H^1(V,\Q/\Z)$, $\Br(V)$, $\CH_0(V)$ and~$\Azero(V)$ are birational invariants
of smooth and proper varieties~$V$ over a field of characteristic~$0$ (see
\cite[Corollaire~3.4]{sga1expX},
\cite[III, \textsection7]{grbr},
\cite[Example~16.1.11]{fulton})
and as the existence of an irreducible component of multiplicity~$1$ in every fiber of~$f$ is equivalent to
the existence of an irreducible component of multiplicity~$1$ in every fiber of~$f\circ\pi$
(see~\cite[Corollary~1.2]{skorodescent} and~\cite{nishimura}),
we may replace~$X$ with~$X'$ 
to prove Proposition~\ref{prop:commonstatement}.
Thus, from now on, we assume that~$X$ is projective,
which allows us to use Theorem~\ref{th:hatsoff}.

Thanks to Lemma~\ref{lem:implications},
we can choose a finite subgroup
$B \subset \Br(X_\eta)$ which surjects onto
the quotient $\Br(X_\eta)/f_\eta^*\mkern.5mu\Br(\eta)$.
Let~$C^0 \subseteq C$ be a dense open subset above which~$f$ is smooth,
small enough that $B \subset \Br(X^0)$, where $X^0=f^{-1}(C^0)$.
By Proposition~\ref{prop:specialisation},
there exists a
Hilbert subset $H' \subseteq C^0$ such that
the natural map
\begin{align}
\label{eq:natmapspec}
B \to \Br(X_c)/f_c^*\mkern.5mu\Br(c)
\end{align}
is surjective for all $c \in H'$.
We define a map $\Phi:\N^2\to \N$ by $\Phi(g,d)=\deg(M)d+2g+2$, where $M=C\setminus C^0$,
and apply Theorem~\ref{th:hatsoff}.
This yields classes
$y\in \Pic(C)$, $\zhat \in \CHzhat(X)$, $z_\A^\eff=(z_v)_{v\in\Omega}\in \Zcyc_{0,\A}^{\effred,y}(X^0)$
and subsets
$S\subset \Omega_f$ and $\sU\subseteq \prod_{v\in S}\Sym_{X^0/k}(k_v)$
satisfying the conclusion of Theorem~\ref{th:hatsoff}.
Note that $z_\A^\eff$ is orthogonal to~$\Br(X)$ with respect to~\eqref{eq:globalpairing}
since $\zhat_\A=z_\A^\eff+\zhat$ in~$\CHzAhat(X)$
and~$\zhat_\A$ and~$\zhat$ are both orthogonal to~$\Br(X)$.
We can therefore apply Theorem~\ref{th:existenceresulthilb}
and obtain a closed point~$c$ of $H\cap H'$ and a family $z'_\A =(z'_v)_{v\in\Omega}\in \Zcyc_{0,\A}(X^0)$
satisfying conditions~(1)--(4) of Theorem~\ref{th:existenceresult}.
By condition~(1), we may regard~$z'_\A$ as an element of $\Zcyc_{0,\A}(X_c)$.
By condition~(2) and the surjectivity of~\eqref{eq:natmapspec},
this element is orthogonal to~$\Br(X_c)$ with respect to the pairing~\eqref{eq:globalpairing}.
By conditions~(3) and~(4) and the conclusion of Theorem~\ref{th:hatsoff},
we have $\zhat_\A=z'_\A+\zhat$ in~$\CHzAhat(X)$.
\end{proof}

\subsection{Proof of Theorem~\ref{th:main}}

\subsubsection{Assertion~(1)}
\label{proof:assertion1}

Let $\zhat_\A \in \CHzAhat(X)$ be orthogonal to~$\Br(X)$ with
respect to~\eqref{eq:globalpairinghat}.
By functoriality, its image $f_*\zhat_\A \in \PicAhat(C)$ is orthogonal to~$\Br(C)$.
As~$C$ satisfies~$\cE$, it follows that~$f_*\zhat_\A$ belongs to the image
of the natural map $\Pichat(C)\to\PicAhat(C)$.
We may therefore apply Proposition~\ref{prop:commonstatement}
to~$\zhat_\A$ and to a Hilbert subset $H \subseteq C$ such that
for every closed point~$c$ of~$H$, the fiber $X_c=f^{-1}(c)$ is smooth
and satisfies~$\cE$.
The class in~$\CHzAhat(X_c)$ of the resulting family~$z'_\A$ is then the image
of an element of~$\CHzhat(X_c)$.  As a consequence, its image in~$\CHzAhat(X)$ is the image
of an element of~$\CHzhat(X)$.

\subsubsection{Assertion~(2)}
\label{proof:assertion2}

By the Mordell--Weil theorem, the group $\Pic(C)$ is finitey generated.
Moreover, by Theorem~\ref{th:recollectionwittdmj},
the kernel and the cokernel of the map $f_*:\CH_0(X)\to\Pic(C)$ have finite exponent.
By Lemma~\ref{lem:abgroup}~(ii),
we deduce that the natural complex
\begin{align}
\label{se:relativech}
\xymatrix{
\CH_0(X/C) \ar[r] & \CHzhat(X) \ar[r] & \Pichat(C)
}
\end{align}
is exact.
Let $z_\A \in \CHzA(X/C)$ be orthogonal to $\Br(X)/f^*\Br(C)$
with respect to~\eqref{eq:globalpairing}.
The argument of~\textsection\ref{proof:assertion1} shows that the image
of~$z_\A$ in~$\CHzAhat(X)$ comes from~$\CHzhat(X)$.
As the natural maps $\Pichat(C) \to \PicAhat(C)$
and $\CHzA(X/C)\to \CHzAhat(X)$
are injective
(see~\cite[Remarques~1.1~(v)]{wittdmj} and Lemma~\ref{lem:injectivityhatchzA})
and the complex~\eqref{se:relativech} is
exact,
we conclude that~$z_\A$ is the image of an element of~$\CH_0(X/C)$.

\subsubsection{Assertion~(3)}

According to the last sentence of Proposition~\ref{prop:commonstatement},
one can ensure,
in~\textsection\ref{proof:assertion1} and~\textsection\ref{proof:assertion2}, 
that~$z'_\A$ is a family
of effective zero-cycles.  As $f_*z_\A'=c$ in $\Zcyc_{0,\A}(C)$,
it follows that $z_\A'\in X_c(\A_{k(c)})$; hence Condition~\ref{cond:intermediate}
can be applied.

\subsubsection{Assertion~(4)}

By assumption there exists a Hilbert subset $H \subseteq C$ such that
for every closed point~$c$ of~$H$, the fiber~$X_c$ is smooth
and satisfies~$\cEu$.
Applying Proposition~\ref{prop:commonstatement} to $z_\A=(z_v)_{v \in\Omega}$,
we find a closed point~$c$ of~$H$,
a family $z'_\A \in \Zcyc_{0,\A}(X_c)$
orthogonal to~$\Br(X_c)$
and a class $\zhat \in \CHzhat(X)$
such that $f_*z'_\A=c$ in~$\Zcyc_{0,\A}(C)$ and $z_\A=z'_\A+\zhat$ in~$\CHzAhat(X)$.
Considering~$z'_\A$ as a collection of local zero-cycles of degree~$1$
on the variety~$X_c$ over~$k(c)$,
we deduce from~$\cEu$ that~$X_c$ possesses a zero-cycle~$z'$ of degree~$1$ over~$k(c)$.
Regarding~$z'$ as a zero-cycle on~$X$,
we have $\deg(z'+\zhat\mkern.7mu)=1$ in~$\Zhat$,
hence the degree map $\CHzhat(X)\to\Zhat$ is surjective.
By Lemma~\ref{lem:abgroup}~(i),
the degree map $\CH_0(X)\to\Z$ is surjective as well.

\subsection{Proof of Corollary~\ref{cor:zerocycleshigherdimbase}}
\label{subsec:proofcorollary}

When~$Y$ is a curve,
this is Theorem~\ref{th:main}~(1) and~(3).
In general, we argue by induction on the dimension of~$Y$.
As the hypotheses and the conclusion of Corollary~\ref{cor:zerocycleshigherdimbase} are invariant under birational equivalence over~$k$
(see \cite[Remarques~1.1~(vi)]{wittdmj}), we may assume, when $\dim(Y)>1$, that $Y=\P^n_k\times C$
for some $n\geq 1$ (setting $C=\P^1_k$ in the first of the three cases considered
in the statement).
Let $g:X\to C$ denote the composition of~$f$ with the second projection.
For $c \in C$, let $X_c=g^{-1}(c)$.
For the sake of brevity, we introduce the following definition.

\begin{defn}
We shall say that
a proper variety~$V$ over an algebraically closed field \emph{satisfies~$(\ast)$}
if it is integral,
if $H^1(V,\Q/\Z)=0$ and if $\Azero(V \otimes K)=0$ for an algebraic closure~$K$ of the function field of~$V$.
\end{defn}

\begin{lem}
\label{lem:geomgenfiberg}
The geometric generic fiber of~$g$ satisfies~$(\ast)$.
\end{lem}

\begin{proof}
Let~$\xibar$ denote the geometric generic point of~$C$
and $f_\xibar:X_\xibar\to \P^n_{k(\xibar)}$ the morphism induced by~$f$.
The geometric generic fiber~$X_\xibar$ of~$g$ is smooth and irreducible.
The geometric generic fiber~$V$ of~$f_\xibar$ satisfies~$(\ast)$
since it is isomorphic to the geometric generic fiber of~$f$.

Letting~$K$ denote an algebraic closure of the function field of~$X_\xibar$,
we deduce from Theorem~\ref{th:recollectionwittdmj}~(ii) applied to~$f_\xibar$ that the abelian group $\Azero(X_\xibar \otimes K)$ has
finite exponent.  It is also divisible, as~$K$ is
algebraically closed
(see~\cite[Lemma~1.3]{blochlectures}); hence $\Azero(X_\xibar\otimes K)=0$.

It remains to check
that $H^1(X_\xibar,\Lambda)=0$ for any finite abelian group~$\Lambda$.
The morphism~$f_\xibar$ has no multiple fiber in codimension~$1$
since~$V$ satisfies~$(\ast)$
(see Lemma~\ref{lem:implications}, (b)$\Rightarrow$(f)).
As a consequence,
the fundamental group of~$V$ surjects onto the fundamental group
of the geometric fiber of~$f_\xibar$ above any codimension~$1$ point of~$\P^n_{k(\xibar)}$
(see~\cite[Proposition~6.3.5]{raynaudspecialisation}).
The constructible
sheaf $R^1(f_{\xibar})_*\Lambda$
must therefore be
supported on a closed subset $F \subset \P^n_{k(\xibar)}$ of codimension~$\geq 2$.
Letting $U= \P^n_{k(\xibar)}\setminus F$,
we have
 $H^1(U, (f_{\xibar})_*\Lambda)=H^1(U,\Lambda)=0$
since $(f_{\xibar})_*\Lambda=\Lambda$
(see~\cite[Proposition~15.5.3]{ega43}).
The Leray spectral sequence for~$f_\xibar$ now shows that
$H^1(f_\xibar^{-1}(U),\Lambda)=0$.
As $H^1(X_\xibar,\Lambda)\subseteq H^1(f_\xibar^{-1}(U),\Lambda)$,
the proof is complete.
\end{proof}

\begin{lem}
\label{lem:propertyast}
There exists a dense open subset $C^0 \subseteq C$ such that for all $c \in C^0$,
the geometric generic fiber
of the morphism $f_c:X_c\to \P^n_{k(c)}$
induced by~$f$ satisfies~$(\ast)$.
\end{lem}

\begin{proof}
It suffices to prove the existence of a dense open subset $Y^0\subseteq Y$ such that for any geometric point~$y$ of~$Y^0$,
the variety $V=f^{-1}(y)$ satisfies~$(\ast)$.

As~$f$ is proper, the restriction of the \'etale sheaf $R^1f_*\Q/\Z$
to any open subset of~$Y$ over which~$f$ is smooth
is a direct limit of locally constant sheaves (see \cite[Ch.~VI, Corollary~4.2]{milneet}).
Its geometric generic stalk vanishes;
thus~$H^1(V,\Q/\Z)$ vanishes
for any geometric fiber~$V$ of~$f$ along which~$f$ is smooth.

To check that the condition $\Azero(V\otimes K)=0$ for the generic fiber of~$f$ implies the same condition
for the fibers of~$f$ over a dense open subset of~$Y$,
we use an argument known as decomposition of the diagonal (see \cite[Appendix to Lecture~1]{blochlectures}),
summarised in the next lemma.

\begin{lem}
\label{lem:decompositiondiagonal}
Let~$V$ be a geometrically integral variety of dimension~$d$ over an arbitrary field.
Let $v \in V$ be a closed point.
Let~$K$ be an algebraic closure of the function field of~$V$.
Let $[\Delta]\in \CH_d(V\times V)$ denote the class of the diagonal.
The group $\Azero(V \otimes K)$ vanishes if and only if for some $N\geq 1$,
one has a decomposition
\begin{align}
\label{eq:decompositiondiagonal}
N[\Delta]=[\Gamma_1]+[\Gamma_2]
\end{align}
in $\CH_d(V\times V)$ for
some cycle~$\Gamma_1$ supported on $V \times \{v\}$
and
some cycle~$\Gamma_2$ supported on~$D \times V$
for a closed subset $D \subset V$ of codimension~$\geq 1$.
\end{lem}

\begin{proof}
Letting~$N[\Delta]$ act on $\Azero(V \otimes K)$ as a correspondence,
we see that $\Azero(V \otimes K)$ is killed by~$N$ if the diagonal can be decomposed as above.
On the other hand, the group $\Azero(V\otimes K)$ is divisible
since~$K$ is
algebraically closed (see~\cite[Lemma~1.3]{blochlectures}).
This proves a half of the desired equivalence.
The other half is established in \cite[Proposition~1]{blochsrinivas}.
\end{proof}

Any rational equivalence which gives rise to the decomposition~\eqref{eq:decompositiondiagonal} on the generic fiber of~$f$
extends and specialises to the fibers of~$f$ over a dense open subset of~$Y$.
Thus, applying Lemma~\ref{lem:decompositiondiagonal} twice
concludes the proof of Lemma~\ref{lem:propertyast}.
\end{proof}

We can now complete the proof of Corollary~\ref{cor:zerocycleshigherdimbase}.
Let $C^0 \subseteq C$ be a dense open subset satisfying the conclusion of
Lemma~\ref{lem:propertyast}, small enough that~$X_c$ is smooth and irreducible for all $c \in C^0$
and that if the codimension~$1$ fibers of~$f$ contain an irreducible component of multiplicity~$1$,
then so do the codimension~$1$ fibers of
$f_c:X_c\to \P^n_{k(c)}$ for $c \in C^0$.

According to the next lemma, which we state with independent notation, there exists a Hilbert subset $H \subseteq C^0$
such that for every closed point~$c$ of~$H$,
the smooth fibers of~$f_c$
above the closed points of a Hilbert subset of~$\P^n_{k(c)}$
satisfy~$\cE$ in case~(1), satisfy Condition~\ref{cond:intermediate} in case~(2).

\begin{lem}
\label{lem:hilbertproduct}
Let~$V_1$ and~$V_2$ be geometrically irreducible varieties over a number field~$k$.
Let $H \subseteq V_1 \times_k V_2$ be a Hilbert subset.
There exists a Hilbert subset $H_2 \subseteq V_2$ such that for every $h\in H_2$,
the set $H \cap (V_1 \times_k \{h\})$ contains a Hilbert subset
of $V_1\times_k \{h\}$ (where~$\{h\}$ denotes $\Spec(k(h))$).
\end{lem}

\begin{proof}
Let~$W_1,\dots,W_n$ denote irreducible \'etale covers of a dense open subset of $V_1 \times_k V_2$ which define~$H$.
By~\cite[Lemme~3.12]{wittlnm}, there exists a Hilbert subset $H_2 \subseteq V_2$
such that for any~$h \in H_2$, the scheme $W_i \times_{V_2} \{h\}$ is irreducible for all $i\in\{1,\dots,n\}$.
The projections $W_i \times_{V_2} \{h\} \to V_1 \times_k \{h\}$ then
define a Hilbert subset
of $V_1 \times_k \{h\}$ contained in~$H$.
\end{proof}

We deduce from the induction hypothesis applied to~$f_c$ that~$X_c$ satisfies~$\cE$
for every closed point~$c$ of~$H$.
In view of Lemma~\ref{lem:geomgenfiberg}, and noting that
the fibers of~$g$ contain an irreducible component of multiplicity~$1$
if the codimension~$1$ fibers of~$f$ do,
we may now apply Theorem~\ref{th:main}
to~$g$ and conclude that~$X$ satisfies~$\cE$.

\section{Rational points in fibrations}
\label{sec:rationalpoints}

The goal of this last part of the paper
is to explore the counterpart,
in the context of rational points,
of the ideas which enter into the proof of the main theorems of~\textsection\ref{sec:maintheorems}.
The main unconditional result here is Theorem~\ref{th:consequenceofmatthiesen},
stated as Theorem~\ref{th:matthintro} in the introduction.

\subsection{A conjecture on locally split values of polynomials over number fields}
\label{subsec:conj}

We propose the following conjectural
variant of
Lemma~\ref{lem:strongapprox}.
Its statement parallels the consequence of Schinzel's hypothesis~($\mathrm{H}$)
first formulated by
Serre in~\cite{serrecollege}
and referred to as Hypothesis~$(\mathrm{H}_1)$
in~\cite{ctsd94}.
Lemma~\ref{lem:strongapprox} corresponds to
the case $n=1$, $\deg(P_1)=1$, $b_1=1$ of Conjecture~\ref{conj:mainstrong},
the only difference between the two statements being that the element~$t$ output by Lemma~\ref{lem:strongapprox}
is required to be integral outside $S\cup\{v_0\}$, while no integrality condition on~$t_0$
appears in Conjecture~\ref{conj:mainstrong} (see also Remark~\ref{rk:geomcrit}).
In fact, the absence of any integrality condition
makes it more accurate to compare
Conjecture~\ref{conj:mainstrong} 
with the homogeneous variant~$(\mathrm{HH}_1)$
of Schinzel's hypothesis,
first considered by Swinnerton-Dyer~\cite[\textsection2]{sdsomeapplications},
than with~$(\mathrm{H}_1)$ itself
(see especially \cite[Proposition~2.1]{hsw}).
We refer the reader to~\textsection\ref{subsubsec:relationschinzel} for more
details on~$(\mathrm{HH}_1)$.

\begin{conj}
\label{conj:mainstrong}
Let~$k$ be a number field.
Let $n \geq 1$ be an integer
and $P_1,\dots,P_n \in k[t]$ denote pairwise distinct irreducible monic polynomials.
Let $k_i=k[t]/(P_i(t))$ and let $a_i \in k_i$ denote the class of~$t$.
For each $i \in \{1,\dots,n\}$,
suppose given a finite extension~$L_i/k_i$
and an element $b_i\in k_i^*$.
Let~$S$ be a finite set of places of~$k$
containing the real places of~$k$ and the finite places above which,
for some~$i$,
either~$b_i$
is not a unit or $L_i/k_i$ is ramified.
Finally, for each $v \in S$,
fix an element $t_v \in k_v$.
Assume that
for every~$i \in \{1,\dots,n\}$ and every~$v \in S$,
there exists $x_{i,v} \in (L_i \otimes_kk_v)^*$
such that
\begin{align*}
b_i(t_v-a_i) = N_{L_i\otimes_k k_v/k_i\otimes_kk_v}(x_{i,v})
\end{align*}
in $k_i\otimes_k k_v$.
Then there exists $t_0 \in k$ satisfying the following conditions:
\begin{enumerate}
\item $t_0$ is arbitrarily close to~$t_v$ for $v \in S$;
\item for every~$i \in \{1,\dots,n\}$ and every finite place~$w$ of~$k_i$
such that $w(t_0-a_i)>0$,
either~$w$ lies above a place of~$S$ or
the field~$L_i$ possesses a place of degree~$1$ over~$w$.
\end{enumerate}
\end{conj}

Conjecture~\ref{conj:mainstrong} will be applied
in~\textsection\ref{subsec:consequencesrationalpoints} to the study of weak
approximation on the total space of a fibration over~$\P^1_k$.  When one is
only interested in the existence of a rational point or in weak
approximation ``up to connected components at the archimedean places'' (as
considered in~\cite[\textsection2, p.~351]{stoll}), the following weaker
variant of Conjecture~\ref{conj:mainstrong} suffices.

\begin{conj}
\label{conj:mainweak}
Same as Conjecture~\ref{conj:mainstrong}, except that~{\rm{(1)}} is weakened to:
\begin{enumerate}
\item[(1')] $t_0$ is arbitrarily close to~$t_v$ for finite $v \in S$;
in addition, for
every real place~$v$ of~$k$,
if~$i$ denotes an element of $\{1,\dots,n\}$
and~$w$ denotes a real place of~$k_i$ dividing~$v$,
the sign of $(t_0-a_i)(t_v-a_i)$ at~$w$
is independent of~$w$ and of~$i$.
\end{enumerate}
\end{conj}

\begin{rmks}
\label{rk:conj}

(i)
Suppose all of the fields~$L_i$ are isomorphic, over~$k$, to the same Galois extension~$L/k$.
This condition can be assumed to hold at no cost
in all of the applications
that we shall consider in this paper.
Assertion~(2) of Conjecture~\ref{conj:mainstrong} may then be reformulated
as follows: for every $i \in \{1,\dots,n\}$, the value $P_i(t_0) \in k$ has nonpositive valuation
at all places $v\notin S$, except perhaps at places~$v$ which split completely in~$L$.
Thus, Conjecture~\ref{conj:mainstrong} may be regarded as a conjecture about
locally split values of polynomials.

(ii)
In order to prove Conjectures~\ref{conj:mainstrong} and~\ref{conj:mainweak}, one may assume, without loss of generality, that $t_v=0$ for all $v \in S$
and~$a_i \neq 0$ for all~$i\in\{1,\dots,n\}$.
Indeed, one can choose $\tau \in k$
arbitrarily close to~$t_v$ for $v \in S$ and perform
the change of variables $t \mapsto t+\tau$.
Thanks to this remark, we see that Conjectures~\ref{conj:mainstrong} and~\ref{conj:mainweak} for~$k=\Q$ imply the same conjectures for arbitrary~$k$ (compare~\cite[Proposition~4.1]{ctsd94}).

(iii)
In order to prove Conjectures~\ref{conj:mainstrong} and~\ref{conj:mainweak},
one is free to adjoin any finite number of places to~$S$.
To see this, it is enough to check that for any finite~$v \notin S$,
there exists $t_v \in k_v$ such that~$b_i(t_v-a_i)$ is a norm from~$(L_i\otimes_k k_v)^*$ for every~$i$
and such that $w(t_v-a_i)\leq 0$ for every~$i$ and every place~$w$ of~$k_i$ dividing~$v$.
As~$b_i$ is a unit above~$v$ and~$L_i/k_i$ is unramified above~$v$, the first condition is satisfied
as soon as~$w(t_v-a_i)$ is a multiple of~$[L_i:k_i]$ for every place~$w$ of~$k_i$ dividing~$v$.
Thus, for finite~$v \notin S$, any $t_v\in k_v$ with valuation~$-N\prod_{i=1}^n [L_i:k_i]$
will satisfy the desired properties if~$N$ is a large enough positive integer.

(iv)
The reader interested in a more geometric (albeit less elementary) formulation
of Conjecture~\ref{conj:mainstrong}
will find one in Proposition~\ref{prop:geomcrit} below.
\end{rmks}

\subsection{Known cases and relation with Schinzel's hypothesis}
\label{subsec:knowncases}

In this paragraph, we collect some results which lend support to Conjectures~\ref{conj:mainstrong} and~\ref{conj:mainweak}
(see Theorem~\ref{th:schinzelconjmain}, Theorem~\ref{th:algebraiccases},
Theorem~\ref{thm:matthiesen},
Theorem~\ref{th:irving}).

\subsubsection{Relation with Schinzel's hypothesis}
\label{subsubsec:relationschinzel}

We prove that when the extensions~$L_i/k_i$ are abelian,
Conjecture~\ref{conj:mainstrong} is a consequence of Schinzel's
hypothesis~$(\mathrm{H})$.  Extensive numerical evidence for~$(\mathrm{H})$ can be found
in~\cite{batemanhornmathcomp}, \cite{batemanhornpspm}.
Exploiting a trick first used by
Colliot-Th\'el\`ene in the context of cubic extensions
(see~\cite[Theorem~3.5]{weioneq}, \cite[Theorem~4.6]{hsw}),
we deal, slightly more
generally, with extensions which are ``almost abelian'' in the
following sense.

\begin{defn}
\label{def:almostabelian}
Let~$k$ be a number field, with algebraic closure~$\kbar$.
A finite extension~$L/k$ is \emph{almost abelian} if it is abelian or
if there exist a prime number~$p$ and
a bijection $\Spec(L \otimes_k \kbar)\simeq \Fp$
such that the natural action of~$\Gal(\kbar/k)$
on $\Spec(L\otimes_k \kbar)$ gives rise,
by transport of structure, to an action on~$\Fp$
by affine transformations, \emph{i.e.}, transformations of the form
$x\mapsto ax+b$ for $a\in \Fp^*$, $b\in \Fp$.
\end{defn}

\begin{examples}
Cubic extensions are almost abelian.  For any $c \in k$ and any prime~$p$,
the extension $k(c^{1/p})/k$ is almost abelian.
\end{examples}

\begin{thm}
\label{th:schinzelconjmain}
Assume Schinzel's hypothesis~$(\mathrm{H})$.
If the extensions $L_1/k_1, \dots, L_n/k_n$ are almost abelian,
the statement of Conjecture~\ref{conj:mainstrong} holds true.
\end{thm}

In the proof of Theorem~\ref{th:schinzelconjmain}, we use~$(\mathrm{H})$ \emph{via} the following homogeneous variant
of Schinzel's hypothesis, which makes sense over an arbitrary number field.

\begin{hypHH}
Let~$k$ be a number field, $n\geq 1$ be an integer and $P_1,\dots,P_n\in k[\lambda,\mu]$ be irreducible homogeneous polynomials.
Let~$S$ be a finite set of places of~$k$ containing the archimedean places,
large enough that
for any $v \notin S$, there exists $(\lambda_v,\mu_v) \in \sOint_v\times \sOint_v$ such that $v(P_1(\lambda_v,\mu_v))=\dots=v(P_n(\lambda_v,\mu_v))=0$.
Suppose given
 $(\lambda_v,\mu_v)\in k_v \times k_v$
for $v \in S$,
with $(\lambda_v,\mu_v)\neq (0,0)$ when~$v$ is archimedean.
There exists $(\lambda_0,\mu_0)\in k \times k$ such that
\begin{enumerate}
\item $(\lambda_0,\mu_0)$ is arbitrarily close to $(\lambda_v,\mu_v) \in k_v \times k_v$ for each finite place $v \in S$;
\item $[\lambda_0:\mu_0]$ is arbitrarily close to $[\lambda_v:\mu_v] \in \P^1(k_v)$ for each archimedean place $v \in S$;
\item $\lambda_0\lambda_v + \mu_0\mu_v>0$ for each real place $v \in S$;
\item $\lambda_0$ and $\mu_0$ are integral outside~$S$;
\item for each $i\in\{1,\dots,n\}$, the element $P_i(\lambda_0,\mu_0)\in k$ is a unit outside~$S$ except at one place, at which it is a uniformiser.
\end{enumerate}
\end{hypHH}

We refer to~\cite[Lemma~7.1]{swdtopics}
for a proof that Schinzel's hypothesis~$(\mathrm{H})$
implies~$(\mathrm{HH}_1)$
(noting that~$\gamma$, in the proof of \emph{loc.\ cit.}, can be chosen totally positive,
so that~(3) holds).

\begin{rmk}
From the work of Heath-Brown and Moroz~\cite{hbm} on primes represented by binary cubic forms,
it is easy to deduce the validity of~$(\mathrm{HH}_1)$ when $k=\Q$, $n=1$ and $\deg(P_1)=3$.
Thus Conjecture~\ref{conj:mainstrong} holds when $n=1$, $k=\Q$, $[k_1:\Q]=3$ and~$L_1/k_1$ is almost abelian.
\end{rmk}

\begin{proof}[Proof of Theorem~\ref{th:schinzelconjmain}]
Let $P_i(\lambda,\mu)=\mu^{\deg(P_i)}P_i(\lambda/\mu)$.
By Remark~\ref{rk:conj}~(iii), we may assume that~$S$ is large enough for~$(\mathrm{HH}_1)$ to be applicable and that~$a_i$ is integral outside~$S$ for all~$i$.
Let $I = \{i \in \{1,\dots,n\}; L_i/k_i \text{ is abelian}\}$ and $J = \{1,\dots,n\}\setminus I$.  Thus $J=\emptyset$
if the extensions $L_1/k_1, \dots, L_n/k_n$ are all abelian;
the reader interested in this case only may simply skip the next lemma and the
paragraph which follows it.

\begin{lem}
For every $j \in J$, there exists a nontrivial cyclic subextension $E_j/k_j$
of the Galois closure of~$L_j/k_j$
such that for any place~$w$ of~$k_j$
which is unramified in~$L_j$,
either~$w$ splits completely in~$E_j$ or~$L_j$ possesses a place of degree~$1$ over~$w$.
\end{lem}

\begin{proof}
As $L_j/k_j$ is almost abelian but is not abelian, its degree is a prime
number~$p$
and
the Galois group~$G_j$ of its Galois closure $L'_j/k_j$ embeds into the group
$G=\Fp\rtimes \Fp^*$ of affine transformations of~$\Fp$.
Let~$E_j$ be the subfield of~$L_j'$ fixed by $H_j=G_j \cap \Fp$.
As~$G_j$ is not abelian, it is not contained in~$\Fp$, hence its image in~$\Fp^*$
is nontrivial.  Thus~$G_j$ is an extension of a nontrivial subgroup of~$\Fp^*$ by~$H_j$.
As a consequence, the extension $E_j/k_j$ is cyclic and nontrivial.
Let~$w$ be a place of~$k_j$ which is unramified in~$L_j$
and let $F \in G_j$ denote the Frobenius automorphism at~$w$ (well-defined up to conjugacy).
If the image of~$F$ in $\Fp^*$ is trivial, then~$w$ splits completely in~$E_j$.
Otherwise,
the action of~$F$ on $\Spec(L_j \otimes_{k_j} L'_j)$ admits a fixed point
(any affine transformation of~$\Fp$ which is not a translation admits a fixed point).
Such a fixed point corresponds to a place of~$L_j$ of degree~$1$ over~$w$.
\end{proof}

By Chebotarev's density theorem, for every $j \in J$, there exist infinitely many places~$w$ of~$k_j$
of degree~$1$ over~$k$ which are inert in~$E_j$
(recall that the set of places of~$k_j$ of degree~$1$ over~$k$ has Dirichlet density~$1$;
see, \emph{e.g.}, \cite[Ch.~VIII, \textsection4]{langant}).
If~$v$ denotes the trace on~$k$ of any such place~$w$,
there exists $t_v \in k_v$ satisfying the following properties:
\begin{enumerate}
\item[(i)] $t_v-a_j$ is a uniformiser at~$w$ and is a unit at all places of~$k_j$ above~$v$ different from~$w$;
\item[(ii)]$t_v-a_i$ is a unit at all places of~$k_i$ above~$v$,
for all $i \in \{1,\dots,n\}\setminus\{j\}$.
\end{enumerate}
By adjoining to~$S$ finitely many places~$v$ as above
and choosing~$t_v$ subject to~(i) and~(ii) at these places,
we may thus assume that
\begin{align}
\sum_{v \in S} \sum_{w|v} \inv_w (E_j/k_j, t_v-a_j) \neq 0
\end{align}
for all $j \in J$, where the second sum ranges over the places~$w$ of~$k_j$ dividing~$v$
and the symbol $(E_j/k_j,t_v-a_j)$ denotes a cyclic algebra over~$(k_j)_w$.

For $v \in S$,
let $\lambda_v=t_v$ and $\mu_v=1$.
Let $(\lambda_0,\mu_0)\in k\times k$ be given by~$(\mathrm{HH}_1)$.
By choosing~$(\lambda_0,\mu_0)$ close enough to~$(\lambda_v,\mu_v)$ at the finite places $v \in S$
and~$[\lambda_0:\mu_0]$ close enough to $[\lambda_v:\mu_v] \in \P^1(k_v)$
at the real places~$v \in S$,
we may assume,
in view of the fact that $\lambda_0\lambda_v+\mu_0\mu_v>0$ for real~$v$,
 that
$b_i(\lambda_0-a_i\mu_0) \in (k_i \otimes_kk_v)^*$ is a norm from $(L_i\otimes_kk_v)^*$ for all $i \in \{1,\dots,n\}$ and all $v\in S$
and that
\begin{align}
\label{eq:sumEj}
\sum_{v \in S} \sum_{w|v} \inv_w (E_j/k_j, \lambda_0-a_j\mu_0) \neq 0
\end{align}
for all $j \in J$.
For $i \in \{1,\dots,n\}$, let~$v_i$ denote the unique finite place of~$k$ outside~$S$ at which $P_i(\lambda_0,\mu_0)$ is not a unit.
As~$\lambda_0$, $\mu_0$ and~$a_i$ are integral outside~$S$,
as $v_i(P_i(\lambda_0,\mu_0))=1$
and as $P_i(\lambda_0,\mu_0)=N_{k_i/k}(\lambda_0-a_i\mu_0)$, there exists a unique place~$w_i$ of~$k_i$ above~$v_i$
such that $w_i(\lambda_0-a_i\mu_0)>0$.  Moreover, we have $w_i(\lambda_0-a_i\mu_0)=1$
and $w(\lambda_0-a_i\mu_0)=0$ for any place~$w$ of~$k_i$ which does not lie above~$S\cup\{v_i\}$.

Let $t_0=\lambda_0/\mu_0$.
If~$w$ is a place of~$k_i$ which does not lie
above~$S$ and if $w(t_0-a_i)>0$,
then $w(\lambda_0-a_i\mu_0)>0$, since $w(\mu_0)\geq 0$;
hence $w=w_i$.
Therefore it only remains to be shown that~$L_i$ possesses a place of degree~$1$
over~$w_i$ for every $i \in \{1,\dots,n\}$.

Suppose first $i \in I$.
As~$b_i$ is a unit outside~$S$
and~$L_i/k_i$ is unramified outside~$S$, the previous paragraphs imply that $b_i(\lambda_0-a_i\mu_0) \in (k_i)_w^*$ is a
norm
from $(L_i \otimes_{k_i} (k_i)_w)^*$ for every place~$w$ of~$k_i$ different from~$w_i$.
As $L_i/k_i$ is abelian, it follows,
by the reciprocity law of global class field theory, that $b_i(\lambda_0-a_i\mu_0)$ is a local
norm from $(L_i \otimes_{k_i} (k_i)_{w_i})^*$ as well
(see~\cite[Ch.~VII, \textsection3, Cor.~1, p.~51]{artintate}).
As $w_i(b_i(\lambda_0-a_i\mu_0))=1$ and $L_i/k_i$ is Galois, we conclude that~$w_i$ must split completely in~$L_i$.

Suppose now $i\in J$.  The reciprocity law of global class field theory,
together with~\eqref{eq:sumEj},
 implies that $\inv_{w_i} (E_i/k_i,\lambda_0-a_i\mu_0)\neq 0$.
It follows that~$w_i$ does not split completely
in~$E_i$, so that~$L_i$ must possess a place of degree~$1$ over~$w_i$.
\end{proof}

\subsubsection{A geometric criterion and cases of small degree}
\label{subsubsec:geomcrit}

Fix~$k$, $n$, and $k_i/k$, $L_i/k_i$, $a_i$, $b_i$, for~$i \in \{1,\dots,n\}$, as in the statement of Conjecture~\ref{conj:mainstrong}.
(We do not fix~$S$ or the~$t_v$'s.)
To this data we associate an irreducible quasi-affine variety~$W$ over~$k$ endowed with a smooth
morphism $p:W\to\P^1_k$ with geometrically irreducible generic fiber, as follows.

For $i\in\{1,\dots,n\}$, let~$F_i$ denote the singular locus of
$R_{L_i/k}(\A^1_{L_i}) \setminus R_{L_i/k}(\mathbf{G}_{\mathrm{m},{L_i}})$.
This is a codimension~$2$ closed subset of the affine space $R_{L_i/k}(\A^1_{L_i})$.
Let~$W$ denote the fiber, above~$0$, of the morphism
\begin{align*}
\left(\A^2_k \setminus \{(0,0)\}\right) \times \prod_{i=1}^n \left(R_{L_i/k}(\A^1_{L_i}) \setminus F_i\right) \to \prod_{i=1}^n R_{k_i/k}(\A^1_{k_i})
\end{align*}
defined by $(\lambda,\mu,x_1,\dots,x_n)\mapsto (b_i(\lambda-a_i\mu)-N_{L_i/k_i}(x_i))_{1\leq i\leq n}$, where~$\lambda,\mu$
are the coordinates of~$\A^2_k$
and~$x_i$ denotes a point of $R_{L_i/k}(\A^1_{L_i})$.
Finally, denote by $p:W\to \P^1_k$ the composition of projection to the first factor with the natural map $\A^2_k \setminus \{(0,0)\} \to \P^1_k$.

Varieties closely related to~$W$ (in fact, partial compactifications of~$W$) first appeared in the context of descent
in~\cite[\textsection3.3]{skorodescent}, see also~\cite[p.~391]{ctskodescent}, \cite[\textsection4.4]{skobook}.
We now state an equivalent, more geometric form of Conjecture~\ref{conj:mainstrong}, in terms
of~$W$.
We shall use it as a criterion to prove a few cases of Conjecture~\ref{conj:mainstrong}
in Theorem~\ref{th:algebraiccases} below.

\begin{prop}
\label{prop:geomcrit}
For $c \in \P^1(k)$, let $W_c=p^{-1}(c)$.
The statement of Conjecture~\ref{conj:mainstrong} holds true for all~$S$ and all $(t_v)_{v \in S}$ satisfying its hypotheses
if and only if $\mkern2mu\bigcup_{c\in \P^1(k)} W_{c}(\A_k)$
is dense in $W(\A_k)$.
\end{prop}

\begin{proof}
Let us first verify that the second condition is sufficient for Conjecture~\ref{conj:mainstrong} to hold.
Let~$S$ and $(t_v)_{v \in S}$ be as in the statement of Conjecture~\ref{conj:mainstrong}.
We assume that $\Omega_\infty\subseteq S$.
For all but finitely many finite places~$v$ of~$k$,
there exists $t_v \in k_v$ such that $w(t_v-a_i)=0$ for every $i\in\{1,\dots,n\}$ and every place~$w$
of~$k_i$ dividing~$v$.
By Remark~\ref{rk:conj}~(iii),
we may assume, after enlarging~$S$, that such~$t_v$'s exist
for $v\notin S$; let us fix them.
For each $i \in \{1,\dots,n\}$ and each $v \in \Omega$, let us then
fix $x_{i,v} \in (L_i \otimes_k k_v)^*$
such that $b_i(t_v-a_i)=N_{L_i\otimes_kk_v/k_i\otimes_kk_v}(x_{i,v})$.
When~$v \notin S$,
we may ensure that~$x_{i,v}$ is a unit at every place of~$L_i$ dividing~$v$.

Let $Q_v \in W(k_v)$
denote
the point with coordinates $(t_v,1,x_{1,v},\dots,x_{n,v})$.
We have defined
an adelic point $(Q_v)_{v \in \Omega} \in W(\A_k)$.
By assumption, there exist $c \in \A^1(k)$ and an adelic point $(Q'_v)_{v\in\Omega}\in W_c(\A_k)$ arbitrarily close to~$(Q_v)_{v\in\Omega}$.
Let $(\lambda'_v,\mu'_v,x'_{1,v},\dots,x'_{n,v})$ denote the coordinates of~$Q'_v$.
By choosing $(Q'_v)_{v \in \Omega}$ close enough to~$(Q_v)_{v \in \Omega}$
for the adelic topology, we may assume
that $v(\mu'_v)\geq 0$
for every $v\in\Omega\setminus S$
and that
the projection of~$Q'_v$
to $R_{L_i/k}(\A^1_{L_i}) \setminus F_i$ extends to an $\sOint_v$\nobreakdash-point
of $R_{\sOint_{L_i}/\sOint_k}(\A^1_{\sOint_{L_i}}) \setminus \sF_i$
for every $i \in \{1,\dots,n\}$ and every $v\in\Omega\setminus S$,
where~$\sF_i$ denotes the Zariski closure of~$F_i$.
As we have seen in the proof of Lemma~\ref{lem:strongapprox},
this implies that~$x'_{i,v}$ is a unit at every place of~$L_i$ dividing~$v$ except at most
at one place, which must then have degree~$1$ over~$v$.

We now check that the coordinate $t_0\in k$
of~$c$ satisfies the conclusion
of Conjecture~\ref{conj:mainstrong}.
We~have $t_0=\lambda'_v/\mu'_v$, hence~$t_0$ is arbitrarily close to~$t_v$ for~$v\in S$.
Fix $i \in \{1,\dots,n\}$ and a place~$w$ of~$k_i$ whose trace~$v$ on~$k$ does not belong to~$S$.
As $v(\mu'_v)\geq 0$ and $w(b_i)=0$,
the condition $w(t_0-a_i)>0$ implies $w(b_i(\lambda'_v-a_i\mu'_v))>0$ and hence $w(N_{L_i \otimes_kk_v/k_i\otimes_kk_v}(x'_{i,v}))>0$,
so that~$L_i$ possesses a place of degree~$1$ over~$w$ according to the
previous paragraph.

Conversely,
if one assumes Conjecture~\ref{conj:mainstrong},
the density of
$\mkern2mu\bigcup_{c\in \P^1(k)} W_{c}(\A_k)$
in $W(\A_k)$
is an immediate consequence of Theorem~\ref{th:ratpointsmain} below
applied with $B=0$ and~$M''=\emptyset$,
in view of the remark
that
the inverse image map $\Br(k)\to \Br_\vert(W/\P^1_k)$ is onto.
(Recall that the vertical Brauer group is defined as $\Br_\vert(W/\P^1_k)=\Br(W)\cap p_\eta^*\mkern.5mu\Br(\eta)$,
if~$\eta$ denotes the generic point of~$\P^1_k$.)
To prove this remark, we first note that
the morphism $p:W \to \P^1_k$ is smooth and that its fibers are irreducible.
Let $m_i\in\P^1_k$ be the closed point
defined by~$P_i(t)=0$.
The fibers of~$p$ over
$\P^1_k\setminus\{m_1,\dots,m_n\}$
are geometrically irreducible,
while for each $i \in \{1,\dots,n\}$, the algebraic closure of $k_i=k(m_i)$ in the
function field of $p^{-1}(m_i)$ is~$L_i$.
These observations on the fibers of~$p$ imply that
any element of $\Br_\vert(W/\P^1_k)$
can be written as $p_\eta^*\gamma$, where $\gamma \in \Br(\eta)$ has the shape
\begin{align*}
\gamma=\sum_{i=1}^n \Cores_{k_i(t)/k(t)}(b_i(t-a_i),\chi_i)+\delta
\end{align*}
for some $\delta \in \Br(k)$ and some $\chi_i \in \Ker\big(H^1(k_i,\Q/\Z) \to H^1(L_i,\Q/\Z)\big)$, $i\in\{1,\dots,n\}$ satisfying $\sum_{i=1}^n \Cores_{k_i/k}(\chi_i)=0$ (see~\cite[Proposition~1.1.1, Proposition~1.2.1]{ctsd94}).
Now we have $b_i(t-a_i)=b_i(\lambda-a_i\mu)/\mu$ in~$k_i(W)$,
and $b_i(\lambda-a_i\mu)$ is a norm from~$L_i(W)$ according to the defining equations of~$W$.
Thus $(b_i(t-a_i),\chi_i)=-(\mu,\chi_i)$ in $\Br(k_i(W))$ and hence $p_\eta^*\gamma=-\big(\mu,\sum_{i=1}^n \Cores_{k_i/k}(\chi_i)\big)+\delta=\delta$,
as required.
\end{proof}

\begin{cor}
\label{cor:strongapproxconj}
Suppose~$W$ satisfies strong approximation off~$v_0$ for any finite place~$v_0$ of~$k$.
Then the statement of Conjecture~\ref{conj:mainstrong} holds true for all~$S$ and
all $(t_v)_{v\in S}$ satisfying its hypotheses.
\end{cor}

\begin{proof}
Fix an integral model~$\sW$ of~$W$.
A glance at the definition of~$W$ shows that
the map $p:\sW(\sOint_v)\to\P^1(\sOint_v)$ is onto for all but finitely many of the places~$v$ of~$k$ which split completely in~$L_i$ for all~$i$.
Let~$(Q_v)_{v\in\Omega} \in W(\A_k)$.
Let~$S$ be a finite set of places of~$k$.  We must prove that there exist $c \in \P^1(k)$ and $(Q'_v)_{v\in\Omega} \in W_c(\A_k)$
with~$Q'_v$ arbitrarily close to~$Q_v$ for $v \in S$ and~$Q'_v$ integral with respect to~$\sW$ for $v\notin S$.
Fix a place $v_0\notin S$ for which
$p:\sW(\sOint_{v_0})\to\P^1(\sOint_{v_0})$
is onto.
By assumption, we can find a point $Q \in W(k)$ arbitrarily close to~$Q_v$ for $v \in S$ and integral with respect to~$\sW$ for $v\notin S\cup\{v_0\}$.
We then let $c=p(Q)$, choose an arbitrary $Q'_{v_0} \in \sW(\sOint_{v_0})$ such that $p(Q'_{v_0})=c$
and let $Q'_v=Q$ for $v \in \Omega\setminus\{v_0\}$.
\end{proof}

Thanks to Corollary~\ref{cor:strongapproxconj}, we can prove Conjecture~\ref{conj:mainstrong}
for non-abelian
extensions~$L_i/k_i$
when $\sum_{i=1}^n [k_i:k]$ is small.  The underlying geometric arguments run parallel
to those used in the proofs of~\cite[Theorems~A and~B]{ctskodescent} (which themselves
originate from~\cite{skoquadrics}).

\begin{thm}
\label{th:algebraiccases}
Conjecture~\ref{conj:mainstrong} holds under each of the following sets of assumptions:
\begin{itemize}
\item[(i)] $\sum_{i=1}^n [k_i:k] \leq 2$;
\item[(ii)] $\sum_{i=1}^n [k_i:k]=3$ and $[L_i:k_i]=2$ for every~$i$.
\end{itemize}
\end{thm}

\begin{proof}
When $\sum_{i=1}^n [k_i:k]\leq 2$, one checks that~$W$ is isomorphic to the complement
of a codimension~$2$ closed subset in an affine
space.  (Specifically, when $\sum_{i=1}^n [k_i:k]=2$, the projection
$W \to \prod_{i=1}^n R_{L_i/k}(\A^1_{L_i})$ is an open immersion and the complement of its
image has codimension~$2$.)
As a consequence, the variety~$W$ satisfies strong approximation off any place of~$k$
(see Lemma~\ref{lem:strongapproxaffine}), so that Corollary~\ref{cor:strongapproxconj} applies.
Suppose now $\sum_{i=1}^n [k_i:k]=3$ and $[L_i:k_i]=2$ for every~$i$.
In this case, just as in \cite[p.~392]{ctskodescent}, the variety~$W$ is the punctured affine cone
over the complement of a closed subset of codimension~$2$ in a smooth projective
quadric of dimension~$4$.  (The smooth projective quadric in question is the
Zariski closure of the image of~$W$ in the projective space of lines of the $6$\nobreakdash-dimensional vector
space $\prod_{i=1}^n R_{L_i/k}(\A^1_{L_i})$.)  The same argument as in the proof of
Lemma~\ref{lem:strongapproxaffine},
with smooth projective conics replacing affine lines, shows
that the complement of a codimension~$2$ closed subset in a smooth projective quadric satisfies strong approximation.
We may therefore apply \cite[Proposition~3.1]{ctx11} and deduce that~$W$ satisfies strong approximation off
any place of~$k$. Applying Corollary~\ref{cor:strongapproxconj} concludes the proof.
\end{proof}

\begin{rmk}
\label{rk:geomcrit}
In hindsight, one may reinterpret the arguments of~\textsection\ref{subsec:constructionc} as follows.
Let~$\A(V)$ denote the affine space whose underlying vector space is $V = H^0(C,\sO_C(c_1))$
and let $\P(V)$ denote the projective space of lines in~$\A(V)$.
Let~$W$ be the fiber, above~$0$, of the morphism
\begin{align*}
\left(\A(V) \setminus \{0\}\right) \times \prod_{m \in M} \left(R_{L_m/k}(\A^1_{L_m}) \setminus F_m\right) \to \prod_{m\in M} R_{k(m)/k}(\A^1_{k(m)})
\end{align*}
defined by $(h,(x_m)_{m\in M}) \mapsto (h(m)-N_{L_m/k(m)}(x_m))_{m\in M}$,
where~$F_m$ stands for the singular locus of
$R_{L_m/k}(\A^1_{L_m}) \setminus R_{L_m/k}(\mathbf{G}_{\mathrm{m},{L_m}})$.
Let $p:W\to \P(V)$ denote the map induced by the first projection.
Let $w_\A \in W(\A_k)$ be an adelic point whose image in $(\A(V) \setminus \{0\})(k_v)$,
for each $v \in S$, is the
function~$h_v$ of~\textsection\ref{subsec:constructionc}.
The choice of a splitting of~\eqref{eq:apprr} determines an isomorphism between~$W$ and the complement
of a codimension~$2$ closed subset in an affine space; therefore~$W$ satisfies strong approximation off any given place.
In particular, by the same argument as in
Corollary~\ref{cor:strongapproxconj}, one can approximate~$w_\A$ by an adelic point
of $p^{-1}(c)$ for some rational point~$c$ of~$\P(V)$.
This~$c$, viewed as an effective divisor on~$C$, is the divisor constructed in~\textsection\ref{subsec:constructionc}.
\end{rmk}

\subsubsection{Invariance under change of coordinates}

As a test for Conjectures~\ref{conj:mainstrong} and~\ref{conj:mainweak},
it is natural to ask whether
they are compatible
with linear changes of coordinates on~$\P^1_k$.  The following lemma
expresses such an invariance
property.
It will be used in the proofs of Theorems~\ref{thm:matthiesen} and~\ref{th:irving} below.
Its statement should be read as follows: given any $\alpha,\beta,\gamma,\delta$ which satisfy (i)--(iii),
Conjectures~\ref{conj:mainstrong} and~\ref{conj:mainweak} for~$k$, $n$, $k_i$, $L_i$, $a_i$, $b_i$, $S$ and $(t_v)_{v \in S}$
reduce to the same conjectures for~$k$, $n$, $k_i$, $L_i$, $a_i'$, $b_i'$, $S'$ and $(t_v')_{v \in S'}$.

\begin{lem}
\label{lem:changeofvariables}
Let~$k$, $n$, $k_i$, $L_i$, $a_i$, $b_i$, $S$ and $(t_v)_{v \in S}$ satisfy the hypotheses of Conjecture~\ref{conj:mainstrong}.
Let $\alpha,\beta,\gamma,\delta \in k$ be such that $\alpha\delta-\beta\gamma\neq 0$.
We make the following assumptions:
\begin{enumerate}
\item[(i)]
 $\gamma a_i +\delta\neq 0$
 for all~$i \in \{1,\dots,n\}$ and $\gamma t_v +\delta\neq 0$ for all $v \in S$;
\item[(ii)] $\alpha$ and~$\alpha a_i+\beta$ for $i \in \{1,\dots,n\}$ are units outside of~$S$;
\item[(iii)]
 for every $i \in \{1,\dots,n\}$
and every $v \in S$,
there exists $y_{i,v} \in (L_i\otimes_k k_v)^*$
such that
$(\alpha\delta-\beta\gamma)/(\gamma t_v+\delta)=
N_{L_i\otimes_kk_v/k_i\otimes_kk_v}(y_{i,v})$.
\end{enumerate}
Set $a_i'=(\alpha a_i+\beta)/(\gamma a_i + \delta)$
and $b_i'=b_i(\gamma a_i + \delta)$ for $i \in \{1,\dots,n\}$.
Let~$S'_0$ be the union of~$S$ with the set of finite places of~$k$
above which one of~$\alpha$, $\gamma$, $a_1',\dots,a_n'$
is not integral
or one of~$b_1',\dots,b_n'$ is not a unit.
Let~$S'$ be a finite set of places of~$k$ containing~$S'_0$.
Set~$t_v'=0$ for $v \in S' \setminus S$
and $t_v'=(\alpha t_v+\beta)/(\gamma t_v+\delta)$ for $v \in S$.
Finally, when given $t_0'\in k$ such that $\alpha\neq\gamma t_0'$,
we denote by~$t_0 \in k$
the unique solution to the equation $t_0'=(\alpha t_0 + \beta)/(\gamma t_0+\delta)$.
We then have:
\begin{enumerate}
\item for every $i \in \{1,\dots,n\}$ and every $v \in S'$, there exists $x_{i,v}'\in (L_i\otimes_k k_v)^*$
such that
$$b_i'(t_v'-a_i')=N_{L_i\otimes_k k_v/k_i\otimes_kk_v}(x_{i,v}')$$
in $k_i\otimes_k k_v$.
\item for $v \in S$, if~$t'_0$ is arbitrarily close to~$t'_v$,
then~$t_0$ is arbitrarily close to~$t_v$;
\item for $v \in S$,
if~$i$ denotes an element of $\{1,\dots,n\}$
and~$w$ denotes a real place of~$k_i$ dividing~$v$,
the sign of $(t'_0-a_i')(t'_v-a_i')(t_0-a_i)(t_v-a_i)$
at~$w$
is independent of~$w$ and of~$i$;
\item if $t'_0$ is arbitrarily close to $t'_v$ for $v \in S' \setminus S$,
then for every $i \in \{1,\dots,n\}$ and every finite place~$w$ of~$k_i$
such that $w(t_0-a_i)>0$,
either~$w$ lies above a place of~$S$, or~$w$ does not lie above a place of~$S'$
and $w(t_0'-a_i')>0$.
\end{enumerate}
\end{lem}

\begin{proof}
It is straightforward to check~(2) and~(3),
as well as~(1) for $v \in S$ by setting $x_{i,v}' = y_{i,v}x_{i,v}$.
To check~(1) for $v \in S' \setminus S$,
we note that above such~$v$, the element $b_i'(t'_v-a_i')=-a_i'b_i'=-(\alpha a_i+\beta)b_i$
is a unit, by~(ii), and is therefore a norm from $L_i \otimes_k k_v$ since~$L_i/k_i$
is unramified above~$v$.  We now turn to~(4) and fix~$i$ and a finite place~$w$
of~$k_i$ such that~$w(t_0-a_i)>0$.
It follows from~(ii) that~$w$ cannot lie
above a place of~$S'\setminus S$, as~$t_0-a_i$ would then have to be a unit at~$w$,
being arbitrarily close to $-\beta/\alpha-a_i=-(\alpha a_i+\beta)/\alpha$.
Thus we may assume that~$w$ lies above a place of~$S'$,
so that $w(b_i)=w(b_i')=0$ and $w(\alpha)\geq 0$, $w(\gamma)\geq 0$, $w(a_i')\geq 0$.
A simple computation shows that
\begin{align*}
b_i'(t_0'-a_i')=b_i(t_0-a_i)(\alpha-\gamma t_0')\rlap{\text{.}}
\end{align*}
As $w(t_0-a_i)>0$,
we deduce that if $w(t_0'-a_i')\leq 0$, then $w(\alpha-\gamma t_0')<0$, which implies that $w(t_0'-a_i')=w(t_0') \leq w(\gamma t_0')=w(\alpha-\gamma t_0')$, a contradiction.
\end{proof}

\subsubsection{Additive combinatorics}
\label{subsubsec:additivecomb}

The most significant result towards Conjecture~\ref{conj:mainstrong} is the following
theorem,
whose proof relies on the methods initiated by Green and Tao.

\begin{thm}[Matthiesen~\cite{matthiesen}]
\label{thm:matthiesen}
Conjecture~\ref{conj:mainstrong} holds
when $k=k_1=\dots=k_n=\Q$.
\end{thm}

\begin{proof}
By Remarks~\ref{rk:conj}~(ii) and~(iii),
we may assume that~$a_i$ is a unit outside~$S$
for all~$i$.
For $N \in \N$,
let
$(\alpha,\beta,\gamma,\delta)=(r^{2N},0,0,r^{N})$,
where~$r$ denotes the product of the primes belonging to~$S$.
Letting~$N$ be a large enough multiple of~$\prod_{i=1}^n [L_i:\Q]$,
we may assume, by Lemma~\ref{lem:changeofvariables},
that $a_i,b_i\in \Z$ for all $i\in\{1,\dots,n\}$, that $v(t_v)\geq 0$ for every prime $v \in S$, that $b_i(t_v-a_i)$ is the norm of an integral element of $L_i \otimes_\Q \Q_v$
for every $i \in \{1,\dots,n\}$ and every prime $v \in S$
and that~$S$ is large enough for the main theorem of~\cite{matthiesen} to be applicable.
Let $t_2 \in \Q$ be arbitrarily close to~$t_v$ for~$v \in S$.
We write $t_2=\lambda_2/\mu_2$ with $\lambda_2,\mu_2 \in \Z$.
Using strong approximation, we choose
$\alpha \in \Q$
arbitrarily close to~$1/\mu_2$ at the finite places of~$S$,
integral outside~$S$ and such that $\alpha\mu_2>0$.

The pair
 $(\lambda_1,\mu_1)=(\alpha\lambda_2,\alpha\mu_2)$
 belongs to~$\Z^2$.
It lies arbitrarily close to $(t_v,1)$ at the finite places of~$S$.
Moreover~$\lambda_1/\mu_1$ is arbitrarily close to~$t_v$ at the real place~$v$
of~$\Q$, and $b_i(\lambda_1-a_i\mu_1)$ is a local integral norm from~$L_i$
at the finite places of~$S$ and is a local norm from~$L_i\otimes_\Q\R$, for all~$i$.
Matthiesen's theorem~\cite{matthiesen}
applied
to the linear forms $f_i(\lambda,\mu)=b_i(\lambda-a_i\mu) \in \Z[\lambda,\mu]$
for $i \in \{1,\dots,n\}$
and
to the vector $(\lambda_1,\mu_1) \in \Z^2$
therefore produces
a pair $(\lambda_0,\mu_0)\in \Z^2$ arbitrarily close to~$(\lambda_1,\mu_1)$
at the finite places of~$S$, such that~$\lambda_0/\mu_0$ is arbitrarily close
to~$\lambda_1/\mu_1$ at the real place
and such that $b_i(\lambda_0-a_i\mu_0)$ is squarefree outside~$S$
and is the norm of an integral element of~$L_i$, for all~$i$.
The rational number $t_0=\lambda_0/\mu_0$ then satisfies all of the desired conclusions.
Indeed, if~$v$ is a prime not in~$S$ such that $v(t_0-a_i)>0$ for some~$i$, then,
as $v(b_i)=0$ and $v(\mu_0)\geq 0$, we must have $v(b_i(\lambda_0-a_i\mu_0))=1$,
which implies that~$L_i$ possesses a place of degree~$1$ over~$v$
since $b_i(\lambda_0-a_i\mu_0)$ is a local integral norm from~$L_i$ at~$v$
and~$L_i/\Q$ is unramified at~$v$.
\end{proof}

\subsubsection{Sieve methods}

Reformulating the work of Irving~\cite{irving}, which rests on sieve methods,
yields the first known case of Conjecture~\ref{conj:mainweak} for which $\sum_{i=1}^n [k_i:k]\geq 4$
and~$k_i\neq k$ for some~$i$ (here the extensions $L_i/k_i$ are
once again almost abelian).

\begin{thm}
\label{th:irving}
Let~$K/\Q$ be a cubic extension.
Let~$q \geq 7$ be a prime number.
If~$K$ has a unique real place and does not embed into
the cyclotomic field~$\Q(\zeta_q)$, then
Conjecture~\ref{conj:mainweak} holds for $n=2$, $k=\Q$, $k_1=K$,
$L_1=K(2^{1/q})$,
$k_2=\Q$ and $L_2=\Q(2^{1/q})$.
\end{thm}

\begin{proof}
Let~$a_1$, $a_2$, $b_1$, $b_2$, $S$ and $(t_v)_{v \in S}$ satisfy the hypotheses of Conjecture~\ref{conj:mainweak}.
When~$v$ is the real place of~$\Q$, we also denote~$t_v$ by~$t_\infty$.

\begin{lem}
\label{lem:signchangeirving}
By a change of coordinates on~$\P^1_k$, we may assume, in order to prove Theorem~\ref{th:irving}, that
\begin{align}
\label{eq:irvinequalities}
b_2 t_\infty<b_2 a_1<b_2a_2
\end{align}
at any real place of~$K$.
\end{lem}


\begin{proof}
We choose $\beta \in \Q$ such that $b_2(a_2+\beta)>0$.
After enlarging~$S$, we may assume,
by Remark~\ref{rk:conj}~(iii),
that~$a_1+\beta$ and~$a_2+\beta$
are units outside~$S$.
Then, using weak approximation, we let $\delta \in \Q$ be arbitrarily large at the finite places of~$S$,
arbitrarily close to (but distinct from)~$a_2$ at the real place of~$\Q$, and such that $\delta-a_2$
has the same sign as
$(t_\infty-a_1)(t_\infty-a_2)(a_2-a_1)$
at the unique real place of~$K$.
By choosing~$\delta$ large enough at the finite places of~$S$, we may assume
that
the image of $(\delta+\beta)/(\delta-t_v)$ in $k_i \otimes_\Q \Q_v$ is a norm from $(L_i\otimes_\Q \Q_v)^*$ for all~$i$ and all $v \in S$ (at the real place, this condition is empty).
Setting $\alpha=1$ and~$\gamma=-1$,
all of the hypotheses of Lemma~\ref{lem:changeofvariables} are now satisfied.
In the notation of that lemma, it is straightforward to check that $b_2'(t_\infty'-a_1')<0<b_2'(a_2'-a_1')$ at the
real place of~$K$.
We may therefore replace~$a_i$, $b_i$, $S$ and $(t_v)_{v \in S}$
with~$a_i'$, $b_i'$, $S'$ and~$(t'_v)_{v\in S'}$ and assume that~\eqref{eq:irvinequalities}
holds.
\end{proof}

Let~$c$ be a nonzero integer such that
the binary cubic form
\begin{align}
f(x,y)=c^q N_{K/\Q}\Bigg(\mkern-.5mub_1\mkern-.5mu\bigg((a_2-a_1)x+\frac{1}{b_2}y\bigg)\Bigg)
\end{align}
has integral coefficients and the coefficient of~$x^3$ is positive.
By Remark~\ref{rk:conj}~(iii), we may assume, after enlarging~$S$, that~$a_1$ is integral outside~$S$,
that the coefficients of~$x^3$ and of~$y^3$ and the discriminant of the polynomial~$f(x,1)$
are units outside~$S$ and that~$S$ contains all of the primes less than or equal to
the constant denoted by~$P_1$ in~\cite{irving} (see the end of~\textsection2).

Let $\tau \in \Q$ be arbitrarily close to~$t_v$ at the finite places of~$S$.  We choose
integers $\lambda,\mu$ such that
\begin{align}
\label{eq:irv2}
\frac{\mu}{\lambda}= b_2(\tau-a_2)
\end{align}
and such that~$\lambda$ is a $q$\nobreakdash-th power.  We then have
\begin{align}
\label{eq:irv1}
\frac{f(\lambda,\mu)}{\lambda^3} = c^q N_{K/\Q}(b_1(\tau-a_1))\rlap{\text{.}}
\end{align}
The hypothesis of Conjecture~\ref{conj:mainweak}
implies that $b_i(\tau-a_i)$ is a local norm in $L_i/k_i$ at the places of~$S$
for $i\in\{1,2\}$.
As~$\lambda$ is a $q$\nobreakdash-th power,
it follows, in view of~\eqref{eq:irv2} and \eqref{eq:irv1},
that~$\mu$ and~$f(\lambda,\mu)$
are local norms in $\Q(2^{1/q})/\Q$ at the places of~$S$.
Note, furthermore, that the real roots of the polynomial $f(x,1)$ are negative,
by the assumption made in Lemma~\ref{lem:signchangeirving}.
We are therefore in the situation considered in \cite[\textsection3--6]{irving}
(the integers $a_0$, $b_0$ of \emph{loc.\ cit.}\ being our $\lambda$, $\mu$).
Irving proves in~\cite{irving} the existence of positive integers $x_0$, $y_0$
with $(x_0,y_0)$ arbitrarily close to~$(\lambda,\mu)$ at the finite places of~$S$,
such that $y_0f(x_0,y_0)$ has no prime factor $p\notin S$ with $p\equiv 1 \text{ mod } q$.
In particular~$2$ is a $q$\nobreakdash-th power in~$\Fp$
for any prime factor~$p\notin S$ of $y_0f(x_0,y_0)$;
hence,
for any such~$p$ and any place~$w$ of~$K$ dividing~$p$,
the number field~$L_1$ (resp.~$L_2$)
possesses a place of degree~$1$ over~$w$
(resp.~over~$p$).

The unique solution~$t_0 \in \Q$ to the equation
\begin{align}
\label{eq:irvt2}
\frac{y_0}{x_0}=b_2(t_0-a_2)
\end{align}
is arbitrarily close to~$t_v$ for $v \in S \cap \Omega_f$ and satisfies
\begin{align}
\label{eq:irvt1}
\frac{f(x_0,y_0)}{x_0^3} = c^q N_{K/\Q}(b_1(t_0-a_1))\rlap{\text{.}}
\end{align}
As~$x_0$ and~$y_0$ are positive, we deduce from~\eqref{eq:irvinequalities}
and~\eqref{eq:irvt2} that
 $b_2t_\infty<b_2 a_1<b_2a_2<b_2t_0$.
Assertion~(1') of Conjecture~\ref{conj:mainweak} follows.
For any finite place $v\notin S$ such that $v(t_0-a_2)>0$, we have $v(y_0)>0$,
so that~$L_2$ possesses a place of degree~$1$ over~$v$.
Let now~$w$ be a finite place of~$K$ above a place $v \notin S$ of~$\Q$,
such that $w(t_0-a_1)>0$.
We remark that~$t_0$ must be integral at~$p$ since~$a_1$ is integral above~$p$.
As a consequence~$t_0-a_1$ is integral at every place of~$K$ dividing~$v$.
It follows that $v(N_{K/\Q}(b_1(t_0-a_1)))>0$ and hence $v(f(x_0,y_0))>0$,
so that~$L_1$ possesses a place of degree~$1$ over~$w$ according to the previous paragraph.
\end{proof}

\subsection{Consequences for rational points in fibrations}
\label{subsec:consequencesrationalpoints}

All of the results on rational points contained in this paper
will be deduced from the following theorem.
We adopt Poonen's notation
$V(\A_k)_{\smash{\bullet}}=\prod_{v \in \Omega_f} V(k_v) \times \prod_{v \in \Omega_\infty} \pi_0(V(k_v))$
for any proper variety~$V$ over a number field~$k$
(see~\cite[\textsection2]{stoll}).

\begin{thm}
\label{th:ratpointsmain}
Let~$X$ be a smooth, irreducible variety over a number field~$k$, endowed with a morphism
$f:X \to \P^1_k$ with geometrically irreducible generic fiber.  Assume that every fiber of~$f$ contains
an irreducible component of multiplicity~$1$.
Let $U \subseteq \P^1_k$ be a dense open subset over which the fibers of~$f$ are split,
with $\infty\in U$.
Let $B \subseteq \Br(f^{-1}(U))$ be a finite subgroup.
Let $(x_v)_{v \in \Omega} \in X(\A_k)$
be orthogonal to $(B+f_\eta^*\mkern.5mu\Br(\eta))\cap\Br(X)$ with respect to the Brauer--Manin pairing.

Let~$M'$ be a nonempty subset of~$\P^1_k\setminus U$ containing all of the points
above which the fiber of~$f$ is not split,
large enough that
for each $m \in M'' = (\P^1_k \setminus U) \setminus M'$,
the map
\begin{align}
\label{eq:mapresidue}
\Br\mkern-1.5mu\left(\P^1_k \setminus (M' \cup \{m\})\right) \longrightarrow H^1(k(m),\Q/\Z)
\end{align}
which sends a class to its residue at~$m$
is surjective.
Let $P_1,\dots,P_n\in k[t]$
denote the irreducible monic polynomials which vanish at the points of~$M'$.
Assume Conjecture~\ref{conj:mainstrong} for~$P_1,\dots,P_n$
(resp.~assume Conjecture~\ref{conj:mainweak}
for~$P_1,\dots,P_n$
and assume that either~$M''=\emptyset$ and~$f$ is proper or~$k$ is totally imaginary).

Then there exist $c \in U(k)$ and
$(x'_v)_{v \in \Omega} \in X_c(\A_k)$
such that~$X_c$ is smooth
and~$(x'_v)_{v \in \Omega}$
is orthogonal to~$B$ with respect to the Brauer--Manin pairing
and is arbitrarily
close to~$(x_v)_{v \in \Omega}$ in $X(\A_k)$ (resp.~in~$X(\A_k)_{\smash{\bullet}}$).
\end{thm}

\begin{rmks}
\label{rk:useofconj}
(i)
The reader who is willing to assume Conjectures~\ref{conj:mainstrong} and~\ref{conj:mainweak}
for the whole collection of polynomials vanishing on the points of $\P^1_k \setminus U$
may take $M'=\P^1_k\setminus U$.  This makes the hypothesis on~\eqref{eq:mapresidue} trivially satisfied
and also leads to simplifications
in the proof of Theorem~\ref{th:ratpointsmain}.

(ii)
The map~\eqref{eq:mapresidue}
is surjective for any~$m$ as soon as~$M'$ contains a rational point or~$k$ is totally imaginary.
This follows from
the Faddeev exact sequence
(see~\cite[\textsection1.2]{ctsd94}),
in view of
the surjectivity of the
corestriction map
$H^1(k',\Q/\Z)\to H^1(k,\Q/\Z)$ for any finite extension~$k'/k$
of totally imaginary number fields
(see~\cite[Note, p.~327]{gras} and note that $H^1(k,\Q/\Z)$ is Pontrjagin dual to the
group denoted $C_k/D_k$ in \emph{loc.\ cit.}).

(iii)
Write $M'=\{m_1,\dots,m_n\}$
and $k_i=k(m_i)$.
When $B=0$, the finite extension~$L_i/k_i$ to which Conjectures~\ref{conj:mainstrong}
and~\ref{conj:mainweak} are applied in the proof below
can be taken to be
the algebraic closure of~$k_i$ in the function field of an irreducible
component of multiplicity~$1$ of~$f^{-1}(m_i)$, for every $i\in\{1,\dots,n\}$.
\end{rmks}

\begin{proof}[Proof of Theorem~\ref{th:ratpointsmain}]
Let us write $M'=\{m_1,\dots,m_n\}$
and $M''=\{m_{n+1},\dots,m_N\}$, with~$N=n$ when $M''=\emptyset$.
For each~$i \in \{1,\dots,N\}$, let $k_i=k(m_i)$
and $X_i=f^{-1}(m_i)$,
choose an irreducible component~$Y_i\subseteq X_i$ of multiplicity~$1$,
with the requirement that~$Y_i$ should be geometrically irreducible over~$k_i$ whenever~$i>n$,
and choose a finite abelian extension $E_i/k(Y_i)$ such that the residue of any element of~$B$
at the generic point of~$Y_i$ belongs to the kernel of the restriction map
$H^1(k(Y_i),\Q/\Z)\to H^1(E_i,\Q/\Z)$.

Thanks to our assumption on~\eqref{eq:mapresidue},
we can choose,
for each $i \in \{n+1,\dots,N\}$,
a finite subgroup $\Gamma_i \subset \Br\left(\P^1_k \setminus (M' \cup\{m_i\})\right)$
which surjects, by the residue map at~$m_i$,
onto the kernel of the restriction map $H^1(k_i,\Q/\Z)\to H^1(E_i,\Q/\Z)$.
Put $\Gamma=\sum_{i=n+1}^N \Gamma_i \subset \Br(U)$.

For $i \in \{1,\dots,N\}$, let~$K_i$ be the algebraic closure of~$k_i$
in~$E_i$.
For $i\in\{n+1,\dots,N\}$, set $L_i=K'_i=K_i$
and note that this is an abelian extension of~$k_i$ as~$k_i$ is algebraically
closed in~$k(Y_i)$ for such values of~$i$.
For $i\in\{1,\dots,n\}$,
fix a finite abelian extension~$K'_i$ of~$k_i$ such that
the residue at~$m_i$ of any element of~$\Gamma$ belongs to the kernel of the restriction
map $H^1(k_i,\Q/\Z)\to H^1(K'_i,\Q/\Z)$
and
let~$L_i/k_i$ be a compositum of the extensions~$K_i/k_i$ and~$K'_i/k_i$.
Finally, choose a finite extension $L_\infty/k$ such that the fields~$L_1,\dots,L_N$ embed
$k$\nobreakdash-linearly into~$L_\infty$.

Let $C^0 = U \setminus\{\infty\}$.  Let $X^0=f^{-1}(C^0)$.
Let~$\Pic_+(\P^1_k)$ and~$\Br_+(\P^1_k)$ denote the groups associated in Definition~\ref{def:picpbrp} to the
curve~$\P^1_k$, to the finite set $M=\{m_1,\dots,m_N,\infty\}\subset \P^1_k$ and to the finite extensions~$L_1/k_1,\dots,L_N/k_N,L_\infty/k$.

By Remark~\ref{rks:picbr}~(ii), the subgroup $B+f^*\Br_+(\P^1_k) \subseteq \Br(X^0)$ is finite
modulo the
inverse image of $\Br(k)$.
Thus, by Harari's formal lemma,
there exists $(x''_v)_{v \in \Omega} \in X^0(\A_k)$
orthogonal to $B+f^*\Br_+(\P^1_k)$ with respect to~\eqref{eq:globalpairing}
and arbitrarily close to $(x_v)_{v \in \Omega}$ in~$X(\A_k)$
(see~\cite[Th\'eor\`eme~1.4]{ctbudapest}).
By the inverse function theorem, we may assume that~$x''_v$ belongs to a smooth fiber of~$f$ for each~$v$.
As $(x''_v)_{v \in \Omega}$ is orthogonal to~$f^*\Br_+(\P^1_k)$,
the class of $(f(x''_v))_{v \in \Omega}$ in $\PicplusA(\P^1_k)$
is orthogonal to~$\Br_+(\P^1_k)$ with respect to the
pairing~\eqref{eq:pairingbmplus}.
By Theorem~\ref{th:arithduality}, we deduce that there exists a divisor $c_1 \in \Div(C^0)$,
necessarily of degree~$1$,
whose class in $\PicplusA(\P^1_k)$ coincides with that of $(f(x''_v))_{v\in\Omega}$.
As $\Pic_+(\P^1_\C)=\Pic(\P^1_\C)=\Z$ and $\deg(c_1)=\deg(f(x''_v))$, the classes of~$c_1$ and of~$f(x''_v)$ in~$\Picplus(\P^1_{k_v})$ must in fact be equal
for all $v \in \Omega$, finite or infinite.

Let us write $\A^1_k=\Spec(k[t])$.
For~$v \in\Omega$,
let $t_v \in k_v$ denote the value of~$t$ at $f(x''_v) \in \A^1_{k_v}$.
For $i \in \{1,\dots,N\}$,
let~$a_i \in k_i$ denote the value
of~$t$ at~$m_i \in \A^1_k$.
Evaluating the invertible function $t-a_i \in \Gm(C^0 \otimes_k k_i)$
along the divisor $-c_1 \otimes_k k_i \in \Div(C^0 \otimes_k k_i)$
yields an element of~$k_i^*$ which we denote~$b_i$
(see~\cite[Ch.~III, \textsection1.1]{serregroupesalg}).

\begin{lem}
\label{lem:checkhypconj}
For every~$i \in \{1,\dots,N\}$ and every $v \in \Omega$,
there exists $x_{i,v} \in (L_i \otimes_kk_v)^*$
such that the equality
$b_i(t_v-a_i) = N_{L_i\otimes_k k_v/k_i\otimes_kk_v}(x_{i,v})$
holds in $k_i\otimes_k k_v$.
\end{lem}

\begin{proof}
For any~$i$, any~$v$ and any place~$w$ of~$k_i$ dividing~$v$,
we must show
that $b_i(t_v-a_i)$, as an element of~$(k_i)_w$, is a norm from $L_i \otimes_{k_i} (k_i)_w$.
Let us fix~$i$, $v$ and~$w$.
As the classes of~$c_1$ and of~$f(x_v'')$ in~$\Pic_+(\P^1_{k_v})$ are equal,
there exists a rational function $h_v \in k_v(t)^*$
such that $\div(h_v)=f(x_v'')-c_1$ and such that
the value of~$h_v$ at any $(k_i)_w$\nobreakdash-point of~$m_i \cup \{\infty\}$
is a norm from $L_i \otimes_{k_i}(k_i)_w$ (recall that~$L_i$ embeds into~$L_\infty$).
Thus, evaluating~$h_v$ along the divisor of the rational function $t-a_i$ on~$\P^1_{(k_i)_w}$
yields an element of
$N_{L_i \otimes_{k_i} (k_i)_w/(k_i)_w}((L_i\otimes_{k_i} (k_i)_w)^*)$.
By Weil's reciprocity law (see~\cite[Ch.~III, \textsection1.4, Proposition~7]{serregroupesalg}),
the lemma follows.
\end{proof}

Let~$S$ be a finite set of places of~$k$ containing the infinite places of~$k$,
the places at which we want to approximate~$(x_v)_{v \in \Omega}$
and the finite places above which, for some~$i$, either~$b_i$ is not a unit or~$L_i/k_i$
is ramified.  We choose~$S$ large enough
that the order of~$B$ is invertible in~$\sOint_S$,
that $\beta(x''_v)=0$ for all $v\in \Omega\setminus S$ and all $\beta\in B+f^*\Gamma$,
that~$X$ extends to a smooth scheme~$\sX$ over~$\sOint_S$
and that $f:X\to \P^1_k$
extends to a flat morphism $f:\sX \to \P^1_{\sOint_S}$.
For~$m \in \P^1_k$, let~$\mtilde$ denote the Zariski closure of~$m$ in~$\P^1_{\sOint_S}$.
For each~$i$, let~$\sX_i$ and~$\sY_i$ denote the Zariski closures of~$X_i$ and~$Y_i$ in~$\sX$, endowed with the reduced
scheme structures, let $\sY_i^0 \subseteq \sY_i$ be a dense open subset
and let~$\sE_i$
denote the normalisation of~$\sY_i^0$ in the finite extension~$E_i/k(Y_i)$.
By shrinking~$\sY_i^0$, we may assume that~$\sE_i$ is finite and \'etale
over~$\sY_i^0$
and that $\sY_i^0$ is smooth over~$\sOint_S$.
Let $\sU=\P^1_{\sOint_S} \setminus \big(\bigcup_{i=1}^N \mtilde_i\big)$.
Finally, for $i \in \{n+1,\dots,N\}$, let $G_i=\Gal(E_i/k(Y_i))$
and $H_i=\Gal(E_i/k(Y_i)K_i) \subseteq G_i$.

After enlarging~$S$, we may assume that $B \subseteq \Br(f^{-1}(\sU))$,
that $\Gamma \subseteq \Br(\sU)$,
that the Zariski closure~$\Mtilde$ of~$M$ in~$\P^1_{\sOint_S}$
is \'etale over~$\sOint_S$ and,
by the Lang--Weil--Nisnevich bounds~\cite{langweil} \cite{nisnevic}
and by a geometric version of Chebotarev's density theorem~\cite[Lemma~1.2]{ekedahl},
that the following statements hold:
\begin{itemize}
\item the fiber of~$f$ above any closed point of~$\sU$
contains a smooth rational point;
\smallskip
\item for any $i \in \{1,\dots,n\}$ and any place~$w$ of~$k_i$
which does not lie above a place of~$S$,
if~$L_i$ possesses a place of degree~$1$ over~$w$
then the fiber of $\sE_i \to \mtilde_i$ above the closed point corresponding to~$w$ contains a rational point;
\smallskip
\item for any $i \in \{n+1,\dots,N\}$,
the closed fibers of $\sY_i^0 \to \mtilde_i$
contain rational points;
\smallskip
\item for any $i \in \{n+1,\dots,N\}$ and any place~$w$ of~$k_i$
which splits completely in~$L_i$
and which does not lie above a place of~$S$,
any element of~$H_i$ can be realised as the Frobenius automorphism
of the irreducible abelian \'etale cover $\sE_i\to\sY_i^0$
at some rational point
of the fiber of $\sY_i^0 \to \mtilde_i$ above the closed point corresponding to~$w$.
\end{itemize}

Using Chebotarev's density theorem, 
we fix pairwise distinct places $v_{n+1},\dots,v_N \in \Omega\setminus S$
such that~$v_i$ splits completely in~$L_i$ for each~$i$.
For each $i \in \{n+1,\dots,N\}$,
we also fix a place~$w_i$ of~$k_i$ lying over~$v_i$
and an element $t_{v_i} \in k_{v_i}$
such that $w_i(t_{v_i}-a_i)=1$.

Thanks to Lemma~\ref{lem:checkhypconj}, we may apply Conjecture~\ref{conj:mainstrong}
(resp.~Conjecture~\ref{conj:mainweak}) to the
polynomials $P_1,\dots,P_n$, to the finite extensions~$L_i/k_i$
for $i \in \{1,\dots,n\}$ and to the set of places $S'=S \cup\{v_{n+1},\dots,v_N\}$:
there exists $t_0 \in k$ satisfying conditions~(1) and~(2) (resp.~(1') and~(2))
of~\textsection\ref{subsec:conj},
arbitrarily close to~$t_{v_i}$ for $i\in\{n+1,\dots,N\}$.

Let $c \in \P^1(k)$ denote the point with coordinate~$t_0$.
As at least one~$x''_v$ belongs to a smooth fiber of~$f$ above~$U$,
we may assume that $c\in U(k)$ and that~$X_c$ is smooth.
By the inverse function theorem,
for each $v \in S$ (resp.~$v \in S \cap \Omega_f$),
we can choose $x'_v \in X_c(k_v)$ arbitrarily close to~$x_v''$.
In
the case in which only Conjecture~\ref{conj:mainweak} is assumed to hold for~$P_1,\dots,P_n$,
let us define~$x'_v$ for $v \in \Omega_\infty$.
If~$v$ is complex, we let~$x'_v$ be an arbitrary point of~$X_c(k_v)$.
If~$v$ is real, then~$k$ is not totally imaginary;
recall that we have assumed, in this case, that~$M''=\emptyset$
and that~$f$ is proper.
As $M''=\emptyset$,
condition~(1') of Conjecture~\ref{conj:mainweak}
implies that~$c$ and~$f(x_v'')$
belong to the same connected component of~$U(k_v)$.
As~$f$ is flat and proper,
the map $f^{-1}(U)(k_v)\to U(k_v)$ induced by~$f$ is open and closed,
so that
it maps any connected component of $f^{-1}(U)(k_v)$ onto a connected component
of~$U(k_v)$.
Thus, we can choose $x'_v \in X_c(k_v)$ such that~$x_v''$ and~$x'_v$ belong to the
same connected component of~$f^{-1}(U)(k_v)$.

Let us now construct~$x'_v$ for $v \in \Omega\setminus S$.

For $v \in \Omega\setminus S$, let $w \in \P^1_{\sOint_S}$ denote the closed point $w=\ctilde \cap \P^1_{\Fv{v}}$.
For $i \in \{1,\dots,N\}$, we define~$\Omega_i$ to be the set of places $v \in \Omega\setminus S$ such that $w\in\mtilde_i$.  The sets~$\Omega_i$ are finite and pairwise disjoint.
When $v \in \Omega_i$, we may view~$w$ as a place of~$k_i$ dividing~$v$;
we then have $w(t_0-a_i)>0$.  For $i>n$,
we may assume that $w_i(t_0-a_i)=1$,
and hence that $v_i \in \Omega_i$,
by choosing~$t_0$ close enough to~$t_{v_i}$.

For each $v \in \Omega \setminus S$
which does not belong to any~$\Omega_i$,
we use Hensel's lemma to lift an arbitrary smooth rational point of~$f^{-1}(w)$ to a $k_v$\nobreakdash-point~$x'_v$ of~$X_c$.
For each $i\in\{1,\dots,n\}$ and each $v\in \Omega_i$,
the field~$L_i$ possesses a place of degree~$1$ over~$w$,
by condition~(2) of Conjecture~\ref{conj:mainstrong}.
Therefore the fiber of $\sE_i\to \mtilde_i$ above~$w$ contains a rational point.
We fix any such rational point, let~$\xi_{w,i}$ denote its image in~$\sY^0_i$
and
use Hensel's lemma to lift~$\xi_{w,i}$ to a $k_v$\nobreakdash-point~$x'_v$ of~$X_c$.
For each $i\in\{n+1,\dots,N\}$ and each $v \in \Omega_i \setminus \{v_i\}$,
we fix an arbitrary rational
point~$\xi_{w,i}$ of the fiber of $\sY_i^0 \to \mtilde_i$ above~$w$
and again lift it to a $k_v$\nobreakdash-point~$x'_v$ of~$X_c$.

We have now defined~$x'_v$ for all $v \in \Omega \setminus \{v_{n+1},\dots,v_N\}$.
Let us finally
construct~$x'_v$ at the remaining places.
For $i \in \{1,\dots,N\}$
and $v \in \Omega_i$, let $n_{w,i} = w(t_0-a_i)$.
For $i \in \{n+1,\dots,N\}$, let
$\sigma_i=\sum_{v \in \Omega_i \setminus \{v_i\}} n_{w,i} \mkern2mu\Frob_{\xi_{w,i}} \in G_i$,
where $\Frob_{\xi_{w,i}}$ denotes the Frobenius automorphism of
the irreducible abelian \'etale cover $\sE_i\to\sY_i^0$
at~$\xi_{w,i}$.

\begin{lem}
\label{lem:isinhi}
For each $i \in \{n+1,\dots,N\}$,
we have $\sigma_i \in H_i$.
\end{lem}

\begin{proof}
Let us fix $i \in \{n+1,\dots,N\}$.
Let~$\gamma$ denote an element of~$\Gamma_i$.
As~$c$ is arbitrarily close to~$f(x''_v)$ for $v \in S\cap\Omega_f$,
as~$c$ and~$f(x''_v)$ belong to the same connected component of~$U(k_v)$
for $v \in \Omega_\infty$,
as $\sum_{v \in \Omega} \inv_v \gamma(f(x''_v))=0$
and as $\gamma(f(x''_v))=0$ for $v \in \Omega\setminus S$,
we have $\sum_{v\in S} \inv_v \gamma(c)=0$.
As $\gamma \in \Br(\sU)$ and $\Br(\sOint_v)=0$ for any finite place~$v$, we have $\inv_v \gamma(c)=0$
for $v \in \Omega \setminus (S \cup \Omega_1 \cup\dots\cup\Omega_N)$.
For $j \in \{1,\dots,N\}$ and $v \in \Omega_j$,
we have
\begin{align}
\inv_v \gamma(c)=n_{w\mkern-1mu,\mkern1muj} \mkern2mu\partial_{\gamma,m_j}(\Frob_w)\rlap{\text{,}}
\end{align}
where $\partial_{\gamma,m_j} \in \Ker\big(H^1(k_j,\Q/\Z)\to H^1(K'_j,\Q/\Z)\big)=\Hom(\Gal(K_j'/k_j),\Q/\Z)$
is the residue of~$\gamma$ at~$m_j$
and $\Frob_w \in \Gal(K_j'/k_j)$ denotes the Frobenius at~$w$,
since $\gamma \in \Br(\sU)$
(see~\cite[Corollaire~2.4.3]{harariduke}).
By the global reciprocity law,
we deduce from these remarks that
\begin{align}
\sum_{j=1}^N \sum_{v \in \Omega_j}
n_{w\mkern-1mu,\mkern1muj} \mkern2mu\partial_{\gamma,m_j}(\Frob_w)=0\rlap{\text{.}}
\end{align}
We have $\Frob_w=0$ when $j\leq n$ and when~$v=v_j$ since~$L_j$ possesses a place of degree~$1$
over~$w$ in these two cases.  Moreover, by the definition of~$\Gamma_i$, we have $\partial_{\gamma,m_j}=0$ if $j>n$ and $j\neq i$.  All in all, we conclude that
\begin{align}
\sum_{v \in \Omega_i \setminus \{v_i\}} n_{w,i} \mkern2mu\partial_{\gamma,m_i}(\Frob_w)=0\rlap{\text{,}}
\end{align}
or in other words
$\partial_{\gamma,m_i}(\barsigma_i)=0$
if~$\barsigma_i$ denotes the image of~$\sigma_i$ in $G_i\mkern.2mu/\mkern-.2muH_i=\Gal(K_i/k_i)$.
As~$\partial_{\gamma,m_i}$ takes all possible values in $\Hom(\Gal(K_i/k_i),\Q/\Z)$ when~$\gamma$
ranges over~$\Gamma_i$, we conclude that $\barsigma_i=0$.
\end{proof}

By Lemma~\ref{lem:isinhi}
and our hypotheses on~$S$, we can choose, for each $i \in \{n+1,\dots,N\}$,
a rational point~$\xi_{w_i,i}$ of the fiber of $\sY_i^0\to \mtilde_i$ above~$w_i$
such that $\Frob_{\xi_{w_i,i}}=-\sigma_i$
and hence
\begin{align}
\label{eq:sumfrob}
\sum_{v \in \Omega_i} n_{w,i} \mkern2mu\Frob_{\xi_{w,i}}=0
\end{align}
in~$G_i$.
We then lift~$\xi_{w_i,i}$ to a $k_{v_i}$\nobreakdash-point~$x'_{v_i}$ of~$X_c$.

We have thus constructed an adelic point $(x'_v)_{v \in \Omega} \in X_c(\A_k)$
arbitrarily close to~$(x_v)_{v \in \Omega}$ in $X(\A_k)$ (resp.~in~$X(\A_k)_{\smash{\bullet}}$).
It remains to check that $(x'_v)_{v \in \Omega}$ is orthogonal to~$B$ with respect to the Brauer--Manin pairing.
By the orthogonality of~$(x''_v)_{v \in \Omega}$ to~$B$ and by the definition of~$S$,
we have $\sum_{v \in S} \inv_v \beta(x''_v)=0$ for all $\beta\in B$.
On the other hand,
as~$x'_v$ is arbitrarily close to~$x''_v$ for $v \in S \cap \Omega_f$ and as~$x'_v$ and~$x''_v$ belong to the same connected component of
$f^{-1}(U)(k_v)$ for $v \in \Omega_\infty$, we have $\beta(x''_v)=\beta(x'_v)$ for all $v \in S$ and all $\beta\in B$,
so that $\sum_{v \in S} \inv_v \beta(x'_v)=0$ for all $\beta\in B$.
For $v \in \Omega\setminus (S \cup \Omega_1 \cup \dots \cup \Omega_N)$
we have $\beta(x'_v)=0$ for all $\beta\in B$ since $B \subseteq \Br(f^{-1}(\sU))$.
For $v \in \Omega_1\cup \dots \cup \Omega_n$,
the existence of a rational point in the fiber of~$\sE_i \to \sY^0_i$ above~$\xi_{w,i}$
implies, by the same argument as in Lemma~\ref{lem:evalvanishes}, that~$\beta(x'_v)=0$ for all $\beta\in B$.  Thus
\begin{align}
\sum_{v \in\Omega} \inv_v \beta(x'_v) = \sum_{i=n+1}^N \sum_{v \in \Omega_i} \inv_v \beta(x'_v)
\end{align}
for all $\beta \in B$.
Now,
as $\beta \in \Br(f^{-1}(\sU))$,
we have $\inv_v\beta(x'_v)=n_{w,i} \mkern2mu\partial_{\beta,Y_i}(\Frob_{\xi_{w,i}})$
for any $\beta \in B$,
any $i\in \{n+1,\dots,N\}$
and any $v \in \Omega_i$,
if $\partial_{\beta,Y_i} \in \Hom(G_i,\Q/\Z)\subset H^1(k(Y_i),\Q/\Z)$
denotes the residue of~$\beta$ at the generic point of~$Y_i$
(see~\cite[Corollaire~2.4.3]{harariduke}).
In view of~\eqref{eq:sumfrob}, 
we conclude that $\sum_{v \in \Omega}\inv_v \beta(x'_v)=0$ for all $\beta \in B$.
\end{proof}

\begin{rmk}
The rather delicate arguments applied in the proof of Theorem~\ref{th:ratpointsmain}
to deal with the split fibers~$f^{-1}(m)$ for $m \in M''$
are those used by Harari in~\cite{harariduke}.
As explained above, these arguments can be avoided entirely
by taking~$M''=\emptyset$, at the expense of losing some control over the polynomials
to which Conjectures~\ref{conj:mainstrong} and~\ref{conj:mainweak}
are applied in the course of the proof.
\end{rmk}

Using arguments similar to those of~\textsection\ref{sec:hilbertsubsets},
we now incorporate a Hilbert set into the conclusion
of Theorem~\ref{th:ratpointsmain}.

\begin{thm}
\label{th:ratpointsmainhilb}
Let us keep the notation and assumptions of Theorem~\ref{th:ratpointsmain}.
For any Hilbert subset $H \subseteq \P^1_k$,
the conclusion of Theorem~\ref{th:ratpointsmain} still holds
if the rational point $c \in U(k)$ is required,
in addition, to belong to~$H$.
\end{thm}

\begin{proof}
By Lemma~\ref{lem:inclusionfiniteness}, there exists
a finite subgroup $B_0 \subset \Br(X)$ such that
\begin{align*}
(B+f_\eta^*\mkern.5mu\Br(\eta))\cap\Br(X) = B_0 + f^*\Br(k)\rlap{\text{.}}
\end{align*}
Let us fix a rational point $h \in H \cap U$.  Such a point exists
by Hilbert's irreducibility theorem.
Let $(x_v)_{v \in \Omega} \in X(\A_k)$
be orthogonal to~$B_0$.
Let $S \subset \Omega$ be a finite subset.  We must prove the existence
of $c \in U(k)$ belonging to~$H$
and of $(x'_v)_{v \in \Omega} \in X_c(\A_k)$ orthogonal to~$B$
such that~$x'_v$ is
arbitrarily close to~$x_v$ for $v \in S$ (resp.~arbitrarily close
to~$x_v$ for $v \in S\cap\Omega_f$ and in the same connected component as~$x_v$ for $v \in \Omega_\infty$).
To this end, we may assume, after enlarging~$S$,
that $X_h(k_v)\neq\emptyset$ for any $v \in \Omega\setminus S$
(see~\cite{langweil}, \cite{nisnevic}),
that $\Omega_\infty\subseteq S$
and that any element of~$B_0$
evaluates trivially on~$X(k_v)$
for any $v \in\Omega\setminus S$.
For $v \in \Omega\setminus S$,
let~$x''_v$ be an arbitrary $k_v$\nobreakdash-point of~$X_h$.
Let $x''_v=x_v$ for $v \in S$.
By the definition of~$B_0$,
the adelic point $(x''_v)_{v \in \Omega}$ is orthogonal to $(B+f_\eta^*\mkern.5mu\Br(\eta))\cap\Br(X)$
for the Brauer--Manin pairing.
We may therefore apply Theorem~\ref{th:ratpointsmain} to it.
The resulting point $c \in U(k)$ will then be arbitrarily close to~$h$ in~$\prod_{v \in \Omega\setminus S}\P^1(k_v)$,
which implies, by~\cite[Proposition~6.1]{smeets}, that $c\in H$.
\end{proof}

In the remainder of~\textsection\ref{subsec:consequencesrationalpoints},
we spell out the most significant conditional corollaries of
Theorem~\ref{th:ratpointsmainhilb}.

\begin{cor}
\label{cor:ratpointsfull}
Let~$X$ be a smooth, proper, irreducible variety over a number field~$k$, endowed with a morphism
$f:X \to \P^1_k$ whose geometric generic fiber~$X_\etabar$ is irreducible.
Assume that
\begin{enumerate}
\item $H^1(X_{\bar\eta},\Q/\Z)=0$ and $H^2(X_\etabar,\sO_{X_\etabar})=0$;
\item every fiber of~$f$ contains
an irreducible component of multiplicity~$1$;
\item Conjecture~\ref{conj:mainstrong} (resp.~Conjecture~\ref{conj:mainweak}) holds;
\item there exists a Hilbert subset $H \subseteq \P^1_k$ such that
$X_c(k)$ is dense in $X_c(\A_k)^{\Br(X_c)}$
(resp.~in~$X_c(\A_k)^{\Br(X_c)}_{\smash{\bullet}}$)
for every rational point~$c$ of~$H$.
\end{enumerate}
Then~$X(k)$ is dense in $X(\A_k)^{\Br(X)}$
(resp. in~$X(\A_k)^{\Br(X)}_{\smash{\bullet}}$).
\end{cor}

\begin{proof}
After a change of coordinates on~$\P^1_k$, we may assume
that $f^{-1}(\infty)$ is split.
Let~$U$ be a dense open subset of~$\P^1_k$ over which the fibers of~$f$ are
split, with $\infty \in U$.
After shrinking~$U$, we may assume that
$U\neq \P^1_k$
and that there exists
a finite
subgroup $B \subset \Br(f^{-1}(U))$
such that $B+f_\eta^*\mkern.5mu\Br(\eta)=\Br(X_\eta)$.
Indeed, it follows from~(1) that $\Br(X_\eta)/f_\eta^*\mkern.5mu\Br(\eta)$ is finite
(see Lemma~\ref{lem:implications}).
By Proposition~\ref{prop:specialisation},
there exists a Hilbert subset $H' \subseteq U$ such that
the natural map $B \to \Br(X_c)/f_c^*\mkern.5mu\Br(k)$ is surjective
for all $c \in H'$.
Applying Theorem~\ref{th:ratpointsmainhilb} to the Hilbert subset $H\cap H'$,
to $(x_v)_{v \in \Omega} \in X(\A_k)^{\Br(X)}$ and to any subset $M' \subseteq \P^1_k \setminus U$
satisfying the hypothesis of Theorem~\ref{th:ratpointsmain} (for instance $M'=\P^1_k\setminus U$)
produces a rational point~$c$ of~$H\cap H'$ and an element of $X_c(\A_k)$
orthogonal to~$B$, and therefore to~$\Br(X_c)$, which is arbitrarily
close to~$(x_v)_{v\in\Omega}$ in~$X(\A_k)$ (resp.~in~$X(\A_k)_{\smash{\bullet}}$).
In view of~(4), we may approximate in the fiber~$X_c$ to conclude the proof.
\end{proof}

\begin{cor}
\label{cor:ratpointsfullbis}
Let~$X$ be a smooth, proper, irreducible variety over a number field~$k$, endowed with a morphism
$f:X \to \P^1_k$ with geometrically irreducible generic fiber,
such that the hypotheses~(1), (2) and~(4) of Corollary~\ref{cor:ratpointsfull} are satisfied
and~$f^{-1}(\infty)$ is split.
Let~$M' \subset \A^1_k$ be a finite closed subset
containing the points with non-split fiber.
Assume~$M'$ contains a rational point or~$k$ is totally imaginary.
If Conjecture~\ref{conj:mainstrong}
holds for the irreducible monic polynomials which vanish at the points of~$M'$,
then~$X(k)$ is dense in $X(\A_k)^{\Br(X)}$.
\end{cor}

\begin{proof}
Same proof as Corollary~\ref{cor:ratpointsfull},
in view of Remark~\ref{rk:useofconj}~(ii).
\end{proof}

\begin{cor}
\label{cor:ratpointsRC}
Let $n\geq 1$. Let~$X$ be a smooth, proper, irreducible variety over a number field~$k$
and let
$f:X \to \P^n_k$
be a dominant morphism with rationally connected geometric generic fiber.
Assume Conjecture~\ref{conj:mainstrong} (resp.~Conjecture~\ref{conj:mainweak}) holds.
If there exists a Hilbert subset $H \subseteq \P^n_k$ such that
$X_c(k)$ is dense in $X_c(\A_k)^{\Br(X_c)}$
(resp.~in~$X_c(\A_k)^{\Br(X_c)}_{\smash{\bullet}}$)
for every rational point~$c$ of~$H$,
then~$X(k)$ is dense in $X(\A_k)^{\Br(X)}$
(resp. in~$X(\A_k)^{\Br(X)}_{\smash{\bullet}}$).
\end{cor}

\begin{proof}
We argue by induction on~$n$.
If $n=1$, Corollary~\ref{cor:ratpointsfull} applies:
its assumption~(1) is satisfied
by Lemma~\ref{lem:implications} and by~\cite[Corollary~4.18(b)]{debarrehigherdim}
while~(2) follows from the Graber--Harris--Starr theorem~\cite{ghs}.
Let us now assume $n>1$.
As~$\P^n_k$ is birationally equivalent to~$\P^{n-1}_k\times\P^1_k$,
we can find a smooth, proper, irreducible variety~$X'$ over~$k$, a morphism
$f':X' \to \P^{n-1}_k \times \P^1_k$
and a Hilbert subset $H' \subseteq \P^{n-1}_k \times \P^1_k$ such that
the generic fibers of~$f$ and of~$f'$ are isomorphic and such that
$X'_c(k)$ is dense in $X'_c(\A_k)^{\Br(X'_c)}$
(resp.~in~$X'_c(\A_k)^{\Br(X'_c)}_{\smash{\bullet}}$)
for every rational point~$c$ of~$H'$.
Let $g:X' \to \P^1_k$ denote the composition of~$f'$
with the second projection.
Let $H_1\subseteq\P^1_k$ be a Hilbert subset satisfying the conclusion of Lemma~\ref{lem:hilbertproduct}
with respect to the Hilbert subset~$H'$.
After replacing~$H_1$ with $H_1\cap U$ for a small enough dense open subset~$U$ of~$\P^1_k$,
we may assume that for every rational point~$h$ of~$H_1$,
the variety~$X'_h=g^{-1}(h)$ is smooth and irreducible and
the morphism $f'_h:X'_h \to \P^{n-1}_k$ is dominant, with rationally connected geometric generic fiber
(see \cite[Ch.~IV, Theorem~3.5.3]{kollarbook}).
By the definition of~$H_1$ and the hypothesis on~$H'$,
the morphism $f'_h:X'_h \to \P^{n-1}_k$ then satisfies all of the assumptions
of Corollary~\ref{cor:ratpointsRC}
if~$h$ is a rational point of~$H_1$.
Thus, by our induction hypothesis, the set~$X'_h(k)$
is dense in $X'_h(\A_k)^{\Br(X'_h)}$
(resp.~in~$X'_h(\A_k)_{\smash{\bullet}}^{\Br(X'_h)}$) for every rational point~$h$ of~$H_1$.
Now,
according to \cite[Corollary~1.3]{ghs},
the geometric generic fiber of$~g$
is a rationally connected variety
since it dominates~$\P^{n-1}$ with rationally connected
geometric generic fiber.
We may therefore apply the case $n=1$ of Corollary~\ref{cor:ratpointsRC}
to~$g$ and conclude that~$X'(k)$ is dense in $X'(\A_k)^{\Br(X')}$
(resp.~in~$X'(\A_k)_{\smash{\bullet}}^{\Br(X')}$).
The desired result finally follows as~$X$ and~$X'$ are birationally equivalent
and $\Coker(\Br(k)\to\Br(X))$ is finite
(see~\cite[Proposition~6.1~(iii)]{cps}; for the finiteness assertion,
see~\cite[Lemma~1.3~(i)]{ctskogoodreduction}, \cite[Corollary~4.18~(b)]{debarrehigherdim},
\cite[Corollary~1.3]{ghs}).
\end{proof}

\begin{rmk}
In the situation of Corollary~\ref{cor:ratpointsfull}, let us assume
that the generic fiber of~$f$ is birationally equivalent to a torsor
under a torus~$T$ over the function field of~$\P^1_k$.  In this case, assumptions~(1), (2) and~(4)
are satisfied (for~(4), see~\cite{sansuclinear}).
Thus~$X(k)$ is dense in~$X(\A_k)^{\Br(X)}$ as soon as the varieties~$W$ associated to~$f$
in~\textsection\ref{subsubsec:geomcrit} satisfy strong approximation off any finite place,
according to Corollary~\ref{cor:strongapproxconj}.
This should be compared with \cite[Theorem~1.1]{skodescenttoric}, which asserts that when the torus~$T$ is
defined over~$k$, the same conclusion can be reached under the sole assumption that the varieties~$W$ satisfy weak approximation.
\end{rmk}

Combining
Theorem~\ref{th:ratpointsmain} (with $M''=\emptyset$ and $B=0$),
Theorem~\ref{th:schinzelconjmain}
and Remark~\ref{rk:useofconj}~(iii) also yields the following conditional corollary,
which extends the main results of~\cite[\textsection1]{ctsksd98}
and recovers \cite[Theorem~3.5]{weioneq}.

\begin{cor}
\label{cor:schinzel}
Let~$X$ be a smooth, proper, irreducible variety over a number field~$k$,
endowed with a morphism $f:X \to \P^1_k$ with geometrically irreducible generic fiber.
We make the following assumptions:
\begin{itemize}
\item Schinzel's hypothesis~$(\mathrm{H})$ holds;
\item for every $m \in \P^1_k$, the fiber $f^{-1}(m)$ possesses an irreducible component of multiplicity~$1$
in the function field of which the algebraic closure of~$k(m)$ is an almost abelian extension of~$k(m)$
in the sense of Definition~\ref{def:almostabelian};
\item the fibers of~$f$ above the rational points of a Hilbert subset of~$\P^1_k$
satisfy weak approximation.
\end{itemize}
Then $X(k)$ is dense in $X(\A_k)^{\Br_\vert(X/\P^1_k)}$,
where $\Br_\vert(X/\P^1_k)=\Br(X)\cap f_\eta^*\mkern.5mu\Br(\eta)$.
\end{cor}

\subsection{Some unconditional results}
\label{subsec:someunconditional}

Combining Corollary~\ref{cor:ratpointsfullbis} with
Matthiesen's Theorem~\ref{thm:matthiesen} yields the following unconditional result.

\begin{thm}
\label{th:consequenceofmatthiesen}
Let~$X$ be a smooth, proper, irreducible variety over~$\Q$
and $f:X \to \P^1_\Q$
be a dominant morphism,
with rationally connected geometric generic fiber,
whose non-split fibers all lie over rational points of~$\P^1_\Q$.
Assume that
$X_c(\Q)$ is dense in $X_c(\A_\Q)^{\Br(X_c)}$
for every rational point~$c$ of a Hilbert subset of~$\P^1_\Q$.
Then~$X(\Q)$ is dense in~$X(\A_\Q)^{\Br(X)}$.
\end{thm}

Theorem~\ref{th:consequenceofmatthiesen} was previously
known when~$f$ has a unique non-split fiber
(see~\cite{hararifleches})
or when
every singular fiber of~$f$ contains an irreducible component of multiplicity~$1$ split by an abelian extension of~$\Q$ and
the smooth fibers of~$f$ above the rational points of~$\P^1_\Q$ satisfy weak approximation (see~\cite{hsw}).

Thanks to Borovoi's theorem \cite[Corollary~2.5]{borovoi},
Theorem~\ref{th:consequenceofmatthiesen}
has the following corollary, whose statement should
be compared with \cite[Conjecture~1]{cttiruchirapalli}.

\begin{cor}
\label{cor:pencilhommatthiesen}
Let~$X$ be a smooth, proper, irreducible variety over~$\Q$ endowed with a morphism $f:X\to \P^1_\Q$ whose generic fiber is birationally equivalent
to a homogeneous space of a connected linear algebraic group, with connected geometric stabilisers.
If the non-split fibers of~$f$ lie over rational points of~$\P^1_\Q$, then~$X(\Q)$ is dense in $X(\A_\Q)^{\Br(X)}$.
\end{cor}

\begin{rmk}
\label{rk:corhommatthiesenpn}
By the same induction argument as in the proof of Corollary~\ref{cor:ratpointsRC}, it
follows from Theorem~\ref{th:consequenceofmatthiesen}
that the statements of
Theorem~\ref{th:consequenceofmatthiesen}
and of Corollary~\ref{cor:pencilhommatthiesen}
remain valid when~$\P^1_\Q$ is replaced with~$\P^n_\Q$ for some~$n\geq 1$, if the non-split codimension~$1$ fibers of~$f$ all lie over
generic points of \emph{hyperplanes} of~$\P^n_\Q$ defined over~$\Q$.
\end{rmk}

Corollary~\ref{cor:pencilhommatthiesen} applies to the total space of an arbitrary pencil of toric varieties, as long as the non-split members of the pencil
are defined over~$\Q$.  Thus,
Corollary~\ref{cor:pencilhommatthiesen}, together with Remark~\ref{rk:corhommatthiesenpn},
subsumes all of the results of~\cite{browningmatthiesen},
\cite{bms}, \cite[\textsection4]{hsw},
as well as \cite[\textsection3.2]{smeets} and \cite[Corollary~1.2, Corollary~1.3]{skodescenttoric}.
One should note, however, that
Theorem~\ref{thm:matthiesen}
builds on the contents of~\cite{browningmatthiesen},
\emph{via}~\cite{matthiesen}.

Other unconditional results can be obtained by combining Corollary~\ref{cor:ratpointsfullbis} with
Theorem~\ref{th:algebraiccases}.
In the following statement, the \emph{rank} of~$f$
is the sum of the degrees of the closed points of~$\P^1_k$ above which
the fiber of~$f$ is not split.

\begin{thm}
\label{th:smallrank}
Let~$X$ be a smooth, proper, irreducible variety over
 a number field~$k$
and $f:X \to \P^1_k$
be a dominant morphism
with rationally connected geometric generic fiber.
Assume\footnote{Added in proof: the hypothesis that either~$k$ is totally imaginary
or the non-split fibers of~$f$ lie over rational points of~$\P^1_k$ is now
known to be superfluous.
Details will be given elsewhere.} that
 $\mathrm{rank}(f)\leq 2$
and that either~$k$ is totally imaginary
or the non-split fibers of~$f$ lie over rational points of~$\P^1_k$.
If~$X_c(k)$ is dense in $X_c(\A_k)^{\Br(X_c)}$ for every
rational point~$c$ of a Hilbert subset of~$\P^1_k$, then~$X(k)$ is dense in~$X(\A_k)^{\Br(X)}$.
\end{thm}

Theorem~\ref{th:smallrank} is due to Harari~\cite{harariduke} \cite{hararifleches} when $\mathrm{rank}(f)\leq 1$.
Fibrations over~$\P^1_k$ with rank~$2$ had been dealt with (using the descent method)
in~\cite[Theorem~A]{ctskodescent} under the assumption that the fibers above a Hilbert set of rational points satisfy weak approximation.
Theorem~\ref{th:smallrank} relaxes this assumption when~$k$ is totally imaginary.
When~$k$ is not totally imaginary, one can recover Theorem~A and Theorem~B of \emph{op.\ cit.}\ by combining Theorem~\ref{th:ratpointsmainhilb} (with $M''=\emptyset$)
and Theorem~\ref{th:algebraiccases}.
Thus, Theorem~\ref{th:smallrank} answers the question raised at the end of~\cite[\textsection2.2]{cttiruchirapalli}.
The results of~\cite{ctskodescent} were
also extended in~\cite[Theorem~2.9]{harskotorsors}
to cover the case of
Theorem~\ref{th:smallrank}
in which the non-split fibers of~$f$ lie over rational points of~$\P^1_k$
and are split by prime degree extensions of~$k$.

Finally, we note that the statement obtained by combining Theorem~\ref{th:irving} (a corollary
of Irving's arguments~\cite{irving})
with Theorem~\ref{th:ratpointsmain} (with $M''=\emptyset$)
recovers the main result of~\cite{irving} on the existence of solutions to the
equation~\eqref{eq:normhyp} for cubic polynomials~$P(t)$, and extends it to
general fibrations with the same degeneration data,
under the assumption that the cubic field~$K$ possesses a unique real place.

\bibliographystyle{amsalpha}
\bibliography{zcfib}
\end{document}